\documentclass[11pt]{article}
\usepackage{graphicx} % Required for inserting images
\usepackage{subfig} 
\usepackage{bbm}
\usepackage{bm}
\usepackage[margin=1in]{geometry}
\usepackage{color}
\usepackage{amsfonts}
\usepackage{algorithm}
\usepackage{algorithmic}
\usepackage{latexsym, amssymb, amsmath, amscd, amsthm, amsxtra}
\usepackage{mathtools}
\usepackage{enumerate}
\usepackage[all]{xy}
\usepackage{mathrsfs}
\usepackage{fancyhdr}
\usepackage{listings}
\usepackage{hyperref}
\usepackage{enumitem}
\usepackage{multirow}
\usepackage{tikz}

\def\P{\mathbb{P}}

\def\E{\mathbb{E}}

\def\11{\mathbbm{1}}

\def\ER{Erd\H{o}s-R\'enyi\ }

\def\Fop{\operatorname{F}}

\newcommand{\Pb}{\mathbb P}
\newcommand{\Qb}{\mathbb Q}

\newtheorem{thm}{Theorem}[section]

\newtheorem{proposition}[thm]{Proposition}

\newtheorem{lemma}[thm]{Lemma}

\newtheorem{remark}[thm]{Remark}
\newtheorem{DEF}[thm]{Definition}

\numberwithin{equation}{section}

\definecolor{myGreen}{RGB}{0,255,0}      % #00FF00
\definecolor{myPurple}{RGB}{160,32,240}  % #A020F0
\definecolor{myRed}{RGB}{255,0,0}        % #FF0000
\definecolor{myBlue}{RGB}{0,0,255}       % #0000FF

\makeatletter
\newenvironment{breakablealgorithm}
{% \begin{breakablealgorithm}
		\begin{center}
			\refstepcounter{algorithm}% New algorithm
			\hrule height.8pt depth0pt \kern2pt% \@fs@pre for \@fs@ruled
			\renewcommand{\caption}[2][\relax]{% Make a new \caption
				{\raggedright\textbf{\ALG@name~\thealgorithm} ##2\par}%
				\ifx\relax##1\relax % #1 is \relax
				\addcontentsline{loa}{algorithm}{\protect\numberline{\thealgorithm}##2}%
				\else % #1 is not \relax
				\addcontentsline{loa}{algorithm}{\protect\numberline{\thealgorithm}##1}%
				\fi
				\kern2pt\hrule\kern2pt
			}
		}{% \end{breakablealgorithm}
		\kern2pt\hrule\relax% \@fs@post for \@fs@ruled
	\end{center}
}
\makeatother

\hypersetup{colorlinks=true,linkcolor=blue}

\title{Fundamental Limits of Community Detection in Contextual Multi-Layer Stochastic Block Models}

\usepackage{authblk}
\author[1]{Shuyang Gong\thanks{Email: \textit{gongsyprob@gmail.com}.}}
\author[2]{Dong Huang\thanks{Email: \textit{hd23@mails.tsinghua.edu.cn}.}}
\author[1]{Zhangsong Li\thanks{Email: \textit{ramblerlzs@pku.edu.cn}.}}

\affil[1]{School of Mathematical Sciences, Peking University}
\affil[2]{Department of Statistics and Data Science, Tsinghua University}

\date{\today}
\begin{document}
\maketitle

\begin{abstract}
    We consider the problem of community detection from the joint observation of a high-dimensional covariate matrix and $L$ sparse networks, all encoding noisy, partial information about the latent community labels of $n$ subjects. In the asymptotic regime where the networks have constant average degree and the number of features $p$ grows proportionally with $n$, we derive a sharp threshold under which detecting and estimating the subject labels is possible. Our results extend the work of \cite{MN23} to the constant-degree regime with noisy measurements, and also resolve a conjecture in \cite{YLS24+} when the number of networks is a constant. 
    
    Our information-theoretic lower bound is obtained via a novel comparison inequality between Bernoulli and Gaussian moments, as well as a statistical variant of the ``recovery to chi-square divergence reduction'' argument inspired by \cite{DHSS25}. On the algorithmic side, we design efficient algorithms based on counting decorated cycles and decorated paths and prove that they achieve the sharp threshold for both detection and weak recovery. In particular, our results show that there is no statistical-computational gap in this setting.
\end{abstract}

\tableofcontents

\section{Introduction}{\label{sec:intro}}

The detection of latent communities is a fundamental problem in modern network analysis. This problem has attracted widespread attention in probability, statistics, computer science, statistical physics and social sciences over the last three decades, leading to a detailed understanding of the basic statistical thresholds for signal recovery, and the introduction of statistically optimal algorithms for community detection. The stochastic block model (SBM) \cite{HLL83} is a popular model for studying community detection. There has been a large literature on theoretical approaches, algorithmic, and application aspects of SBM. We refer the interested reader to \cite{Abbe18} for a survey of the recent progress in this research endeavor.

Ever-growing techniques for data acquisition have led us to a new paradigm where one could have multiple data sets as multiple sources of information about the community structure. For instance, for a set of $n$ people, one could potentially have several social networks observed on them (Facebook, LinkedIn, etc.) together with a large collection of socioeconomic (genomic, neuroimaging, etc.) covariates for each individual. This raises the challenge of how to best integrate information from these multiple sources to uncover the common underlying community structure.

When one observes multiple networks without any covariate, this scenario has been studied in the multilayer network literature \cite{PC20, LZZ24, YLS24+}. The arguably more general and interesting scenario is where one observes one or more networks together with a collection of covariates. In the pioneering work \cite{DMMS18}, the authors considered the case where the available data is an adjacency matrix of an SBM and a high-dimensional Gaussian covariate matrix, both containing the same balanced two block community structure. Under this stylized yet informative model, they rigorously established a sharp information-theoretic threshold for detecting the community structure (i.e., to uncover the community structure better than random guessing) when the feature dimension $p$ and the network size $n$ tend to infinity proportionally and the average degree of the network diverges with $n$. In addition, they proposed a heuristic algorithm which supports the information-theoretic threshold empirically. Subsequently, the sharp threshold was extended to the case where the average degree is bounded \cite{LS23}. More recently, \cite{MN23} derived phase transitions and information-theoretic limits for contextual multi-layer SBMs with a common label structure, again under the assumption of diverging average degree.

The present paper is motivated by two key limitations left open by \cite{MN23}: (1) The \emph{sparse} setting where the average degree of the networks remains constant; (2) The \emph{noisy measurement} setting where the community labels of the networks are noisy versions of the labels indicated by the covariates. To this end, we introduce a general framework for community detection in contextual multi-layer networks. Our main contributions are two-fold:
\begin{itemize}
    \item We establish sharp thresholds for phase transition between the regime where detecting the community structure is feasible and the regime where no procedure performs better than random guessing, when the number of networks is a constant.
    \item We propose an efficient algorithm based on subgraph counts for signal recovery under our general framework.
\end{itemize}
Consequently, our results generalize those of \cite{MN23} to the practically critical sparse and noisy setting. Furthermore, our results extend straightforwardly to the setting without contextual information, thereby resolve a conjecture in \cite{YLS24+} when the number of networks is constant.

\subsection{Contextual multi-layer SBMs and main results}{\label{subsec:main-results}}

A general formulation of the contextual multi-layer stochastic block model is as follows. Specifically, denoting by $\operatorname{U}_n$ the set of unordered pairs $(i,j)$ where $1\le i<j\le n$, we can define this model as follows. 

\begin{DEF}[Stochastic block model]{\label{def-SBM}}
    Given an integer $n\ge 1$ and two parameters $\lambda>0, \epsilon \in (0,1)$, we define a random graph $\bm G$ on $[n]=\{ 1,\ldots,n \}$ as follows. First, we select a labeling $\bm x \in \{ -1,+1 \}^{n}$ uniformly at random. For each distinct pair $(i,j) \in \operatorname{U}_n$, we independently add an edge $(i,j)$ with probability $\frac{(1+\epsilon)\lambda}{n}$ if $\bm x(i)=\bm x(j)$, and with probability $\frac{(1-\epsilon)\lambda}{n}$ if $\bm x(i)\neq \bm x(j)$. We denote $\mathcal{S}(n,\lambda,\epsilon)$ as the law of $\bm G$.
\end{DEF}

\begin{DEF}[Contextual multi-layer SBM]{\label{def-contxtual-multilayer-SBM}}
    Given integers $p,n,L \in \mathbb N$, parameters $\mu,\rho>0$ and two families of parameters $\{ \lambda_{\ell}: 1 \leq \ell \leq L \}, \{ \epsilon_\ell \in (0,1): 1 \leq \ell \leq L \}$, we generate a matrix $\bm Y\in \mathbb R^{n*p}$ and $L$ random graphs $\bm G_1,\ldots,\bm G_L$ as follows. First, sample a labeling $\bm x \in \{ -1,+1 \}^n$ uniformly at random. Next, sample i.i.d.\ random vectors $\bm z_1,\ldots,\bm z_L \in \{ -1,+1 \}^n$ such that 
    \begin{equation}{\label{eq-def-z-ell}}
        \Pb\big( \bm z_{\ell}(i)=1 \big) = 1-\Pb\big( \bm z_{\ell}(i)=-1 \big) = \frac{1+\rho}{2} \mbox{ independently for } 1 \leq \ell \leq L, 1 \leq i \leq n \,.
    \end{equation}
    Define (where $\odot$ denotes the Hadamard product)
    \begin{equation}{\label{eq-def-x-ell}}
        \bm x_{\ell} = \bm x \odot \bm z_{\ell} \mbox{ for } 1 \leq \ell \leq L, \mbox{ i.e., } \bm x_{\ell}(i) = \bm x(i) \bm z_{\ell}(i) \mbox{ for } 1 \leq i \leq n \,.
    \end{equation}
    Conditioned on $\bm x,\bm x_1,\ldots,\bm x_{L}$, our observations are independent:
    \begin{equation}{\label{eq-def-contextual-multilayer-SBM}}
        \bm Y = \sqrt{\frac{ \mu }{ n }} \bm x \bm u^{\top} + \bm Z, \quad \bm G_{\ell} \sim \mathcal S(n,\lambda_{\ell},\epsilon_{\ell}) \mbox{ with labeling } \bm x_{\ell}\,, 1 \leq \ell \leq L \,.
    \end{equation}
    Here $\bm u \in \mathcal N(0,\mathbb I_p)$ is a standard $p$-dimensional Gaussian vector, and $\bm Z \in \mathbb R^{n*p}$ has i.i.d.\ standard normal entries. Let $\mathcal S(n,p;\mu,\rho,\{ \lambda_{\ell} \},\{ \epsilon_{\ell} \})$ denote the law of $(\bm Y;\bm G_1,\ldots,\bm G_L)$. We focus on the regime where $\frac{p}{n}=\frac{1}{\gamma}$ for some constant $\gamma\in (0,\infty)$ and $L$ is fixed.
\end{DEF}
We denote by $\Pb=\Pb_{n,p,\mu,\rho,\{ \lambda_{\ell} \},\{ \epsilon_{\ell} \}}$ the law of $(\bm Y,\bm G_1,\ldots,\bm G_L)$ under Definition~\ref{def-contxtual-multilayer-SBM}. In addition, we denote by $\Qb=\Qb_{n,p,\{ \lambda_{\ell} \}}$ the law of $(\bm Y,\bm G_1,\ldots,\bm G_L)$ such that $\bm Y\in \mathbb R^{n*p}$ constitutes i.i.d.\ standard normal entries, and $\bm G_{\ell}$ are independent \ER graphs $\mathcal G(n,\frac{\lambda_\ell}{n})$ with average degree $\lambda_{\ell}$. One of the simplest statistical tasks is binary hypothesis testing. In our case this amounts to, given $(\bm Y,\bm G_1,\ldots,\bm G_L)$, determining whether it was sampled from $\Pb$ or $\Qb$. We refer to $\Pb$ as the planted distribution, and $\Qb$ as the null distribution.
\begin{DEF}{\label{def-strong-detection}}
    We say a test $\mathcal A=\mathcal A(\bm Y,\bm G_1,\ldots,\bm G_L) \in \{ 0,1 \}$ achieves \emph{strong detection} between $\Pb$ and $\Qb$, if
    \begin{align}{\label{eq-def-strong-detection}}
        \Pb\big( \mathcal A( \bm{Y} )=0 \big) + \Qb( \mathcal A(\bm Y)=1 ) \to 0 \mbox{ as } n \to \infty \,.
    \end{align}
\end{DEF}
It is well known that the minimal sum of the Type~I and Type~II errors for testing $\Pb$ versus $\Qb$ equals $1-\mathsf{TV}(\Pb,\Qb)$ (see, e.g., \cite[Theorem 13.1.1]{LR05}), where $\mathsf{TV}(\Pb,\Qb)=\tfrac12\int |d\Pb-d\Qb|$ is the total variation distance between $\Pb$ and $\Qb$. Consequently, the possibility and impossibility of strong detection correspond to $\mathsf{TV}(\Pb,\Qb)=1-o(1)$ and $\mathsf{TV}(\Pb,\Qb)=o(1)$, respectively. 
An additional goal is to recover the planted vector $\bm x$. Note that this model cannot distinguish $\bm x$ with $-\bm x$. Thus, our goal is to instead estimate the rank-one deformation $\bm x \bm x^{\top}$. For $\mathcal X\in\mathbb R^{n*n}$ and $\bm x\in\mathbb R^n$, define $\langle \mathcal X,\bm x\bm x^\top\rangle \;:=\; \mathsf{tr}\!(\mathcal X^\top \bm x\bm x^\top).$
\begin{DEF}{\label{def-weak-recovery}}
    We say an estimator $\mathcal X:=\mathcal X(\bm Y,\bm G_1,\ldots,\bm G_L) \in \mathbb R^{n*n}$ achieves \emph{weak recovery}, if 
    \begin{align}{\label{eq-def-weak-recovery}}
        \mathbb E_{\Pb}\Bigg[ \frac{ \langle \mathcal X, \bm x \bm x^{\top} \rangle }{ \| \mathcal X \|_{\Fop} \| \bm x \bm x^{\top} \|_{\Fop} } \Bigg] \geq c \mbox{ for some constant } c>0 \,.
    \end{align}
\end{DEF}
Define
\begin{equation}{\label{eq-def-F}}
    F(\mu,\rho,\gamma,\{ \lambda_{\ell} \},\{ \epsilon_{\ell} \})=\max\left\{ \frac{ \mu^2 }{ \gamma }, \max_{1 \leq \ell \leq L} \big\{ \epsilon_{\ell}^2 \lambda_{\ell} \big\}, \frac{ \mu^2 }{ \gamma } + \sum_{\ell=1}^{L} \frac{ \rho^4 \epsilon_{\ell}^2 \lambda_{\ell} }{ 1-(1-\rho^4) \epsilon_{\ell}^2 \lambda_{\ell} } \right\} \,.
\end{equation}
Our main result can be summarized as follows. 
\begin{thm}{\label{MAIN-THM-informal}}
    Suppose that $L=O(1)$ and $F(\mu,\rho,\gamma,\{ \lambda_{\ell} \},\{ \epsilon_{\ell} \})<1$. Then strong detection and weak recovery are information-theoretically impossible.

    On the contrary, suppose that $L=O(1)$ and $F(\mu,\rho,\gamma,\{ \lambda_{\ell} \},\{ \epsilon_{\ell} \})>1$. Then there exists a test $\mathcal A$ (respectively, an estimator $\mathcal X$) that achieves strong detection (respectively, weak recovery). In addition, the test $\mathcal A$ and the estimator $\mathcal X$ can be computed in polynomial time. As a result, there is no statistical–computational gap for either detection or recovery in contextual multi-layer SBMs.
\end{thm}

\begin{remark}
    Note that in the definition of $F(\mu,\rho,\gamma,\{ \lambda_{\ell} \},\{ \epsilon_{\ell} \})$, the term $\frac{\mu^2}{\gamma}=1$ corresponds to the Ben Arous-Baik-P\'ech\'e (BBP) threshold for the spiked matrix $\bm Y$ \cite{BBP05}, and the term $\epsilon^2_\ell \lambda_\ell=1$ corresponds to the Kesten-Stigum (KS) threshold for the block model $\bm G_\ell$ \cite{KS66, DKMZ11, MNS15, MNS18, MSS25a}. In contrast, the last term $\frac{ \mu^2 }{ \gamma } + \sum_{\ell=1}^{L} \frac{ \rho^4 \epsilon_{\ell}^2 \lambda_{\ell} }{ 1-(1-\rho^4) \epsilon_{\ell}^2 \lambda_{\ell} }$ only emerges in the contextual multi-layer setting and reflects the interplay of low-rank matrix recovery and community detection. 
    %Also, it is straightforward to check that as long as the correlation $\rho\neq 0$ and $1-(1-\rho^4)\epsilon_\ell^2 \lambda_\ell>0$ for all $1\le \ell \le L$, we have 
    %\begin{align*}
    %    \max\left\{ \max_{1 \leq \ell \leq L} \big\{ \epsilon_{\ell}^2 \lambda_{\ell} \big\}, \frac{\mu^2}{\gamma} \right\} \geq 1 \implies \frac{ \mu^2 }{ \gamma } + \sum_{\ell=1}^{L} \frac{ \rho^4 \epsilon_{\ell}^2 \lambda_{\ell} }{ 1-(1-\rho^4) \epsilon_{\ell}^2 \lambda_{\ell} } \geq 1 \,.
    %\end{align*}
    Thus, our results show that an algorithm can take advantage of the correlation between contextual information and the labelings of different graphs to detect and estimate the signals even in certain regimes where efficiently recovering $\bm x$ from $\bm{Y}$ alone or recovering $\bm x_\ell$ from $\bm{G}_{\ell}$ alone is information-theoretically impossible. 
\end{remark}

\begin{remark}
    Under the same setting as in \cite{MN23}, a direct implication of our result is that in this model there is \emph{no statistical-computational gap} when the number of networks is an arbitrary constant. Nevertheless, we would like to remark that our information-theoretic lower bound for recovery extends to the setting $L=o(\log n)$ (see Remark~\ref{remark-larger-L} for details). Our recovery algorithmic results also extend to the regime $L=o(\log n)$, although in this case our algorithm takes pseudo-polynomial running time $n^{O(\log L)}$. As a result, the recovery threshold remains to be characterized by $F(\mu,\rho,\gamma;\{ \lambda_{\ell} \},\{ \epsilon_{\ell} \})=1$. In contrast, we expect the detection problem to behave differently in the two regimes $L=O(1)$ and $L=\omega(1)$.
\end{remark}

\begin{remark}
    While preparing this manuscript, we became aware of the following connection between our results and the work \cite{YLS24+}. In the absence of contextual information, our model reduces to $L$ sparse SBMs with correlated labelings. In this special case, our established threshold simplifies precisely to
    \begin{align}
        \max\left\{ \max_{1 \leq \ell \leq L} \big\{ \epsilon_{\ell}^2 \lambda_{\ell} \big\}, \sum_{\ell=1}^{L} \frac{ \rho^4 \epsilon_{\ell}^2 \lambda_{\ell} }{ 1-(1-\rho^4) \epsilon_{\ell}^2 \lambda_{\ell} } \right\} = 1 \,.  \label{eq-threshold-multi-layer-SBM}
    \end{align}
    This aligns with the model and proposed threshold studied in \cite{YLS24+}. Their analysis, which employs techniques from statistical physics, rigorously establishes \eqref{eq-threshold-multi-layer-SBM} under the assumption that the average degrees $\lambda_{\ell}$ grows \emph{logarithmically} with respect to $n$. This regime permits a Gaussian approximation central to their method. Consequently, their approach does not extend to the (arguably more interesting) constant degree regime. The authors of \cite{YLS24+} conjectured that the same threshold \eqref{eq-threshold-multi-layer-SBM} remains valid for constant $\lambda_\ell$. Our results confirm this conjecture when the number of networks $L$ is an arbitrary constant.   
\end{remark}

\subsection{Discussions and other related works}{\label{subsec:related-works}}

{\bf Multi-layer SBMs.} The problem of community detection in multilayer networks has recently attracted considerable attention. Much of the existing literature, particularly early work in the area, models multilayer networks using the Stochastic Block Model (SBM) and assumes that community assignments remain constant across all layers (see, e.g., \cite{BC20+, LCL20, MN23}). This assumption, however, is often unrealistic in practice, as argued in \cite{VMGP16, CLM22}. In many applications, community assignments are not identical across layers but are instead expected to be correlated. To address this, \cite{CLM22} introduced a tractable model for such correlated communities and analyzed the performance of a two-step algorithm combining spectral clustering with a maximum a posteriori (MAP) refinement. 

Our model reduces to the same model as \cite{CLM22} when no contextual information is observed. However, while their theoretical analysis focuses on the high signal-to-noise-ratio (SNR) regime where exact recovery is possible, we focus on the more challenging low SNR setting and study the problem of weak recovery. Our model is closest in spirit to the one in \cite{MN23}. However, we go beyond the homogeneous community assignment critical in that work. Finally, the recent work \cite{LZZ24} studies testing for latent community structures in multilayer dense networks.

\

\noindent{\bf Subgraph counts.} Our algorithm is based on subgraph counting, a method widely used for network analysis in both theory \cite{MNS15, BDER16} and practice \cite{ADH+08, RPS+21}. For instance, in community detection, counting cycles of logarithmic length has been shown to achieve the optimal threshold for distinguishing a symmetric stochastic block model (SBM) from an Erdős–Rényi graph in the sparse regime \cite{MNS15, MSS25a}. Similarly, in the dense regime, counting signed cycles yields asymptotically optimal power \cite{Ban18, BM17}. Other approaches based on counting non-backtracking or self-avoiding walks \cite{Mas14, MNS18, BLM15, AS15, AS18} have also been developed to achieve sharp thresholds for related community recovery problems.

In the setting of contextual multi-layer SBMs, however, conventional subgraph counting (such as counting simple paths or cycles) is inherently suboptimal, as it fails to leverage correlations between the contextual information and the labelings across different graph layers. Our work is instead aligned with recent advances in multi-modal learning \cite{Li25+}, which emphasize the importance of counting \emph{decorated paths and cycles}. Concretely, we consider cycles whose edges are ``decorated'' according to whether they originate from the contextual matrix $\bm Y$ or one of the graph layers $\bm G_{\ell}$. This decoration scheme ensures that the number of admissible decorated cycles grows exponentially faster than the number of unlabeled cycles, allowing us to surpass the BBP threshold/KS threshold within a single spiked model or a single SBM.

While the concept is natural, its analysis poses significant technical challenges. Especially in sparse networks, the sparsity induces complex correlations among different decorated cycles, making it difficult to control their counts precisely. To overcome this, we introduce a delicate \emph{weighting scheme}, where each decorated cycle is assigned a weight based on its specific combinatorial structure (see \eqref{eq-def-f-mathcal-H} and \eqref{eq-def-Phi-i,j-mathcal-J} for details). This weighted approach is a key departure from ``unweighted'' methods, and we believe such weighting is indeed crucial for an efficient algorithm to achieve the sharp threshold in this model.

\

\noindent{\bf Recovery-to-detection reduction.} Although well-developed methods exist for proving recovery lower bounds in SBMs, such as connecting the recovery problem to broadcasting on trees \cite{MNS15, MSS25a} or deriving the limiting mutual information (see, e.g., \cite{DAM17, ZZ16, GMZZ17} for single SBM and \cite{MN23, YLS24+} for multi-layer SBM), these approaches are not straightforward to apply in our sparse and noisy setting. The challenge is twofold. First, the local weak limit in contextual multi-layer SBMs corresponds to a broadcasting problem on multiple trees with correlated labels and contextual information, which appears analytically challenging. Second, the computation of limiting mutual information heavily relies on a Gaussian approximation scheme that relates the SBM to a suitable spiked matrix, thereby restricting the method to regimes with diverging degrees.

Instead, we adopt a reduction-based approach, which leverages the fact that weak recovery is at least as ``hard'' as detection. Our method is inspired by and can be viewed as a statistical variant of \cite{DHSS25}; essentially, we argue that if a statistic achieves weak recovery, then it can also be used to construct a statistic for strong detection (in a suitable sense, see Lemma~\ref{lem-bounded-Adv} for details). This perspective allows us to derive recovery lower bounds by analyzing the more tractable detection problem, where direct calculation of chi-square divergence is feasible. We believe this approach is not only effective for the present problem but also constitutes a general and easily implementable methodology that may be applied to a broader class of problems.

\ 

\noindent{\bf Open problems.} Our work has focused on contextual multi-layer SBMs with two balanced communities. A natural extension is to settings with multiple or unbalanced communities. We expect that our algorithmic techniques remain applicable as long as the number of communities is fixed, although the resulting detection thresholds may no longer be as explicit as those presented here. For a larger number of communities, we anticipate further information–computation gaps to emerge, consistent with phenomena observed in simpler models \cite{AS18, MSS25a}. This extension warrants a dedicated investigation, which we leave for future work.

Another interesting direction is the design of more practical algorithms that do not rely on the color-coding scheme, since our current recovery algorithm runs in time polynomial with a high-degree exponent in $n$. For community recovery in a single block model, spectral methods based on the non-backtracking matrix are known to achieve statistical performance comparable to counting self-avoiding walks \cite{KMM+13, BLM15}. Extending such spectral approaches to the contextual and multi-layer setting would require substantial new insights, and we also leave this question open for future study.

\subsection{Notation and paper organization}{\label{subsec:notation}}

We record in this subsection some notation conventions. For two probability measures $\mathbb P$ and $\mathbb Q$, we denote the total variation distance between them as $\mathsf{TV}(\mathbb P,\mathbb Q)$. The chi-squared divergence from $\Pb$ to $\Qb$ is defined as $\chi^2(\Pb \| \Qb)= \mathbb E_{X\sim\Qb}[ (\frac{\mathrm{d}\Pb}{\mathrm{d}\Qb}(X))^2 ]$. For a matrix or a vector $M$, we will use $M^{\top}$ to denote its transpose. For a $k*k$ matrix $M=(m_{ij})_{k*k}$, let $\mathsf{det}(M)$ and $\mathsf{tr}(M)$ be the determinant and trace of $M$, respectively. Denote $M \succ 0$ if $M$ is positive definite and $M \succeq 0$ if $M$ is positive semidefinite. Furthermore, if $M$ has real eigenvalues, we denote $\varsigma_1(M) \geq \ldots \geq \varsigma_{k}(M)$ as the eigenvalues of $M$. For two $k*l$ matrices $M_1$ and $M_2$, we define their inner product to be
\begin{align*}
    \big\langle M_1,M_2 \big\rangle:=\sum_{i=1}^k \sum_{j=1}^l M_1(i,j)M_2(i,j) \,.
\end{align*}
In addition, for a $k*l$ matrix $M$, define
\begin{align*}
    \| M \|_{\operatorname{F}} = \langle M,M \rangle^{\frac{1}{2}} = \left( \sum_{i=1}^{k} \sum_{j=1}^{l} M_{i,j}^2 \right)^{\frac{1}{2}}, \
    \| M \|_{\operatorname{op}} = \sqrt{\varsigma_1(M M^{\top})}, 
\end{align*}
to be its Frobenius norm and operator norm respectively. We will use $\mathbb{I}_{k}$ to denote the $k*k$ identity matrix (and we drop the subscript if the dimension is clear from the context). We will use the following notation conventions on graphs.

{\em Labeled graphs}. Denote by $\mathsf K_n$ the complete graph with vertex set $[n]$. For any graph $H$, let $V(H)$ denote the vertex set of $H$ and let $E(H)$ denote the edge set of $H$. We say $H$ is a subgraph of $G$, denoted by $H\subset G$, if $V(H) \subset V(G)$ and $E(H) \subset E(G)$. For all $v \in V(H)$, define $\mathsf{deg}_H(v)=\#\{ e \in E(H): v \in e \}$ to be the degree of $v$ in $H$. We say $v$ is an isolated vertex of $H$, if $\mathsf{deg}_H(v)=0$. Denote $\mathsf I(H)$ as the set of isolated vertices of $H$. We say $v$ is a leaf of $H$, if $\mathsf{deg}_H(v)=1$. Denote $\mathsf{L}(H)$ as the set of leaves in $H$. For $H,S \subset \mathsf K_n$, denote by $H \cap S$ the graph with vertex set given by $V(H) \cap V(S)$ and edge set given by $E(H)\cap E(S)$, and denote by $S \cup H$ the graph with vertex set given by $V(H) \cup V(S)$ and edge set $E(H) \cup E(S)$.  

{\em Graph isomorphisms and unlabeled graphs.} Two graphs $H$ and $H'$ are isomorphic, denoted by $H\cong H'$, if there exists a bijection $\sigma:V(H) \to V(H')$ such that $(\sigma(u),\sigma(v)) \in E(H')$ if and only if $(u,v)\in E(H)$. Denote by $[H]$ the isomorphism class of $H$; it is customary to refer to these isomorphic classes as unlabeled graphs. Let $\mathsf{Aut}(H)$ be the number of automorphisms of $H$ (graph isomorphisms to itself). 

We use standard asymptotic notations: for two sequences $a_n$ and $b_n$ of positive numbers, we write $a_n = O(b_n)$, if $a_n<Cb_n$ for an absolute constant $C$ and for all $n$ (similarly we use the notation $O_h$ if the constant $C$ is not absolute but depends only on $h$); we write $a_n = \Omega(b_n)$, if $b_n = O(a_n)$; we write $a_n = \Theta(b_n)$, if $a_n =O(b_n)$ and $a_n = \Omega(b_n)$; we write $a_n = o(b_n)$ or $b_n = \omega(a_n)$, if $a_n/b_n \to 0$ as $n \to \infty$. For notational simplicity, we assume throughout the paper that $\mu,\gamma,\{ \lambda_{\ell} \},\{ \epsilon_{\ell} \}$ are fixed constants and suppress dependence on them, e.g., we simply write $O(1)$ instead of $O_{\mu,\gamma,\{ \lambda_{\ell} \},\{ \epsilon_{\ell} \}}(1)$. For two real numbers $a$ and $b$, we let $a \vee b = \max \{ a,b \}$ and $a \wedge b = \min \{ a,b \}$. For two sets $A$ and $B$, we define $A\sqcup B$ to be the disjoint union of $A$ and $B$ (so the notation $\sqcup$ only applies when $A, B$ are disjoint). The indicator function of sets $A$ is denoted by $\mathbf{1}_{A}$. In addition, we use both $|A|$ and $\#A$ to denote the cardinality of $A$.

The rest of the paper is organized as follows. In Section~\ref{sec:info-lower-bound}, we prove the information-theoretic lower bounds for detection and recovery stated in Theorem~\ref{MAIN-THM-informal}. In Section~\ref{sec:efficient-algs}, we propose efficient detection and recovery algorithms and formally state the algorithmic upper bounds in Theorem~\ref{MAIN-THM-informal}. Section~\ref{sec:stat-analysis} provides the statistical analysis of our algorithms. In Section~\ref{sec:numer-results} we conduct an empirical evaluation of our algorithms and compare their performance with single-channel methods. Several auxiliary proofs are postponed to the appendices to ensure a smooth flow of presentation.

\section{Information-theoretic lower bound}{\label{sec:info-lower-bound}}

In this section, we prove that strong detection and weak recovery are information-theoretically impossible when $L=O(1)$ and $F(\mu,\rho,\gamma,\{ \lambda_{\ell} \},\{ \epsilon_{\ell} \})<1$. Specifically, throughout this section we will assume that 
\begin{equation}{\label{eq-assum-lower-bound}}
    L=O(1) \mbox{ and } F(\mu,\rho,\gamma,\{ \lambda_{\ell} \},\{ \epsilon_{\ell} \}) \leq 1-\delta \mbox{ for some constant } \delta>0 \,.
\end{equation}

\subsection{The detection lower bound}{\label{subsec:info-lower-bound-detection}}

In this subsection we show that strong detection is information-theoretically impossible under \eqref{eq-assum-lower-bound}. The first step of our argument is to employ suitable truncations. Specifically, we choose a sufficiently small constant $\iota=\iota( L,\mu,\rho,\gamma,\{ \lambda_{\ell} \},\{ \epsilon_{\ell} \} )>0$ such that
\begin{align}{\label{eq-choice-iota}}
    \frac{ (1+\iota)^2\mu^2 }{ \gamma } \leq 1-\Omega(1) \mbox{ and } \frac{ (1+\iota)^2\mu^2 }{ \gamma } + \sum_{ 1 \leq \ell \leq L } \frac{ \rho^4 \epsilon_{\ell}^2 \lambda_{\ell} }{ 1-(1-\rho^4) \epsilon_{\ell}^2 \lambda_{\ell} } \leq 1-\Omega(1) \,.
\end{align}
Recall that $\bm u\sim \mathcal N(0,\mathbb I_p)$ under $\Pb$.
Define the event 
\begin{equation}{\label{eq-def-E-diamond}}
    \mathcal E_{\diamond}:=\Big\{ (1-\iota)p \leq \| \bm u \|^2 \leq (1+\iota)p \Big\} \,.
\end{equation}
Define $\widetilde{\Pb}:=\Pb(\cdot \mid\mathcal E_{\diamond})$. Since standard Gaussian concentration inequality implies that $\Pb(\mathcal E_{\diamond})=1-o(1)$, we immediately have $\mathsf{TV}(\widetilde{\Pb},\Pb)=o(1)$, and thus it suffices to show that $\mathsf{TV}(\widetilde{\Pb},\Qb)=o(1)$. Proposition~\ref{prop-bound-chi-2-divergence} gives $\chi^2(\widetilde{\Pb} \| \Qb)=O(1)$, which implies $\mathsf{TV}(\widetilde{\Pb},\Qb)=o(1)$ (see, e.g., ~\cite[Equation 2.27]{Tsy09}), and therefore strong detection is impossible (recall that impossibility for strong detection is equivalent to $\mathsf{TV}(\Pb,\Qb) = o(1)$).

%We now state the main result in this section, which bounds the chi-square divergence between $\widetilde{\Pb}$ and $\Qb$.
\begin{proposition}{\label{prop-bound-chi-2-divergence}}
    Suppose \eqref{eq-assum-lower-bound} holds. Then we have $\chi^2(\widetilde{\Pb} \| \Qb)=O(1)$.
\end{proposition}
\begin{remark}
    It is standard to deduce from Proposition~\ref{prop-bound-chi-2-divergence} that strong detection between $\Pb$ and $\Qb$ is information-theoretically impossible given \eqref{eq-assum-lower-bound}. Thus, we omit further details here and simply refer the readers to \cite[Proposition~2.5]{PWBM18}. 
\end{remark}

Now we present the proof of Proposition~\ref{prop-bound-chi-2-divergence}. To this end, denote by $\mathfrak F$ the distribution of $\| \bm u \|^2/p$ under $\widetilde{\Pb}$. Clearly $\mathfrak F$ is supported in $[1-\iota,1+\iota]$. Also, for any $|t|\leq \iota$ define
\begin{equation}{\label{eq-def-Pb-t}}
    \Pb_{t}:=\Pb\left( \cdot\mid \|\bm u\|^2=(1+t)p \right) \,.
\end{equation}
It is straightforward to check from the Cauchy-Schwarz inequality that
\begin{align*}
    \chi^2( \widetilde{\Pb}\| \Qb ) \leq \mathbb E_{\Qb}\left[ \left( \int_{-\iota}^{\iota} \frac{ \mathrm{d}\Pb_{t} }{ \mathrm{d}\Qb } \mathrm{d}\mathfrak F(t) \right)^2 \right] \leq \mathbb E_{\Qb}\left[ \int_{-\iota}^{\iota} \left( \frac{ \mathrm{d}\Pb_{t} }{ \mathrm{d}\Qb } \right)^2 \mathrm{d}\mathfrak F(t) \right] = \int_{-\iota}^{\iota} \chi^2(\Pb_t \|\Qb) \mathrm{d}\mathfrak F(t) \,.
\end{align*}
Thus, to prove Proposition~\ref{prop-bound-chi-2-divergence} it suffices to bound the chi-square divergence between $\Pb_t$ and $\Qb$, as incorporated in the next lemma.
\begin{lemma}{\label{lem-bound-chi-2-divergence-given-norm}}
    Suppose that \eqref{eq-assum-lower-bound} holds and we choose $\iota$ according to \eqref{eq-choice-iota}. Then we have $\chi^2(\Pb_t\|\Qb)=O(1)$ for all $|t| \leq \iota$.
\end{lemma}
The rest part of this section is devoted to the proof of Lemma~\ref{lem-bound-chi-2-divergence-given-norm}. To this end, note that under $\Pb_t$, we have $\bm u$ is uniformly sampled from the sphere $\sqrt{(1+t)p} \cdot \mathbb S^d$. The next lemma provides a tractable bound on the chi-square divergence $\chi^2(\Pb_t\|\Qb)$.
\begin{lemma}{\label{lem-chi-2-divergence-relax}}
    Suppose \eqref{eq-assum-lower-bound} holds. Then for all $|t| \leq \iota$, we have
    \begin{equation}{\label{eq-chi-2-divergence-relax}}
        \chi^2(\Pb_t\|\Qb) \leq \mathbb E_{ \substack{ \bm x, \bm x_1,\ldots,\bm x_{L} \\ \bm x',\bm x_1',\ldots,\bm x_L' } }\left\{ \exp\left( \frac{(1+t)^2\mu^2}{2\gamma}\left( \frac{ \langle \bm x,\bm x' \rangle }{ \sqrt{n} } \right)^2 + \sum_{1 \leq \ell \leq L} \frac{\epsilon_{\ell}^2\lambda_{\ell}}{2} \left( \frac{ \langle \bm x_{\ell},\bm x_{\ell}' \rangle }{ \sqrt{n} } \right)^2 \right) \right\} \,.
    \end{equation}
    Here $(\bm x',\bm x_1',\ldots,\bm x_L')$ is an independent copy of $(\bm x,\bm x_1,\ldots,\bm x_L)$ (as defined in \eqref{eq-def-x-ell}).
\end{lemma}

The proof of Lemma~\ref{lem-chi-2-divergence-relax} is postponed to Section~\ref{subsec:proof-lem-2.4} of the appendix.
In light of Lemma~\ref{lem-chi-2-divergence-relax}, in order to prove Lemma~\ref{lem-bound-chi-2-divergence-given-norm}, it suffices to bound the right-hand side of \eqref{eq-chi-2-divergence-relax}. The main challenge is that the inner products $\langle \bm x,\bm x' \rangle$ and $\langle \bm x_{\ell},\bm x_{\ell}' \rangle$ have non-vanishing correlations, leading to the breakdown of existing methods for the single signal setting (e.g., the method in \cite{PWBM18} or in \cite{LS23}). However, there is a clear intuition on how to control the right-hand side of \eqref{eq-chi-2-divergence-relax}: from central limit theorem, we expect that 
\begin{align}
    \left( \frac{ \langle \bm x,\bm x' \rangle }{ \sqrt{n} }, \frac{ \langle \bm x_1,\bm x_1' \rangle }{ \sqrt{n} }, \ldots, \frac{ \langle \bm x_L,\bm x_L' \rangle }{ \sqrt{n} } \right) \mbox{ behaves like } \left( U,V_1,\ldots,V_{\ell} \right) \,,  \label{eq-intuition-Gaussian-approx}
\end{align}
where $U,V_1,\ldots,V_L$ are jointly mean-zero Gaussian variables such that
\begin{align}
    \mathbb E\big[ U^2 \big]=\mathbb E\big[ V_\ell^2 \big]=1 \mbox{ and } \mathbb E\big[ UV_{\ell} \big]=\rho^2, \mathbb E\big[ V_{\ell}V_{\ell'} \big]=\rho^4 \,.  \label{eq-Gaussian-covariance-structure}
\end{align}
Note that equivalently we can write $V_{\ell}=\rho^2 U+\sqrt{1-\rho^4} W_\ell$ such that $U,W_1,\ldots,W_L$ are i.i.d.\ standard normal variables. The key in our argument is to show an analog of the Gaussian approximation scheme in \eqref{eq-intuition-Gaussian-approx} designed specifically for the Rademacher variables, as incorporated in the next lemma.
\begin{lemma}{\label{lem-Bernoulli-Gaussian-moment-compare}}
    Let $U,V_1,V_L$ are jointly mean-zero Gaussian variables defined in \eqref{eq-Gaussian-covariance-structure}. Then for any $\alpha,\alpha_1,\ldots,\alpha_L \in \mathbb N$, we have
    \begin{align*}
        \mathbb E_{ \substack{ \bm x, \bm x_1,\ldots,\bm x_{L} \\ \bm x',\bm x_1',\ldots,\bm x_L' } }\left[ \left( \frac{ \langle \bm x,\bm x' \rangle }{ \sqrt{n} } \right)^{2\alpha} \left( \frac{ \langle \bm x_1,\bm x_1' \rangle }{ \sqrt{n} } \right)^{2\alpha_1} \ldots \left( \frac{ \langle \bm x_L,\bm x_L' \rangle }{ \sqrt{n} } \right)^{2\alpha_L} \right] \leq \mathbb E_{U,V_1,\dots,V_L}\left[ U^{2\alpha} V_1^{2\alpha_1} \ldots V_L^{2\alpha_L} \right] \,.
    \end{align*}
    In particular, by Taylor expansion, we have that the right-hand side of \eqref{eq-chi-2-divergence-relax} is bounded by
    \begin{equation}{\label{eq-chi-2-divergence-relax-2}}
        \mathbb E\left[ \exp\left( \frac{(1+t)^2\mu^2}{2\gamma} \cdot U^2 + \sum_{1 \leq \ell \leq L} \frac{ \epsilon_{\ell}^2 \lambda_{\ell} }{2} \cdot V_{\ell}^2 \right) \right] \,.
    \end{equation}
\end{lemma}

The proof of Lemma~\ref{lem-Bernoulli-Gaussian-moment-compare} is incorporated in Section~\ref{subsec:proof-lem-2.5} of the appendix. Now we can finish the proof of Lemma~\ref{lem-bound-chi-2-divergence-given-norm}.

\begin{proof}[Proof of Lemma~\ref{lem-bound-chi-2-divergence-given-norm}]
    Using Lemmas~\ref{lem-chi-2-divergence-relax} and \ref{lem-Bernoulli-Gaussian-moment-compare}, it suffices to show that $\eqref{eq-chi-2-divergence-relax-2}=O(1)$ given \eqref{eq-choice-iota} and $L=O(1)$. Since we can write $V_{\ell}=\rho^2 U+\sqrt{1-\rho^4}W_\ell$ such that $U,W_1,\ldots,W_L$ are i.i.d.\ standard normal variables, we then have
    \begin{align*}
        \eqref{eq-chi-2-divergence-relax-2} &= \mathbb E\left[ \exp\left( \frac{(1+t)^2\mu^2}{2\gamma} \cdot U^2 + \sum_{1 \leq \ell \leq L} \frac{ \epsilon_{\ell}^2 \lambda_{\ell} }{2} \cdot \big( \rho^2 U+\sqrt{1-\rho^4}W_\ell \big)^2 \right) \right] \\
        &= \prod_{1\le \ell\le L}\left( 1-\epsilon_\ell^2\lambda_\ell(1-\rho^4) \right)^{-1/2} \cdot \mathbb E\left[ \exp\left( U^2 \left( \frac{(1+t)^2\mu^2}{2\gamma} + \sum_{1 \leq \ell \leq L} \frac{ \rho^4 (\epsilon_{\ell}^2 \lambda_{\ell})^2 }{ 2( 1-(1-\rho^4)\epsilon_\ell^2 \lambda_\ell ) } \right) \right) \right] \\ &= \Bigg( 1-\frac{(1+t)^2\mu^2}{\gamma} - \sum_{1 \leq \ell \leq L} \frac{ \rho^4 (\epsilon_{\ell}^2 \lambda_{\ell})^2 }{ ( 1-(1-\rho^4)\epsilon_\ell^2 \lambda_\ell ) }\Bigg)^{-1/2}\prod_{1\le \ell\le L}(1-\epsilon_\ell^2\lambda_\ell(1-\rho^4))^{-1/2} = O(1) \,,
    \end{align*}
    where the second equality follows from the fact that for $W\sim\mathcal N(0,1)$ and $2AC^2<1$, 
    \begin{align*}
        \E\left[ e^{ A(BU+CW)^2 } \mid U \right]=(1-2AC^2)^{-1/2} \exp\left(\frac{AB^2U^2}{1-2AC^2}\right) \,,
    \end{align*}
    here $2AC^2=(1-\rho^4)\epsilon_\ell^2\lambda_\ell=1-\Omega(1)$ for all $1\le \ell\le L$, so the condition holds; the last equality follows from $\E[\exp(AU^2)]=(1-2A)^{-1/2}$ for $A<\tfrac12$, as well as
    \begin{align*}
        \frac{(1+t)^2\mu^2}{2\gamma}+\sum_{1\le \ell\le L} \frac{\rho^4(\epsilon_\ell^2\lambda_\ell)^2}{2(1-(1-\rho^4)\epsilon_\ell^2\lambda_\ell)}<\frac12
    \end{align*}
    by \eqref{eq-choice-iota} and $|t|\le \iota$.
\end{proof}

\subsection{The recovery lower bound}{\label{subsec:info-lower-bound-recovery}}

In this subsection, we show that weak recovery is information-theoretically impossible under \eqref{eq-assum-lower-bound}. As we mentioned in Section~\ref{subsec:related-works}, we will reduce the recovery lower bound to the chi-square divergence calculation in Section~\ref{subsec:info-lower-bound-detection}.
Our argument combines a proof by contradiction with the above reduction technique.
Specifically, we summarize our proof outline for the impossibility of weak recovery in Figure~\ref{fig:logic-flow-impossible-IT}.

\begin{figure}[ht]
    \centering
    \includegraphics[width=0.63\linewidth]{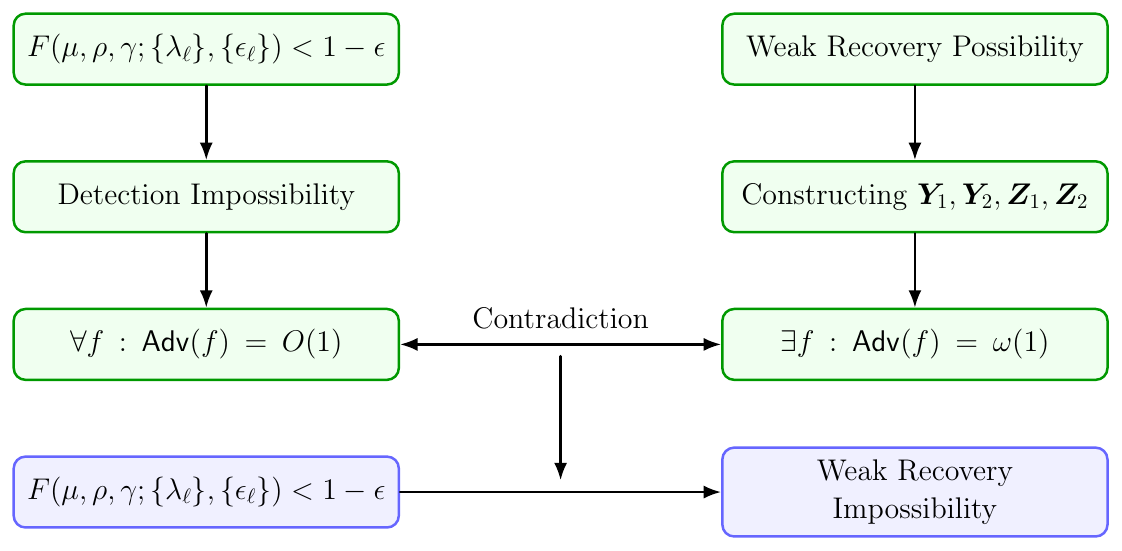}
    \caption{Proof outline for the impossibility of weak recovery}
    \label{fig:logic-flow-impossible-IT}
\end{figure}

To this end, recall Definition~\ref{def-contxtual-multilayer-SBM}. Note that the problem becomes easier as $\mu$ increases. Moreover, for any $\epsilon\in (0,1)$ and all parameters such that $F(\mu,\rho,\gamma;\{ \lambda_{\ell} \},\{ \epsilon_{\ell} \})<1-\epsilon$, we can always choose a small constant $\mu<\mu'$ such that $F(\mu',\rho,\gamma;\{ \lambda_{\ell} \},\{ \epsilon_{\ell} \})<1-\frac{\epsilon}{2}$. Thus, without loss of generality we may assume that $\mu=\Theta(1)$. The first step of our argument is to introduce external randomness by sampling $\bm W \in \mathbb R^{n*p}$ constituting i.i.d.\ standard normal entries (independent of $(\bm Y,\bm G_1,\ldots,\bm G_{L})$). Denote $\Pb_{\bullet}$ as the joint law of $(\bm Y,\bm G_1,\ldots,\bm G_L) \sim \Pb$ and $\bm W$, and denote $\Qb_{\bullet}$ as the joint law of $(\bm Y,\bm G_1,\ldots,\bm G_L) \sim \Qb$ and $\bm W$. In addition, recall \eqref{eq-def-E-diamond}. Denote $\widetilde{\Pb}_{\bullet}$ as the joint law of $(\bm Y,\bm G_1,\ldots,\bm G_L) \sim \Pb(\cdot\mid\mathcal E_{\diamond})$ and $\bm W$. We first derive an immediate corollary of Proposition~\ref{prop-bound-chi-2-divergence}.

\begin{lemma}{\label{lem-bounded-Adv}}
    For all functions $f$ measurable with $(\bm Y,\bm G_1,\ldots,\bm G_L)$ and $\bm W$, we have
    \begin{equation}{\label{eq-bounded-Adv}}
        \mathsf{Adv}(f):= \frac{ \mathbb E_{\widetilde\Pb_{\bullet}}[f] }{ \sqrt{ \mathbb E_{\Qb_{\bullet}}[f^2] } } = O(1) \,. 
    \end{equation}
\end{lemma}
\begin{proof}
    Note that $\frac{\mathrm{d}\widetilde\Pb_{\bullet}}{\mathrm{d}\Qb_{\bullet}} (\bm W,\bm Y,\bm G_1,\ldots,\bm G_L) = \frac{\mathrm{d}\widetilde\Pb}{ \mathrm{d}\Qb }(\bm Y,\bm G_1,\ldots,\bm G_L)$. Using Cauchy-Schwarz inequality, we then have (let $\mathcal L=\frac{ \mathrm{d}\widetilde\Pb_{\bullet} }{ \mathrm{d}\Qb_{\bullet} }$)
    \begin{align*}
        \mathsf{Adv}(f) = \frac{ \mathbb E_{\Qb_{\bullet}}[ \mathcal L \cdot f ] }{ \sqrt{ \mathbb E_{\Qb_{\bullet}}[f^2] } } \leq \sqrt{ \mathbb E_{\Qb_{\bullet}}[\mathcal L^2] } = \sqrt{ \chi^2( \widetilde\Pb_{\bullet} \| \Qb_{\bullet} ) } = \sqrt{ \chi^2(\widetilde\Pb\|\Qb) } = O(1) \,,
    \end{align*}
    where the last equation follows from Proposition~\ref{prop-bound-chi-2-divergence}.
\end{proof}

In the rest of this subsection, we will show the following result.
\begin{lemma}{\label{lem-reduction}}
    Suppose that for some $F(\mu,\rho,\gamma,\{ \lambda_{\ell} \},\{ \epsilon_{\ell} \})<1-\epsilon$, given  
    \begin{align*}
        (\bm Y,\bm G_1,\ldots,\bm G_L) \sim \mathcal S(n,p;\mu,\rho,\{ \lambda_{\ell} \},\{ \epsilon_{\ell} \})
    \end{align*}
    there exists an estimator $\mathcal X(\bm Y,\bm G_1,\ldots,\bm G_L)$ achieves weak recovery in the sense of Definition~\ref{def-weak-recovery}. Then for some $F(\mu',\rho,\gamma,\{ \lambda_{\ell} \},\{ \epsilon_{\ell} \})<1-\epsilon'$ and 
    \begin{align*}
        (\bm Y,\bm G_1,\ldots,\bm G_L) \sim \mathcal S(n,p;\mu',\rho,\{ \lambda_{\ell} \},\{ \epsilon_{\ell} \}) \,,
    \end{align*}
    there exists a function $f=f(\bm W,\bm Y,\bm G_1,\ldots,\bm G_L)$ (here $\bm W$ is the external randomness) such that $\mathsf{Adv}(f)=\omega(1)$. 
\end{lemma}

\begin{remark}
    Combining Lemmas~\ref{lem-bounded-Adv} and \ref{lem-reduction}, we know that if \eqref{eq-assum-lower-bound} holds, then weak recovery is information-theoretically impossible.
\end{remark}

The rest of this subsection is devoted to the proof of Lemma~\ref{lem-reduction}. To this end, suppose that for some parameters $\mu,\rho,\gamma,\{ \lambda_{\ell} \},\{ \epsilon_{\ell} \}$ satisfying \eqref{eq-assum-lower-bound}, given 
\begin{align*}
    (\bm Y,\bm G_1,\ldots,\bm G_L) \sim \mathcal S(n,p,\mu,\rho,\{ \lambda_{\ell} \},\{ \epsilon_{\ell} \}),
\end{align*}
then there exists an estimator $\mathcal X(\bm Y,\bm G_1,\ldots,\bm G_L)$ satisfying Definition~\ref{def-weak-recovery}, i.e., 
\begin{align*}
    \mathbb E_{\Pb}\Bigg[ \frac{ \langle \mathcal X, \bm x \bm x^{\top} \rangle }{ \| \mathcal X \|_{\Fop} \| \bm x \bm x^{\top} \|_{\Fop} } \Bigg] \geq c \mbox{ for some constant } c>0 \,.
\end{align*}
Using standard Markov's inequality, we then have
\begin{align*}
    \Pb\Bigg( \frac{ \langle \mathcal X, \bm x \bm x^{\top} \rangle }{ \| \mathcal X \|_{\Fop} \| \bm x \bm x^{\top} \|_{\Fop} } \geq \frac{c}{2} \Bigg) = 1-\Pb \Bigg(1-\frac{ \langle \mathcal X, \bm x \bm x^{\top} \rangle }{ \| \mathcal X \|_{\Fop} \| \bm x \bm x^{\top} \|_{\Fop} }\ge 1-\frac{c}{2}\Bigg)\ge 1-\frac{1-c}{1-\frac{c}{2}} \geq \frac{c}{2} \,.
\end{align*}
Recall the choice of $\iota$ in \eqref{eq-choice-iota} and the definition of $\widetilde{\Pb}=\Pb(\cdot\mid\mathcal E_{\diamond})$ in \eqref{eq-def-E-diamond}. Since $\mathsf{TV}(\Pb,\widetilde{\Pb})=o(1)$, we then have
\begin{align*}
    \widetilde\Pb\Bigg( \frac{ \langle \mathcal X, \bm x \bm x^{\top} \rangle }{ \| \mathcal X \|_{\Fop} \| \bm x \bm x^{\top} \|_{\Fop} } \geq \frac{c}{2} \Bigg) \geq \frac{c}{2}-o(1) \,.
\end{align*}
We now choose a sufficiently small constant $\kappa>0$ such that (recall \eqref{eq-choice-iota})
\begin{equation}{\label{eq-choice-kappa}}
    F\left( (1+\kappa^2)(1+\iota)\mu, \rho,\gamma,\{ \lambda_{\ell} \},\{ \epsilon_{\ell} \} \right) < 1-\Omega(1) \,.
\end{equation}
Now, suppose $(\bm Y,\bm G_1,\ldots,\bm G_L) \sim \mathcal S(n,p, \mu(1+\kappa^2),\rho,\{ \lambda_{\ell} \},\{ \epsilon_{\ell} \} )$ and recall we introduce external randomness $\bm W$. We set
\begin{align}
    \bm Y_1 = \frac{ 1 }{ \sqrt{1+\kappa^2} } \left( \bm Y+ \kappa \bm W \right), \quad \bm Y_2 = \frac{ 1 }{ \sqrt{1+\kappa^{-2}} } \left( \bm Y - \kappa^{-1} \bm W \right) \,.  \label{eq-split-into-two-matrices}
\end{align}
The reasons for the definition in \eqref{eq-split-into-two-matrices} is as follows (recall Definition~\ref{def-contxtual-multilayer-SBM}):
\begin{itemize}
    \item Under $(\bm Y,\bm G_1,\ldots,\bm G_L) \sim\Qb$, we have $(\bm Y_1,\bm Y_2)=(\bm Z_1,\bm Z_2)$, where
    \begin{align}
        \bm Z_1 = \frac{ 1 }{ \sqrt{1+\kappa^2} } \left( \bm Z+ \kappa \bm W \right), \quad \bm Z_2 = \frac{ 1 }{ \sqrt{1+\kappa^{-2}} } \left( \bm Z - \kappa^{-1} \bm W \right) \,.  \label{eq-behavior-Y-1,2-Qb}
    \end{align}
    In particular, $\bm Z_1,\bm Z_2$ are matrices with i.i.d.\ standard normal entries and $\bm Z_2$ is \underline{independent of} $(\bm Z_1,\bm G_1,\ldots,\bm G_L)$.
    \item Under $(\bm Y,\bm G_1,\ldots,\bm G_L) \sim\Pb$, we have
    \begin{align}
        \bm Y_1 = \frac{ \sqrt{\mu} }{ \sqrt{n} } \bm x \bm u^{\top} + \bm Z_1, \quad \bm Y_2 = \frac{ \sqrt{ \mu(1+\kappa^2) } }{ \sqrt{n(1+\kappa^{-2})} } \bm x \bm u^{\top} + \bm Z_2 \,,  \label{eq-behavior-Y-1,2-Pb}
    \end{align}
    where 
    \begin{align*}
        \bm Z_1 = \frac{ 1 }{ \sqrt{1+\kappa^2} } \left( \bm Z+ \kappa \bm W \right), \quad \bm Z_2 = \frac{ 1 }{ \sqrt{1+\kappa^{-2}} } \left( \bm Z - \kappa^{-1} \bm W \right) \,.
    \end{align*}
    In particular, $\bm Z_1,\bm Z_2$ are matrices with i.i.d.\ standard normal entries and $\bm Z_2$ is \underline{independent of} $(\bm Y_1,\bm G_1,\ldots,\bm G_L)$. Also we have $(\bm Y_1,\bm G_1,\ldots,\bm G_L) \sim \mathcal S(n,p;\mu,\rho,\gamma,\{ \lambda_{\ell} \},\{ \epsilon_{\ell} \})$.
\end{itemize}

Now, since $(\bm Y_1,\bm G_1,\ldots,\bm G_L) \sim \mathcal S(n,p;\mu,\rho,\gamma,\{ \lambda_{\ell} \},\{ \epsilon_{\ell} \})$ under $\Pb_{\bullet}$, we can find an estimator 
\begin{align}{\label{eq-generate-estimator}}
    \mathcal X( \bm Y_1,\bm G_1,\ldots,\bm G_L ) \mbox{ such that } \widetilde\Pb_{\bullet}\Bigg( \frac{ \langle \mathcal X, \bm x \bm x^{\top} \rangle }{ \| \mathcal X \|_{\Fop} \| \bm x \bm x^{\top} \|_{\Fop} } \geq \frac{c}{2} \Bigg) \geq \frac{c}{2}-o(1)
\end{align}
The next lemma shows that $\langle \mathcal X, \bm Y_2 \bm Y_2^{\top} - p \mathbb I_n \rangle$ is ``large'' under $\widetilde{\Pb}_{\bullet}$ and ``small'' under $\Qb_{\bullet}$.
\begin{lemma}{\label{lem-behavior-Pb-Qb}}
    Suppose we choose $\mathcal X( \bm Y_1,\bm G_1,\ldots,\bm G_L )$ as in \eqref{eq-generate-estimator}. Then 
    \begin{align}
        &\widetilde{\Pb}_{\bullet}\left( \big\langle \mathcal X, \bm Y_2 \bm Y_2^{\top} - p \mathbb I_n \big\rangle \geq \frac{ c\mu(1-\iota)(1+\kappa^2)n }{ 4(1+\iota)(1+\kappa^{-2}) } \| \mathcal X \|_{\Fop} \right) \geq \frac{c}{2}-o(1) \,;  \label{eq-behavior-Pb} \\
        &\Qb_{\bullet}\left( \big\langle \mathcal X, \bm Y_2 \bm Y_2^{\top} - p \mathbb I_n \big\rangle \geq \frac{ c\mu (1-\iota)(1+\kappa^2)n }{ 4(1+\iota)(1+\kappa^{-2}) } \| \mathcal X \|_{\Fop} \right) \leq n^{-\Omega(1)} \,.  \label{eq-behavior-Qb}
    \end{align}
\end{lemma}
The proof of Lemma~\ref{lem-behavior-Pb-Qb} is postponed to Section~\ref{subsec:proof-lem-2.7} of the appendix. Now we can finish the proof of Lemma~\ref{lem-reduction}.
\begin{proof}[Proof of Lemma~\ref{lem-reduction}]
    Define $\mu'=\mu\sqrt{1+\kappa^2}$ and  
    \begin{equation}{\label{eq-choice-contradicting-f}}
        f=f(\bm W;\bm Y,\bm G_1,\ldots,\bm G_L):= \mathbf 1\left\{ \big\langle \mathcal X, \bm Y_2 \bm Y_2^{\top} - p \mathbb I_n \big\rangle \geq \frac{ c\mu(1-\iota)(1+\kappa^2)n }{ 4(1+\iota)(1+\kappa^{-2}) } \| \mathcal X \|_{\Fop} \right\} \,.
    \end{equation}
    Lemma~\ref{lem-behavior-Pb-Qb} then implies that $\mathbb E_{\Pb}[f]=\Omega(1)$ and $\mathbb E_{\Qb}[f^2] \leq n^{-\Omega(1)}$. This yields that $\mathsf{Adv}(f)\geq n^{\Omega(1)}$ and thus completes the proof of Lemma~\ref{lem-reduction}.
\end{proof}

\begin{remark}{\label{remark-larger-L}}
    It is easy to check that our reduction remains valid as long as $\chi^2(\widetilde{\Pb}\|\Qb)=n^{o(1)}$. Since our proof of Proposition~\ref{prop-bound-chi-2-divergence} yields that $\chi^2(\widetilde{\Pb}\|\Qb)=n^{o(1)}$ as long as $F(\mu,\rho,\gamma;\{ \lambda_{\ell} \},\{ \epsilon_{\ell} \})<1$ and $L=o(\log n)$, our argument actually shows that recovery is information-theoretically impossible as long as $F(\mu,\rho,\gamma;\{ \lambda_{\ell} \},\{ \epsilon_{\ell} \})<1$ and $L=o(\log n)$.
\end{remark}

%\remarkhd{To be added by hd: logic flow figure and intuition explanation on the construction for $\mathbf 1 \{...\}$}

\section{Efficient algorithms based on counting decorated graphs}{\label{sec:efficient-algs}}

In this section, we prove that strong detection and weak recovery can be achieved by efficient algorithms when $L=O(1)$ and $F(\mu,\rho,\gamma;\{ \lambda_{\ell} \},\{ \epsilon_{\ell} \})>1$. Note that when $\frac{\mu^2}{\gamma}>1$, we are able to detect and estimate the latent labeling $\bm x$ using the single observation $\bm Y$. Similarly, when $\epsilon_{\ell}^2 \lambda_{\ell}>1$, we are able to detect and estimate the labeling $\bm x_{\ell}$ using the single observation $\bm G_{\ell}$\footnote{In addition, since $\bm x,\bm x_{\ell}$ are positively correlated, we have that any estimator $\mathcal X(\bm G_{\ell})$ that is positively correlated with $\bm x_{\ell} \bm x_{\ell}^{\top}$ is also positively correlated with $\bm x \bm x^{\top}$.}. Thus, throughout this section we will assume that
\begin{equation}{\label{eq-assum-upper-bound}}
    \frac{\mu^2}{\gamma}, \epsilon^2_1 \lambda_1,\ldots, \epsilon^2_{L} \lambda_L \leq 1, \quad \frac{ \mu^2 }{ \gamma } + \sum_{\ell=1}^{L} \frac{ \rho^4 (\epsilon_{\ell}^2 \lambda_{\ell})^2 }{ 1-(1-\rho^4) (\epsilon_{\ell}^2 \lambda_{\ell})^2 } > 1+\epsilon \mbox{ for some constant } \epsilon>0 \,.
\end{equation}
Recall that $\operatorname{U}_n$ is the set of unordered pairs $(i,j)$ where $1\le i<j\le n$. We first introduce the notion of decorated graphs, which is the key to our algorithm. See Figure~\ref{fig:decorated-cycle-path} for examples on decorated graphs.
\begin{DEF}[Factor graph]{\label{def-link-graph}}
    Denote $\mathsf K_{n,p}$ as the graph generated as follows. Define the vertex set of $\mathsf{K}_{n,p}$ as $V^{\mathsf a} \cup V^{\mathsf b}$, where $V^{\mathsf a}=\{ a_1,\ldots,a_n \}$ and $V^{\mathsf b}=\{ b_1,\ldots,b_p \}$. In addition, define the edge set of $\mathsf K_{n,p}$ as $E^{\mathsf a} \cup E^{\mathsf b}$, where 
    \begin{align*}
        E^{\mathsf a}=\big\{ (a_i,a_j): (i,j) \in \operatorname{U}_n \big\}, \quad E^{\mathsf b}=\big\{ (a_i,b_k): i \in [n], k \in [p] \big\} \,.
    \end{align*}
\end{DEF}
\begin{DEF}[Decorated graphs]{\label{def-decorated-graph}}
    For any graph $H\subset \mathsf{K}_{n,p}$, we say $\chi_H:E(H) \to \{ 0,1,\ldots,L \}$ is a \emph{decoration} of $E(H)$, if
    %\remarksg{It would be better to replace $\gamma(e)$ by some other notation? It is the same as the ratio $n/p$. Will edit.}
    \begin{align*}
        \chi_H(e)=0 \Longleftrightarrow e \in E^{\mathsf b}(\mathsf K_{n,p}) \,.
    \end{align*}
    We say $H=(V(H),E(H),\chi_H)$ is a decorated graph, if $(V(H),E(H))$ is a graph and $\chi_H$ is a decoration of $E(H)$. We define
    \begin{equation}{\label{eq-def-V-a-V-b-E-a-E-b(H)}}
        \begin{aligned}
            & V^{\mathsf a}(H)=V^{\mathsf a} \cap V(H), \quad V^{\mathsf b}(H)=V^{\mathsf b} \cap V(H) \,; \\
            & E^{\mathsf a}(H)=E^{\mathsf a} \cap E(H), \quad E^{\mathsf b}(H)=E^{\mathsf b} \cap E(H) \,.
        \end{aligned}
    \end{equation}
    For $0 \leq \ell \leq L$, define
    \begin{equation}{\label{eq-def-E-ell-decorated}}
        E_{\ell}(H)=\big\{ e \in E(H): \chi_H(e)=\ell \big\} \,.
    \end{equation}
    (Note that $E_0(H)=E^{\mathsf b}(H)$ by definition).
    In addition, for $0 \leq \ell \leq L$ define
    \begin{equation}{\label{eq-def-V-ell-decorated}}
        V_{\ell}(H) = \big\{ v \in V(H): (v,u) \in E_{\ell}(H) \mbox{ for some } u \in V(H) \big\} \,.
    \end{equation}
    Also define $H_{\ell}$ to be the subgraph of $H$ with
    \begin{equation}{\label{eq-def-H-ell}}
        V(H_{\ell}) = V_{\ell}(H), \ E(H_{\ell})=E_{\ell}(H) \,.
    \end{equation}
    We say $H=(V(H),E(H),\chi_H)$ is a decorated cycle (respectively, decorated path), if $(V(H),E(H))$ is a cycle (respectively, path). Finally, for a decorated cycle or a decorated path $H$, define
    \begin{equation}{\label{eq-def-mathsf-dif}}
    \begin{aligned}
        &\mathsf{dif}_0(H):= V_0(H) \cap \big( \cup_{1 \leq \ell \leq L} V_{\ell}(H) \big) \,; \\ 
        &\mathsf{dif}(H):= \cup_{1 \leq \ell<\ell' \leq L} (V_{\ell}(H) \cap V_{\ell'}(H)) \,.
    \end{aligned}
    \end{equation}
    In addition, denote by $S \Cap H$ the graph induced by the edge set $\cup_{0 \leq \ell \leq L} (E_\ell(S)\cap E_\ell(H))$ (in particular, edge induced graphs have no isolated vertices), and let $S \cap H$ be the \emph{undecorated graph} with $V(S \cap K)=V(S) \cap V(K)$ and $E(S \cap K)=E(S) \cap E(K)$. 
\end{DEF}

\subsection{The detection statistic and theoretical guarantees}{\label{subsec:detect-stat}}

We first focus on the detection problem. Recall \eqref{eq-def-contextual-multilayer-SBM}. For each $1 \leq \ell \leq L$ and $(i,j) \in \operatorname{U}_n$, define
\begin{align}
    \overline{\bm G}_{\ell}(i,j):= \frac{1}{\sqrt{\lambda_{\ell}/n}} \Big( \bm G_{\ell}(i,j) - \frac{\lambda_{\ell}}{n} \Big) \,. \label{eq-def-overline-G-ell-i,j}
\end{align}
In addition, for any decorated graph $S \subset \mathsf{K}_{n,p}$, we define
\begin{equation}{\label{eq-def-f-H,A}}
    f_{S} := \prod_{ (a_i,b_k) \in E_0(S) } \bm Y(i,k) \prod_{1 \leq \ell \leq L} \prod_{(a_i,a_j) \in E_\ell(S) } \overline{\bm G}_{\ell}(i,j) \,.
\end{equation}
Let $\aleph\in\mathbb N$ be a parameter that will be decided later and define $\mathcal H=\mathcal H(\aleph)$ to be the collection of unlabeled decorated cycles $[H]$ such that $|V^{\mathsf a}(H)|=\aleph$. It is clear that for all $[H] \in \mathcal H(\aleph)$ we have (see Figure~\ref{fig:decorated-cycle})
\begin{equation}{\label{eq-basic-property-mathcal-H}}
    |E(H)|=\aleph+\tfrac{1}{2}|E_0(H)|, \quad |V^{\mathsf b}(H)|=\tfrac{1}{2}|E_0(H)| \,.
\end{equation}

% =========================
% Preamble: make sure you have
% =========================
% \usepackage{tikz}
% \usepackage{subfig}   % provides \subfloat
% and your color defs, e.g. myRed/myBlue/myGreen/myPurple
% =========================

\begin{figure}[htbp]
\centering

% ============================================================
% (a) original schematic (kept as-is)
% ============================================================
\subfloat[A decorated cycle $H$ with $L=3$ colors. Here $V^{\mathsf{b}}(H) = \{v_2,v_6\}$, $V^{\mathsf{a}}(H) = V(H)\setminus V^{\mathsf b}(H)$, $\mathsf{dif}_0(H) = \{v_1,v_3,v_5,v_7\}$, and $\mathsf{dif}(H) = \{v_4,v_9,v_{11}\}$.\label{fig:decorated-cycle}]{
\includegraphics[width=0.45\linewidth]{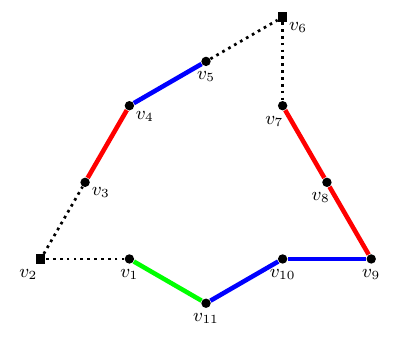}
}\hfill 
\subfloat[A decorated path $H$ with $L=3$ colors. Here $V^{\mathsf{b}}(H) = \{v_2,v_5,v_{11}\}$, $V^{\mathsf{a}}(H) = V(H)\setminus V^{\mathsf b}(H)$, $\mathsf{dif}_0(H) = \{v_3,v_4,v_6,v_{10}\}$, and $\mathsf{dif}(H) = \{v_7,v_{8}\}$.
\label{fig:decorated-path}]{
\includegraphics[width=0.45\linewidth]{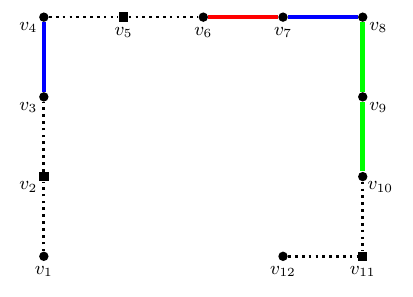}
}
\caption{Decorated cycle and path: disks denote vertices in $V^{\mathsf a}(H)$, squares denote vertices in $V^{\mathsf b}(H)$, and the edge colors (among $L$ total colors) indicate which graph $\bm{G}_\ell$ the edge is drawn from, $\ell \in [L]$.}
\label{fig:decorated-cycle-path}
%\remarksg{It is advisable to replace some edges in the middle part of right-hand side (e.g., $v5-v6$) by some dotted lines.}
\end{figure}

%\begin{figure}
%    \centering
%    \includegraphics[width=0.4\linewidth]{decorated_cycle.pdf}
%        \caption{Decorated cycle: disks denote vertices in $V^{\mathsf a}(H)$, squares denote vertices in $V^{\mathsf b}(H)$, and the edge colors (among $L$ total colors) indicate which graph $\bm{G}_\ell$ the edge is drawn from, $\ell \in [L]$.}
%    \label{}
%\end{figure}
In addition, define 
\begin{equation}{\label{eq-def-f-mathcal-H}}
\begin{split}
    f_{\mathcal H}%&:=\frac{1}{\sqrt{\beta_{\mathcal H}}}\sum_{ [H] \in \mathcal H } \mathbb E_\mathbb P[f_H] \sum_{ S \subset \mathsf{K}_{n,p},S \cong H } f_{S}\\
    &:=\frac{1}{\sqrt{\beta_{\mathcal H}}}\sum_{ [H] \in \mathcal H } \frac{ \Xi(H) }{ n^{\frac{1}{2}\aleph} p^{ \frac{1}{4}|E_0(H)| }  } \sum_{ S \subset \mathsf{K}_{n,p},S \cong H } f_{S} \,.
\end{split}
\end{equation}
Here
\begin{equation}{\label{eq-def-Xi(H)}}
    \Xi(H)= \rho^{ |\mathsf{dif}_0(H)|+2|\mathsf{dif}(H)| } \big( \tfrac{\mu^2}{\gamma} \big)^{ \frac{1}{4}|E_0(H)| } \prod_{1 \leq \ell \leq L} \big( \epsilon^2_{\ell} \lambda_{\ell} \big)^{ \frac{1}{2}|E_{\ell}(H)| } 
\end{equation}
and 
\begin{equation}{\label{eq-def-beta-mathcal-H}}
    \beta_{\mathcal H} = \sum_{ [H] \in \mathcal H } \frac{ \Xi(H)^2 }{ |\mathsf{Aut}(H)| } \,.
\end{equation}
We remark that the coefficients in \eqref{eq-def-f-mathcal-H} are chosen so that 
\begin{align*}
    \mathbb E_\mathbb P[f_H]=[1+o(1)] \frac{ \Xi(H) }{ n^{\frac{1}{2}\aleph} p^{ \frac{1}{4}|E_0(H)| }  } \,, 
\end{align*}
and we refer to Section~\ref{sec:stat-analysis} for more details of the choice of such coefficients.

Algorithm~\ref{alg:detection-meta} below describes our proposed method for detection in contextual multi-layer SBMs.
\begin{breakablealgorithm}{\label{alg:detection-meta}}
\caption{Detection in contextual multi-layer SBMs}
    \begin{algorithmic}[1]
    \STATE {\bf Input:} A rectangular matrix $\bm Y$ and $L$ adjacency matrices $\bm G_1,\ldots,\bm G_L$, a family $\mathcal H$ of non-isomorphic decorated graphs, and a threshold $\tau\geq 0$.
    \STATE Compute $f_{\mathcal H}(\bm Y,\bm G_1,\ldots,\bm G_L)$ according to \eqref{eq-def-f-mathcal-H}. 
    \STATE Let $\mathtt q=1$ if $f_{\mathcal H}(\bm Y,\bm G_1,\ldots,\bm G_L) \geq \tau$ and $\mathtt q=0$ if $f_{\mathcal H}(\bm Y,\bm G_1,\ldots,\bm G_L) < \tau$.
    \STATE {\bf Output:} $\mathtt q$.
    \end{algorithmic}
\end{breakablealgorithm}
We now show that Algorithm~\ref{alg:detection-meta} succeeds under \eqref{eq-assum-upper-bound} and an appropriate choice of the parameter $\aleph$.
The key ingredients are (i) a lower bound on the signal mean $\E_\Pb[f_{\mathcal H}]$, and (ii) control of the fluctuations of $f_{\mathcal H}$, quantified by $\E_{\mathcal Q}[f^2_{\mathcal H}]$ and $\operatorname{Var}_\Pb[f_{\mathcal H}]$, both of which rely on bounds $\beta_{\mathcal H}$.
To this end, we first show the following bound on $\beta_{\mathcal H}$. 
\begin{lemma}{\label{lem-bound-beta-mathcal-H}}
    Define the matrix
    \begin{equation}{\label{eq-def-mathbf-P}}
        \mathbf P = 
        \begin{pmatrix}
            \frac{\mu^2}{\gamma} & 0 & 0 & \ldots & 0 \\
            0 & \epsilon^2_1 \lambda_1 & 0 & \ldots & 0 \\
            0 & 0 & \epsilon^2_2 \lambda_2 & \ldots & 0 \\
            \vdots & \vdots & \vdots & \ddots & \vdots  \\
            0 & 0 & 0 & \ldots & \epsilon^2_L \lambda_L
        \end{pmatrix}
        \begin{pmatrix}
            1 & \rho^2 & \rho^2 & \ldots & \rho^2 \\
            \rho^2 & 1 & \rho^4 & \ldots & \rho^4 \\
            \rho^2 & \rho^4 & 1 & \ldots & \rho^4 \\
            \vdots & \vdots & \vdots & \ddots & \vdots \\
            \rho^2 & \rho^4 & \rho^4 & \ldots & 1 
        \end{pmatrix} \,.
    \end{equation}
    Also denote $\sigma_+(\mathbf P)$ the largest eigenvalue of $\mathbf P$. Then there exists a constant $D=\Theta(1)$ such that 
    \begin{equation}{\label{eq-bound-beta-mathcal-H}}
        \frac{D^{-1}}{\aleph^2} \sigma_+(\mathbf P)^{\aleph} \leq \beta_{\mathcal H} \leq D \sigma_+(\mathbf P)^{\aleph} \,,
    \end{equation}
    In addition, if \eqref{eq-assum-upper-bound} holds, then there exists a constant $\delta>0$ such that $\sigma_+(\mathbf P)>1+\delta$. 
\end{lemma}

The proof of Lemma~\ref{lem-bound-beta-mathcal-H} is incorporated in Section~\ref{subsec:proof-lem-3.2} of the appendix. Our main result for the detecting algorithm for the contextual multi-layer SBMs can be summarized as follows:
\begin{proposition}{\label{main-prop-detection}}
    Suppose that \eqref{eq-assum-upper-bound} holds, and we choose  
    \begin{align}{\label{eq-condition-strong-detection}}
        \omega(1)=\aleph=o(\tfrac{\log n}{\log\log n})  \,.
    \end{align}
    Then we have
    \begin{align}{\label{eq-moment-control}}
        \mathbb E_{\Pb}\big[ f_{\mathcal H} \big]=\omega(1), \ \mathbb E_{\Qb}\big[ f_{\mathcal H}^2 \big]=1+o(1) \mbox{, and } \operatorname{Var}_{\Pb}\big[ f_{\mathcal H} \big]=o(1) \cdot \mathbb E_{\Pb}\big[ f_{\mathcal H} \big]^2 \,.
    \end{align}
\end{proposition}

Combining these variance bounds with Chebyshev’s inequality, we arrive at the following sufficient condition for the statistic $f_{\mathcal H}$ to achieve strong detection.
\begin{thm}{\label{MAIN-THM-detection}}
    Suppose that \eqref{eq-assum-upper-bound} holds and we choose $\aleph$ according to \eqref{eq-condition-strong-detection}. Then the testing error satisfies
    \begin{equation}{\label{eq-testing-error}}
        \Pb\big( f_{\mathcal H} \leq\tau \big) + \Qb\big( f_{\mathcal H} \geq\tau \big) = o(1) \,,
    \end{equation}
    where the threshold is chosen as
    \begin{equation*}
        \tau = c \cdot \mathbb E_{\Pb}\big[ f_{\mathcal H} \big]
    \end{equation*}
    for any fixed constant $0<c<1$. In particular, Algorithm~\ref{alg:detection-meta} described above achieves strong detection between $\Pb$ and $\Qb$.
\end{thm}

\begin{remark}
    From a computational perspective, the evaluation of each
    \begin{align*}
        \sum_{ S \subset \mathsf K_{n,p}: S \cong H } f_{S}
    \end{align*}
    in \eqref{eq-def-f-mathcal-H} by exhaustive search takes $n^{O(\aleph)}$ time which is super-polynomial when $\aleph=\omega(1)$. To resolve this computational issue, we can find a polynomial-time algorithm to compute an approximation $\widetilde{f}_{\mathcal H}$ for $f_{\mathcal H}$ using the strategy of color coding \cite{AYZ95, AR02, HS17, MWXY24, Li25+}. Similarly as in \cite{HS17, MWXY24, Li25+}, we can show that the statistic $\widetilde f_{\mathcal H}$ achieves strong detection under the same condition as in Theorem~\ref{MAIN-THM-detection}. Since this type of analysis is now standard in the color coding literature, we omit the details and refer readers to e.g., \cite{Li25+} for details. For brevity, analogous color-coding arguments in the recovery algorithm are also omitted.
\end{remark}

\subsection{The recovery statistic and theoretical guarantees}

Recall \eqref{eq-def-f-H,A}. Let $\aleph\in\mathbb N$ be a parameter that will be decided later. In addition, define $\mathcal J_{\star}=\mathcal J_{\star}(\aleph)$ to be the collection of unlabeled decorated paths $[H]$ such that $|V^{\mathsf a}(H)|=\aleph+1$, and define (recall \eqref{eq-def-V-ell-decorated} and that $\mathsf L(H)$ denotes the set of leaves of $H$)
\begin{equation}{\label{eq-def-mathcal-J}}
    \mathcal J=\mathcal J(\aleph):= \Big\{ [H] \in \mathcal J_{\star}: \mathsf{L}(H) \subset V^{\mathsf a}(H) \Big\} \,.
\end{equation}
See Figure~\ref{fig:decorated-path} for an example of a decorated path. It is clear that if we label the vertices in $V(H)$ in the order $v_1,\ldots,v_{m+1}$ such that $(v_i,v_{i+1}) \in E(H)$, then we have $v_1,v_{m+1} \in V^{\mathsf a}(H)$. Also, for all $[H] \in \mathcal J(\aleph)$ we have
\begin{equation}{\label{eq-basic-property-mathcal-J}}
    |E(H)|=\aleph+\tfrac{1}{2}|E_0(H)|,\ |V^{\mathsf b}(H)|=\tfrac{1}{2}|E_0(H)| \,.
\end{equation}
Recall \eqref{eq-def-Xi(H)}. Define 
\begin{equation}{\label{eq-def-Phi-i,j-mathcal-J}}
    \Phi_{u,v}^{\mathcal J} := \sum_{[H] \in \mathcal J} \frac{ \Xi(H) }{ n^{\frac{1}{2}\aleph-1} p^{\frac{1}{4}|E_0(H)|} \beta_{\mathcal J} } \sum_{ \substack{ S \subset \mathsf K_{n,p}: S \cong H \\ \mathsf L(S)=\{ a_u,a_v \} } } f_{S} \, 
\end{equation}
for each $u,v \in [n]$, where $a_u,a_v$ are vertices corresponding to $u,v$. Here 
\begin{equation}{\label{eq-def-beta-mathcal-J}}
    \beta_{\mathcal J} = \sum_{[H] \in \mathcal J} \frac{ \Xi(H)^2 }{ |\mathsf{Aut}(H)| } \,.
\end{equation}
Our proposed method of recovery in contextual multi-layer SBMs is as follows.
\begin{breakablealgorithm}{\label{alg:recovery-meta}}
\caption{Recovery in contextual multi-layer SBMs}
    \begin{algorithmic}[1]
    \STATE {\bf Input:} A rectangular matrix $\bm Y$ and $L$ adjacency matrices $\bm G_1,\ldots,\bm G_L$, a family $\mathcal J$ of non-isomorphic decorated graphs.
    \STATE For each pair $u,v \in [n]$, compute $\Phi_{u,v}^{\mathcal J}$ as in \eqref{eq-def-Phi-i,j-mathcal-J}. 
    \STATE Use the correlation preserving projection technique in \cite[Section~2.3]{HS17} to round the matrix $\Phi^{\mathcal J}$ into a positive definite matrix $\widehat{\Phi}$.
    \STATE {\bf Output:} $\widehat \Phi$.
    \end{algorithmic}
\end{breakablealgorithm}
Similarly as in Lemma~\ref{lem-bound-beta-mathcal-H}, we need the following bound on $\beta_{\mathcal J}$.

\begin{lemma}{\label{lem-bound-beta-mathcal-J}}
    Recall the definition of $\mathbf P$ in \eqref{eq-def-mathbf-P} and recall we let $\sigma_+(\mathbf P)$ be the largest eigenvalue of $\mathbf P$. Then there exists a constant $D=\Theta(1)$ such that 
    \begin{equation}{\label{eq-bound-beta-mathcal-J}}
        D^{-1} \sigma_+(\mathbf P)^{\aleph} \leq \beta_{\mathcal J}, \beta_{\mathcal J_{\star}} \leq D \sigma_+(\mathbf P)^{\aleph} \,,
    \end{equation}
    In addition, if \eqref{eq-assum-upper-bound} holds, then there exists a constant $\delta>0$ such that $\sigma_+(\mathbf P)>1+\delta$. 
\end{lemma}
The proof of Lemma~\ref{lem-bound-beta-mathcal-J} is incorporated in Section~\ref{subsec:proof-lem-3.6} of the appendix. 
Compared with Lemma~\ref{lem-bound-beta-mathcal-H} for $\beta_{\mathcal H}$, the lower bounds for
$\beta_{\mathcal J}$ and $\beta_{\mathcal J_\star}$ are smaller by a factor $\aleph^{-2}$.
This is due to the automorphism normalization: a decorated path has at most $2$ automorphisms,
whereas a length-$\aleph$ cycle has $\Theta(\aleph)$ automorphisms. See Appendix~\ref{sub:supp-proofs-sec-3} for more details.
The key to our argument is the following result, which shows that $\Phi_{u,v}^{\mathcal J}$ is positively correlated with $\bm x_u \bm x_v$ and has bounded variance. 
\begin{proposition}{\label{main-prop-recovery}}
    Suppose that \eqref{eq-assum-upper-bound} holds and we choose $\aleph$ such that 
    \begin{align}{\label{eq-condition-weak-recovery}}
        \aleph=\Theta_{\delta}(\log n), \quad (1+\delta)^{\aleph} \geq n^2 \,.
    \end{align}
    Then we have 
    \begin{align}
        \mathbb E_{\Pb}\Big[ \Phi^{\mathcal J}_{u,v} \cdot \bm x_u \bm x_v \Big] \geq \rho^2 \mbox{ and } \mathbb E_{\Pb}\Big[ \big( \Phi^{\mathcal J}_{u,v} \big)^2 \Big] \leq O(1) \,.  \label{eq-L2-estimation-error}
    \end{align}
\end{proposition}
Combining these variance bounds with Markov's inequality, we arrive at the following sufficient condition for Algorithm~\ref{alg:recovery-meta} to achieve weak recovery.
\begin{thm}{\label{MAIN-THM-recovery}}
    Suppose that \eqref{eq-assum-upper-bound} holds, and we choose $\aleph$ according to \eqref{eq-condition-weak-recovery}. Then we have (below we write $\widehat{\Phi}$ to be the output of Algorithm~\ref{alg:recovery-meta})
    \begin{align*}
        \mathbb E_\Pb\left[ \frac{ |\langle \widehat \Phi,\bm x \bm x^{\top} \rangle| }{ \| \widehat \Phi \|_{\Fop} \| \bm x \bm x^{\top} \|_{\Fop} } \right] \geq \Omega(1) \,.
    \end{align*}
    In particular, Algorithm~\ref{alg:recovery-meta} described above achieves weak recovery.
\end{thm}

\begin{remark}
   Note that our estimator naturally satisfies $\widehat{\Phi}\succeq 0$, according to \cite[Lemma~3.5]{HS17} and \cite[Section~5]{LS23}, by considering $\bm w \sim \mathcal N(0,\widehat{\Phi})$ and let $\widehat{\bm x}$ be coordinate-wise signs of $\bm w$, we can get an estimator $\widehat{\bm x} \in \{ -1,+1 \}^n$ such that  
    \begin{align*}
        \mathbb E_{\Pb}\left[ \frac{ |\langle \bm x,\widehat{\bm x} \rangle| }{ \|\bm x\|\| \widehat{\bm x} \| } \right] \geq \Omega(1) \,,
    \end{align*}
    given \eqref{eq-assum-upper-bound} holds and we choose $\aleph$ according to \eqref{eq-condition-weak-recovery}.
\end{remark}

\section{Statistical analysis of the subgraph counts}{\label{sec:stat-analysis}}

\subsection{Proof of Proposition~\ref{main-prop-detection}}{\label{subsec:proof-prop-3.4}}

\begin{lemma}{\label{lem-mean-var-f-H-part-1}}
    Suppose that \eqref{eq-assum-upper-bound} holds, and we choose $\aleph$ according to \eqref{eq-condition-strong-detection}. Then
    \begin{align}
        &\mathbb E_{\Qb}[ f_{\mathcal H} ] =0 \,, \label{eq-mean-Qb-f-H} \\
        &\mathbb E_{\Pb}[ f_{\mathcal H} ] =[1+o(1)] \sqrt{\beta_{\mathcal H}} = \omega(1) \,, \label{eq-mean-pb-f-H} \\
        &\mathbb E_{\Qb}[ f_{\mathcal H}^2 ] = 1+o(1) \,. \label{eq-var-Qb-f-H}
    \end{align}
\end{lemma}
\begin{proof}
    Recall~\eqref{eq-def-overline-G-ell-i,j} and~\eqref{eq-def-f-mathcal-H}. Since $\mathbb E_{\Qb}[ f_S(\bm X,\bm Y) ]=0$ for all $S \subset \mathsf K_n$ and $E(S) \neq \emptyset$, we have $\mathbb E_{\Qb}[f_{\mathcal H}]=0$ by linearity. In addition, recall \eqref{eq-def-contextual-multilayer-SBM}. Conditioned on $\bm x,\bm u$ and $\bm x_1,\ldots,\bm x_L$, 
    \begin{equation*}
        \big\{ \bm Y(i,k):i \in [n],k \in [p] \big\} \bigcup \Big( \cup_{1 \leq \ell \leq L} \big\{ \overline{\bm G}_{\ell}(i,j): (i,j) \in \operatorname{U}_n \big\} \Big)
    \end{equation*}
    is a collection of conditionally independent random variables, with
    \begin{align*}
        \mathbb E_{\Pb}\big[ \bm Y(i,k) \mid \bm x,\bm u,\bm x_{1},\ldots,\bm x_L \big]&= \tfrac{ \sqrt{\mu} }{ \sqrt{n} } \bm x(i) \bm u(k),\\ \mathbb E_{\Pb}\big[ \overline{\bm G}_{\ell}(i,j) \mid \bm x,\bm u,\bm x_{1},\ldots,\bm x_L \big]&=\tfrac{1}{\sqrt{\lambda_\ell/n}}\cdot \tfrac{\epsilon_\ell \lambda_\ell}{n}\bm x_{\ell}(i) \bm x_{\ell}(j) =  \tfrac{ \sqrt{\epsilon_{\ell}^2\lambda_{\ell}} }{ \sqrt{n} } \bm x_{\ell}(i) \bm x_{\ell}(j) \,.
    \end{align*}
    Thus, we have (recall \eqref{eq-def-f-H,A})
    \begin{align}
        \mathbb E_{\Pb}\big[ f_S \big] &= \mathbb E_{\bm x,\bm u,\bm x_{1},\ldots,\bm x_L}\mathbb E_{\Pb}\big[ f_S \mid \bm x,\bm u,\bm x_1,\ldots,\bm x_L \big] \nonumber \\
        &= \mathbb E_{\bm x,\bm u,\bm x_{1},\ldots,\bm x_L}\Bigg\{ \prod_{(a_i,b_k)\in E_0(S)} \tfrac{ \sqrt{\mu} }{ \sqrt{n} } \bm x(i) \bm u(k) \prod_{1 \leq \ell \leq L} \prod_{(a_i,a_j)\in E_{\ell}(S)} \tfrac{ \sqrt{\epsilon_{\ell}^2\lambda_{\ell}} }{ \sqrt{n} } \bm x_{\ell}(i) \bm x_{\ell}(j) \Bigg\} \nonumber \\
        &= \frac{ \mu^{\frac{1}{2}|E_0(S)|} \prod_{1 \leq \ell \leq L} (\epsilon_\ell^2\lambda_\ell)^{\frac{1}{2}|E_\ell(S)|} }{ n^{\frac{1}{2}|E(S)|} } \mathbb E\Bigg\{ \prod_{(a_i,b_k)\in E_0(S)} \bm x(i) \bm u(k) \prod_{1 \leq \ell \leq L} \prod_{(a_i,a_j)\in E_{\ell}(S)} \bm x_{\ell}(i) \bm x_{\ell}(j) \Bigg\} \nonumber \\
        &\overset{\mathrm{(a)}}{=}  \frac{ \mu^{\frac{1}{2}|E_0(S)|} \prod_{1 \leq \ell \leq L} (\epsilon_\ell^2\lambda_\ell)^{\frac{1}{2}|E_\ell(S)|} }{ n^{\frac{1}{2}\aleph+\frac{1}{4}|E_0(S)|} } \cdot \rho^{|\mathsf{dif}_0(S)|+2|\mathsf{dif}(S)|} \nonumber \\
        &\overset{\mathrm{(b)}}{=} \frac{ (\mu^2/\gamma)^{\frac{1}{4}|E_0(S)|} \prod_{1 \leq \ell \leq L} (\epsilon_\ell^2\lambda_\ell)^{\frac{1}{2}|E_\ell(S)|} \rho^{|\mathsf{dif}_0(S)|+2|\mathsf{dif}(S)|} }{ n^{\frac{1}{2}\aleph} p^{\frac{1}{4}|E_0(S)|} } \overset{\mathrm{(c)}}{=} \frac{ \Xi(S) }{ n^{\frac{1}{2}\aleph} p^{\frac{1}{4}|E_0(S)|} } \,.  \label{eq-cal-expectation}
    \end{align}
    Here $\mathrm{(b)}$ follows from $\gamma = n/p$; $\mathrm{(c)}$ is due to~\eqref{eq-def-Xi(H)};
    $\mathrm{(a)}$ follows from $\aleph = |E(S)| - \tfrac{1}{2}|E_0(S)|$ in \eqref{eq-basic-property-mathcal-H} and (below we write $a_i \sim e$ if a vertex $a_i$ is incident to an edge $e$)
    \begin{align*}
        &\mathbb E\bigg\{ \prod_{(a_i,b_k)\in E_0(S)} \bm x(i) \bm u(k) \prod_{1 \leq \ell \leq L} \prod_{(a_i,a_j)\in E_{\ell}(S)} \bm x_{\ell}(i) \bm x_{\ell}(j) \bigg\} \\
        =\ & \mathbb E\bigg\{ \prod_{b_k \in V^{\mathsf b}(S)} \bm u(k)^2 \prod_{ a_i \in V^{\mathsf a}(S) } \bm x(i)^{ \Lambda_0(i) } \bm x_1(i)^{\Lambda_1(i)} \ldots \bm x_L(i)^{\Lambda_L(i)}  \bigg\} \\
        =\ & \prod_{ a_i \in V^{\mathsf a}(S) } \mathbb E\Big[ \bm x(i)^{\Lambda_0(i)}\bm x_1(i)^{\Lambda_1(i)} \ldots \bm x_L(i)^{\Lambda_L(i)} \Big] \,, \mbox{ where } \Lambda_\ell(i)= \#\big\{ e \in E_\ell(S): a_i \sim e \big\} \,,    \end{align*}
    so it is clear that
    \begin{itemize}
        \item If $a_i \in \mathsf{dif}_0(S)$, then $\Lambda_0(i)=1$ and $\Lambda_\ell(i)=1$ for some $1 \leq \ell\leq L$ (with all others being $0$), and thus this contributes a factor of $\rho$ to the product (recall \eqref{eq-def-x-ell});
        \item If $a_i \in \mathsf{dif}(S)$, then $\Lambda_0(i)=0$ and $\Lambda_\ell(i),\Lambda_{\ell'}(i)=1$ for exactly two $1 \leq \ell<\ell'\leq L$, and thus this contributes a factor of $\rho^2$ to the product (recall \eqref{eq-def-x-ell});
        \item For all other $a_i$, $\Lambda_\ell(i)=2$ for some $1 \leq \ell \leq L$ (with all others being $0$), and thus this contributes a factor of $1$ to the product.
    \end{itemize}
    Thus, we have
    \begin{align*}
        \mathbb E_{\Pb}[f_{\mathcal H}] &\overset{\eqref{eq-def-f-mathcal-H}}{=} \frac{1}{ \sqrt{\beta_{\mathcal H}} } \sum_{ [H] \in \mathcal H } \frac{ \Xi(H) }{ n^{\aleph/2} p^{|E_0(H)|/4} } \sum_{ S \subset \mathsf{K}_{n,p}, S \cong H } \Xi(S) n^{ -\frac{1}{2}\aleph } p^{-\frac{1}{4}|E_0(S)|} \\
        &= \frac{1}{ n^{\aleph} p^{|E_0(H)|/2} \sqrt{\beta_{\mathcal H}} } \sum_{[H] \in \mathcal H} \Xi(H)^2 \cdot \#\big\{ S \cong \mathsf K_{n,p}: S \cong H \big\} \\
        & = [1+o(1)] \cdot \frac{1}{ \sqrt{\beta_{\mathcal H}} }  \sum_{[H] \in \mathcal H} \frac{ \Xi(H)^2 }{ |\mathsf{Aut}(H)| } \overset{\eqref{eq-def-beta-mathcal-H}}{=} [1+o(1)]\sqrt{\beta_{\mathcal H}} \,,
    \end{align*}
    where the third equality follows from
    \begin{align}
        \nonumber \#\big\{ S \cong \mathsf K_{n,p}: S \cong H \big\} &=\frac{\prod_{i=0}^{|V^{\mathsf a}(H)|-1}(n-i)\prod_{j=0}^{|V^{\mathsf b}(H)|-1}(p-j)}{|\mathsf{Aut}(H)|}\\\label{eq-enumerate-H-in-K-n,p}&= \frac{ (1+o(1)) n^{|V^{\mathsf a}(H)|} p^{|V^{\mathsf b}(H)|} }{ |\mathsf{Aut}(H)| } \overset{\eqref{eq-basic-property-mathcal-H}}{=} \frac{ (1+o(1))n^{\aleph} p^{|E_0(H)|/2} }{ |\mathsf{Aut}(H)| } \,.
    \end{align}
    Also, from Lemma~\ref{lem-bound-beta-mathcal-H} and our choice of $\ell$ in \eqref{eq-condition-strong-detection} we see that $\sqrt{\beta_{\mathcal H}}=\omega(1)$. 
    Consequently, we have $\mathbb E_{\Pb}[f_{\mathcal H}]  = \omega(1)$.
    Finally, note that
    \begin{align}{\label{eq-standard-orthogonal}}
        \mathbb E_{\Qb}\big[ f_{S}^2 \big]=1+o(1) \mbox{ and } \mathbb E_{\Qb}\big[ f_{S} f_{K} \big] = 0 \mbox{ for } S \neq K \,,
    \end{align}
    we have
    \begin{align*}
        \mathbb E_{\Qb}[f_{\mathcal H}^2] &\overset{\eqref{eq-def-f-mathcal-H}}{=} [1+o(1)] \sum_{ [H],[I] \in \mathcal H } \frac{ \Xi(H)\Xi(I) }{ n^{\aleph} p^{\frac{1}{4}(|E_0(H)|+|E_0(I)|)} \beta_{\mathcal H} } \sum_{ \substack{ S,K \subset \mathsf K_{n,p} \\ S \cong H, K \cong I } } \mathbf 1_{ \{ S=K \} } \\
        &~\ = \sum_{ [H] \in \mathcal H } \frac{ \Xi(H)^2 \#\{ S \cong \mathsf K_{n,p}: S \cong H \} }{ n^{\aleph} p^{\frac{1}{2}|E_0(H)|} \beta_{\mathcal H} }   = \frac{ 1+o(1) }{ \beta_{\mathcal H} } \sum_{ [H] \in \mathcal H } \frac{ \Xi(H)^2 }{ |\mathsf{Aut}(H)| } \overset{\eqref{eq-def-beta-mathcal-H}}{=} 1+o(1) \,. \qedhere
    \end{align*}
\end{proof}

Now we bound the variance of $f_{\mathcal H}$ under the alternative hypothesis $\Pb$, as incorporated in the following lemma.
\begin{lemma}{\label{lem-mean-var-f-H-part-2}}
    Assume that \eqref{eq-assum-upper-bound} holds, and we choose $\aleph$ according to \eqref{eq-condition-strong-detection}. Then we have
    \begin{equation}{\label{eq-var-Pb-f-H}}
        \frac{ \operatorname{Var}_{\Pb}[f_{\mathcal H}] }{ \mathbb E_{\Pb}[f_{\mathcal H}]^2 } = o(1) \,.
    \end{equation}
\end{lemma}

Assuming Lemma~\ref{lem-mean-var-f-H-part-2}, we can complete the proof of Proposition~\ref{main-prop-detection}.
\begin{proof}[Proof of Proposition~\ref{main-prop-detection} assuming Lemma~\ref{lem-mean-var-f-H-part-2}]
    It suffices to note that combining Lemmas~\ref{lem-mean-var-f-H-part-1} and \ref{lem-mean-var-f-H-part-2} yields Proposition~\ref{main-prop-detection}.
\end{proof}

The remaining part of this subsection is devoted to the proof of Lemma~\ref{lem-mean-var-f-H-part-2}. Note that
\begin{align}
    \frac{ \operatorname{Var}_{\Pb}[f_{\mathcal H}] }{ \mathbb E_{\Pb}[f_{\mathcal H}]^2 } \overset{\eqref{eq-def-f-mathcal-H},\text{Lemma~\ref{lem-mean-var-f-H-part-1}}}{=}\ & [1+o(1)] \sum_{[H],[I] \in \mathcal H} \frac{ \Xi(H) \Xi(I) }{ n^{\aleph} p^{ \frac{1}{4}(|E_0(H)|+|E_0(I)|) } \beta_{\mathcal H}^2 }  \sum_{ \substack{ S,K \subset \mathsf K_n \\ S \cong H, K \cong I } } \operatorname{Cov}_{\Pb}\big( f_S, f_K \big) \,,  \label{eq-var-Pb-f-H-relax-1}
\end{align}
where $\operatorname{Cov}_{\Pb}\big( f_S, f_K \big):=\E_{\Pb}[f_Sf_K] - \E_{\P}[f_S]\E_\P[f_K]$.
The first step of our proof is to show the following bound on $\operatorname{Cov}_{\Pb}(f_S,f_K)$. Define 
\begin{equation}{\label{eq-def-Omega(S)}}
    \Omega(S)=n^{|E(S)|-\frac{1}{2}|E_0(S)|} p^{\frac{1}{2}|E_0(S)|} \,.
\end{equation}
For $[H],[I]\in \mathcal H(\aleph)$, we have $n^{\aleph} p^{(|E_0(H)|+|E_0(I)|)/4} = \Omega(H)^{1/2}\Omega(I)^{1/2}$ since $\aleph = |E(H)| - \tfrac{1}{2}|E_0(H)|$ for all $H\in \mathcal H(\aleph)$.
We also define
\begin{align}
    \mathtt P(S,K):=\ & \big( \tfrac{\mu^2}{\gamma} \big)^{ \frac{1}{4} |E_{0}(S) \triangle E_{0}(K)| } \prod_{1 \leq \ell \leq L} (\epsilon^2_{\ell} \lambda_{\ell})^{ \frac{1}{2}|E_{\ell}(S) \triangle E_{\ell}(K)| } \nonumber  \\
    & \cdot \rho^{|\mathsf{dif}_0(S) \setminus V(K)|+ |\mathsf{dif}_0(K) \setminus V(S)| +2|\mathsf{dif}(S) \setminus V(K)|+ 2|\mathsf{dif}(K) \setminus V(S)| }  \,.  \label{eq-def-mathtt-P}
\end{align}

\begin{lemma}{\label{lem-est-cov-f-S-f-K}}
    Suppose that $[S],[K] \in \mathcal H$. We have 
    \begin{equation}{\label{eq-cov-equal-0-non-intersect}}
        \operatorname{Cov}_{\Pb}\big( f_S,f_K \big)=0 \mbox{ if } V(S) \cap V(K) = \emptyset \,.
    \end{equation}
    Recall that $\mathsf L (H)$ denotes the set of leaves in $H$.
    For general decorated paths or decorated cycles $S,K$, there exists a constant $C>0$ such that
    \begin{align}{\label{eq-est-general-cov-f-S-f-K}}
        \mathbb E_{\Pb}\big[ f_Sf_K \big] \leq \frac{ C^{|\mathsf{L}(S \Cap K)|+|\mathsf{L}(S \cap K)|+1} \mathtt P(S,K) \Omega(S \Cap K) }{ \Omega(S)^{\frac{1}{2}} \Omega(K)^{\frac{1}{2}} }  \,.
    \end{align} 
\end{lemma}
The proof of Lemma~\ref{lem-est-cov-f-S-f-K} is incorporated in Section~\ref{subsec:proof-lem-4.3} of the appendix. By \eqref{eq-cal-expectation}, we have $\E_\P[f_S]\ge 0$ for any $[S]\in \mathcal H$, and thus $\operatorname{Cov}_{\Pb}\big( f_S,f_K \big)\le \mathbb E_{\Pb}\big[ f_Sf_K \big]$.
Using Lemma~\ref{lem-est-cov-f-S-f-K}, we see that the right-hand side of \eqref{eq-var-Pb-f-H-relax-1} can be written as (recall \eqref{eq-def-Omega(S)})
\begin{align}
    [1+o(1)] \sum_{ \substack{ S,K \subset \mathsf K_{n,p}: [S],[K] \in \mathcal H \\ V(S) \cap V(K) \neq \emptyset } } \frac{ C^{|\mathsf{L}(S \Cap K)|+|\mathsf{L}(S \cap K)|+1} \Xi(S) \Xi(K) \mathtt P(S,K) \Omega(S \Cap K) }{ \Omega(S) \Omega(K) \beta_{\mathcal H}^2 }  \,.  \label{eq-var-Pb-f-H-relax-2}
\end{align}
Now we split \eqref{eq-var-Pb-f-H-relax-2} into two parts, the first part counts the contribution from $S=K$ (in this case we have $S \Cap K=S$ and $\mathsf L(S\Cap K) = \mathsf L(S\cap K) = \emptyset$): 
\begin{align}
    \sum_{ S \in \mathsf K_{n,p}: [S] \in \mathcal H } \frac{ C \Xi(S)^2 \cdot \mathtt P(S,S) }{ \Omega(S) \beta_{\mathcal H}^2 }  \,.  \label{eq-var-Pb-f-H-relax-3-Part-1}
\end{align}
The second part counts the contribution from $V(S) \cap V(K) \neq \emptyset$ and $S \neq K$:
\begin{align}
    \sum_{ \substack{ S,K \subset \mathsf K_{n,p}: [S],[K] \in \mathcal H \\ V(S) \cap V(K) \neq \emptyset, S \neq K } } \frac{ C^{|\mathsf{L}(S \Cap K)|+|\mathsf{L}(S \cap K)|+1} \Xi(S) \Xi(K) \mathtt P(S,K) \Omega(S\Cap K) }{ \beta_{\mathcal H}^2 \Omega(S) \Omega(K) }  \,,  \label{eq-var-Pb-f-H-relax-3-Part-2}
\end{align}
We now bound \eqref{eq-var-Pb-f-H-relax-3-Part-1} and \eqref{eq-var-Pb-f-H-relax-3-Part-2} separately via the following lemma.

\begin{lemma}{\label{lem-detection-most-technical}}
    Suppose that \eqref{eq-assum-upper-bound} holds and we choose $\aleph$ according to \eqref{eq-condition-strong-detection}. Then we have
    \begin{align}
        & \eqref{eq-var-Pb-f-H-relax-3-Part-1} \leq \frac{1+o(1)}{\beta_{\mathcal H}} = o(1) \,.  \label{eq-bound-var-Pb-f-H-relax-3-Part-1} \\
        & \eqref{eq-var-Pb-f-H-relax-3-Part-2} \leq n^{-1+o(1)} = o(1) \,.  \label{eq-bound-var-Pb-f-H-relax-3-Part-2}
    \end{align}
\end{lemma}
The proof of Lemma~\ref{lem-detection-most-technical} is postponed to Section~\ref{subsec:proof-lem-4.4} in the appendix. 
Clearly, plugging \eqref{eq-bound-var-Pb-f-H-relax-3-Part-1} and \eqref{eq-bound-var-Pb-f-H-relax-3-Part-2} into \eqref{eq-var-Pb-f-H-relax-2} yields Lemma~\ref{lem-mean-var-f-H-part-2}.

\subsection{Proof of Theorem~\ref{MAIN-THM-detection}}{\label{subsec:proof-thm-3.5}}

With Proposition~\ref{main-prop-detection}, we are ready to prove Theorem~\ref{MAIN-THM-detection}.
\begin{proof}[Proof of Theorem~\ref{MAIN-THM-detection}]
    Under condition \eqref{eq-condition-strong-detection}, from Proposition~\ref{main-prop-detection} we get that
    \begin{align*}
        \mathbb E_{\Qb}[f_{\mathcal H}]=0, \ \mathbb E_{\Qb}[f_{\mathcal H}^2] = 1+o(1), \ \mathbb E_{\Pb}[f_{\mathcal H}] = \omega(1), \ \operatorname{Var}_{\Pb}[f_{\mathcal H}] = o(1) \cdot \mathbb E_{\Pb}[f_{\mathcal H}]^2 \,.
    \end{align*}
    Thus, for any constant $0<c<1$ we obtain
    \begin{align*}
        &\Pb\Big( f_{\mathcal H} \leq c\cdot \mathbb E_{\Pb}[f_{\mathcal H}] \Big) \leq \frac{ \operatorname{Var}_{\Pb}[f_{\mathcal H}] }{ (1-c)^2 \mathbb E_{\Pb}[f_{\mathcal H}]^2 } = o(1) \,, \\
        &\Qb\Big( f_{\mathcal H} \geq c\cdot \mathbb E_{\Pb}[f_{\mathcal H}] \Big) \leq \frac{ \operatorname{Var}_{\Qb}[f_{\mathcal H}] }{ c^2 \mathbb E_{\Pb}[f_{\mathcal H}]^2 } = o(1) \,.  \qedhere
    \end{align*}
\end{proof}

\subsection{Proof of Proposition~\ref{main-prop-recovery}}{\label{subsec:proof-prop-3.8}}

\begin{lemma}{\label{lem-conditional-exp-given-endpoints}}
    For all $[S] \in \mathcal J$ and $\mathsf L(S)=\{ a_u,a_v \}$, we have 
    \begin{align*}
        \mathbb E_{\Pb}\big[ f_S \cdot \bm x_u \bm x_v \big] = \frac{ \Xi(S) }{ n^{\frac{1}{2}\aleph} p^{\frac{1}{4}|E_0(S)|} } \cdot \rho^{\mathbf 1_{u \in V_0(S)}+\mathbf 1_{v \in V_0(S)}} \,.
    \end{align*}
\end{lemma}
\begin{proof}
    Denote $\pi_{u,v}$ as the law of $(\bm x,\bm x_1,\ldots,\bm x_L)$ given $\bm x_u,\bm x_v$. We then have
    \begin{align*}
        \mathbb E_{\Pb}\big[ f_S \cdot \bm x_u \bm x_v \big] = \mathbb E_{(\bm x,\bm x_1,\ldots,\bm x_L)\sim\pi_{u,v}}\Big[ \bm x_u \bm x_v \cdot \mathbb E_\mathbb P\big[ f_S \mid \bm x,\bm x_1,\ldots,\bm x_L \big] \Big] \,.
    \end{align*}
    Recall \eqref{eq-def-f-H,A}, we have
    \begin{align*}
        &\mathbb E_\mathbb P\big[ f_S \mid \bm x,\bm x_1,\dots,\bm x_L \big] = \mathbb E_{\bm u}\Bigg\{ \prod_{(a_i,b_k)\in E_0(S)} \tfrac{ \sqrt{\mu} }{ \sqrt{n} } \bm x(i) \bm u(k) \prod_{1 \leq \ell \leq L} \prod_{(a_i,a_j)\in E_{\ell}(S)} \tfrac{ \sqrt{\epsilon_{\ell}^2\lambda_{\ell}} }{ \sqrt{n} } \bm x_{\ell}(i) \bm x_{\ell}(j) \bigg\}  \\
        =\ & \frac{ \mu^{\frac{1}{2}|E_0(S)|} \prod_{1 \leq \ell \leq L} (\epsilon_\ell^2\lambda_\ell)^{\frac{1}{2}|E_\ell(S)|} }{ n^{\frac{1}{2}|E(S)|} } \mathbb E_{\bm u}\Bigg\{ \prod_{(a_i,b_k)\in E_0(S)} \bm x(i) \bm u(k) \prod_{1 \leq \ell \leq L} \prod_{(a_i,a_j)\in E_{\ell}(S)} \bm x_{\ell}(i) \bm x_{\ell}(j) \bigg\}  \\
        =\ & \frac{ (\mu^2/\gamma)^{\frac{1}{4}|E_0(S)|} \prod_{1 \leq \ell \leq L} (\epsilon_\ell^2\lambda_\ell)^{\frac{1}{2}|E_\ell(S)|} }{ n^{\frac{1}{2}\aleph} p^{\frac{1}{4}|E_0(S)|} } \mathbb E_{\bm u}\Bigg\{ \prod_{(a_i,b_k)\in E_0(S)} \bm x(i) \bm u(k) \prod_{1 \leq \ell \leq L} \prod_{(a_i,a_j)\in E_{\ell}(S)} \bm x_{\ell}(i) \bm x_{\ell}(j) \bigg\} \,,
    \end{align*}
    where the last equality follows from $n=\gamma p$ and \eqref{eq-basic-property-mathcal-J}. In addition, recall \eqref{eq-def-V-ell-decorated} and recall that for all $[S] \in \mathcal J$ with $\mathsf L(S)=\{ u,v \}$, we have $\{ u,v \} \in V_{0} \cap V^{\mathsf a}(S)$. Thus, similarly as in \eqref{eq-cal-expectation} we have
    \begin{align*}
        & \prod_{(a_i,b_k)\in E_0(S)} \bm x(i) \bm u(k) \prod_{1 \leq \ell \leq L} \prod_{(a_i,a_j)\in E_{\ell}(S)} \bm x_{\ell}(i) \bm x_{\ell}(j) \\
        =\ & \prod_{b_k \in V^{\mathsf b}(S)} \bm u(k)^2 \prod_{ a_i \in V^{\mathsf a}(S) } \bm x(i)^{ \Lambda_0(i) } \bm x_1(i)^{\Lambda_1(i)} \ldots \bm x_L(i)^{\Lambda_L(i)} \,,
    \end{align*}
    where $(\Lambda_0(i),\Lambda_1(i),\ldots,\Lambda_L(i))$ is defined similarly as in \eqref{eq-cal-expectation}. This yields that
    \begin{align*}
        & \mathbb E_{(\bm x,\bm x_1,\ldots,\bm x_L) \sim \pi_{u,v}}\left\{ \mathbb E_{\bm u} \left[ \bm x_u \bm x_v \cdot \prod_{(a_i,b_k)\in E_0(S)} \bm x(i) \bm u(k) \prod_{1 \leq \ell \leq L} \prod_{(a_i,a_j)\in E_{\ell}(S)} \bm x_{\ell}(i) \bm x_{\ell}(j) \right] \right\} \\
        =\ & \rho^{|\mathsf{dif}_0(S)|+2|\mathsf{dif}(S)|+\mathbf 1_{u \in V_0(S)} + \mathbf 1_{v \in V_0(S)}}  \,, 
    \end{align*}
    leading to the desired result (recall \eqref{eq-def-Xi(H)}).
\end{proof}

Based on Lemma~\ref{lem-conditional-exp-given-endpoints}, we see that 
\begin{align}
    &\mathbb E_{\Pb}\big[ \Phi^{\mathcal J}_{u,v} \cdot \bm x_u \bm x_v \big] \overset{\eqref{eq-def-Phi-i,j-mathcal-J}}{=} \sum_{[H] \in \mathcal J} \frac{ \Xi(H) }{ n^{\frac{1}{2}\aleph-1} p^{\frac{1}{4}|E_0(H)|} \beta_{\mathcal J} } \sum_{ \substack{ S \subset \mathsf K_n: S \cong H \\ \mathsf L(S)=\{ u,v \} } } \frac{ \Xi(H) \rho^{\mathbf 1_{u \in V_0(S)} + \mathbf 1_{v \in V_0(S)}} }{ n^{\frac{1}{2}\aleph} p^{\frac{1}{4}|E_0(H)|} \beta_{\mathcal J} }  \nonumber \\
    \geq\ & \rho^2 \sum_{[H] \in \mathcal J} \frac{ \Xi(H)^2 }{ n^{\aleph-1} p^{\frac{1}{2}|E_0(H)|} \beta_{\mathcal J} } \cdot \#\big\{ S \subset \mathsf K_n: S \cong H, \mathsf L(S)=\{ u,v \} \big\} \nonumber \\
    =\ & [1+o(1)] \cdot \rho^2 \sum_{[H] \in \mathcal J} \frac{ \Xi(H)^2 }{ n^{\aleph-1} p^{\frac{1}{2}|E_0(H)|} \beta_{\mathcal J} } \cdot \frac{ n^{\aleph-1} p^{\frac{1}{2}|E_0(H)|} }{ |\mathsf{Aut}(H)| }  \overset{\eqref{eq-def-beta-mathcal-J}}{=} (1+o(1)) \rho^2 \,.  \label{eq-conditional-mean-Phi-i,j}
\end{align}
We now bound the variance of $\Phi_{i,j}^{\mathcal J}$ via the following lemma.
\begin{lemma}{\label{lem-conditional-var-Phi-i,j}}
    Suppose that \eqref{eq-assum-upper-bound} holds and we choose $\aleph$ according to \eqref{eq-condition-weak-recovery}. Then
    \begin{equation}{\label{eq-conditional-var-Phi-i,j}}
        \mathbb E_{\Pb}\Big[ \big( \Phi_{u,v}^{\mathcal J} \big)^2 \Big]  \leq O(1) \mbox{ for all } u \neq v, u,v \in [n]  \,.
    \end{equation}
\end{lemma}
Based on Lemma~\ref{lem-conditional-var-Phi-i,j}, we can now complete the proof of Proposition~\ref{main-prop-recovery}.
\begin{proof}[Proof of Proposition~\ref{main-prop-recovery} assuming Lemma~\ref{lem-conditional-var-Phi-i,j} holds]
    Note that \eqref{eq-conditional-mean-Phi-i,j} immediately implies that $\mathbb E_{\Pb}[ \Phi^{\mathcal J}_{u,v} \cdot \bm x_u \bm x_v] \geq \rho^2$. Combined with Lemma~\ref{lem-conditional-var-Phi-i,j}, we have shown Proposition~\ref{main-prop-recovery}.
\end{proof}
The rest part of this subsection is devoted to the proof of Lemma~\ref{lem-conditional-var-Phi-i,j}. Recall $\Phi_{u,v}^{\mathcal J}$ defined in \eqref{eq-def-Phi-i,j-mathcal-J} and $\Omega(S)$ defined in \eqref{eq-def-Omega(S)}, we have
\begin{align}
    \mathbb E_{\Pb}\Big[ \big( \Phi_{u,v}^{\mathcal J} \big)^2 \Big] = \sum_{ \substack{ [S],[K] \in \mathcal J \\ \mathsf L(S)=\mathsf L(K)=\{ u,v \} } } \frac{ n^2 \Xi(S) \Xi(K) }{ \Omega(S)^{\frac{1}{2}} \Omega(K)^{\frac{1}{2}} \beta_{\mathcal J}^2 } \mathbb E_{\Pb}\big[ f_S f_K \big] \,.  \label{eq-var-Phi-i,j-relax-1}
\end{align}
In addition, by plugging Lemma~\ref{lem-est-cov-f-S-f-K} in \eqref{eq-var-Phi-i,j-relax-1}, we get that
\begin{align}
    \eqref{eq-var-Phi-i,j-relax-1} &\leq \sum_{ \substack{ [S],[K] \in \mathcal J \\ \mathsf L(S)=\mathsf L(K)=\{ u,v \} } } \frac{ n^2 \Xi(S) \Xi(K) }{ \Omega(S)^{\frac{1}{2}} \Omega(K)^{\frac{1}{2}} \beta_{\mathcal J}^2 } \cdot \frac{ C^{|\mathsf{L}(S \Cap K)|+|\mathsf{L}(S \cap K)|+1} \mathtt P(S,K) \Omega(S \Cap K) }{ \Omega(S)^{\frac{1}{2}} \Omega(K)^{\frac{1}{2}} } \nonumber \\
    &= \sum_{ \substack{ [S],[K] \in \mathcal J \\ \mathsf L(S)=\mathsf L(K)=\{ u,v \} } } \frac{ n^2 C^{|\mathsf{L}(S \Cap K)|+|\mathsf{L}(S \cap K)|+1} \Xi(S) \Xi(K) \mathtt P(S,K) \Omega(S \Cap K) }{ \Omega(S) \Omega(K) \beta_{\mathcal J}^2 } \,.  \label{eq-var-Phi-i,j-relax-2}
\end{align}
We first show the main contribution of \eqref{eq-var-Phi-i,j-relax-2} comes from those $(S,K)$ such that $E(S) \cap E(K)=E(S \Cap K)$, as incorporated in the following lemma.
\begin{lemma}{\label{lem-reduce-to-good-S,K}}
    Assume that \eqref{eq-assum-upper-bound} holds, we have
    \begin{align}
        \eqref{eq-var-Phi-i,j-relax-2} \leq [1+o(1)] \sum_{ \substack{ [S],[K] \in \mathcal J \\ \mathsf L(S)=\mathsf L(K)=\{ u,v \} \\ E(S) \cap E(K) = E(S \Cap K) } } \frac{ n^2 C^{2|\mathsf{L}(S \cap K)|+1} \Xi(S) \Xi(K) \mathtt P(S,K) \Omega(S \Cap K) }{ \Omega(S) \Omega(K) \beta_{\mathcal J}^2 } \,.  \label{eq-var-Phi-i,j-relax-3}
    \end{align}
\end{lemma}
\begin{proof}
    For notational convenience, define
    \begin{align}\label{eq-def-of-mathcalF-SK}
        \mathcal F(S,K)= \frac{ C^{|\mathsf{L}(S \cap K)|+|\mathsf{L}(S \Cap K)|+1} \Xi(S) \Xi(K) \mathtt P(S,K) \Omega(S \Cap K) }{ \Omega(S) \Omega(K) \beta_{\mathcal J}^2 } \,.
    \end{align}
    Fix an arbitrary decorated path $S$. Note that given $E(K),V(K)$ and the decorations $\chi_K$ restricted in $E(K) \setminus E(S)$, the decorated graph $K$ is uniquely determined by the decorations $\chi_K$ on $E(K) \cap E(S)$. In addition, for the edge $e \in E_0(K) \cap E_0(S)$, we must have $\chi_K(e)=0$ by Definition~\ref{def-decorated-graph}. In conclusion, when given $V(K),E(K)$, and $\zeta_K$ (which is defined below), we have $\varphi:K\longrightarrow ( V(K),E(K), \zeta_K,\eta_K )$ with
    \begin{align*}
        \zeta_K:E^{\mathsf a}(K) \setminus E^{\mathsf a}(S) \to \{ 1,\ldots,L \}, \ \eta_K: E^{\mathsf a}(K) \cap E^{\mathsf a}(S) \to \{ 1,\ldots,L \}
    \end{align*}
    is a bijection. Denote $\eta_S$ as the restriction of $\chi_S$ in $E^{\mathsf a}(K) \cap E^{\mathsf a}(S)$. We first show that for each fixed vertex set $V(K)=\mathsf V$, edge set $E(K)=\mathsf E$, and partial decorations $\zeta_K=\zeta$, we have
    \begin{align}
        \sum_{ \eta_K } \mathcal F\big( S,\varphi^{-1}(\mathsf V,\mathsf E,\zeta,\eta_K) \big) = [1+o(1)] \mathcal F\big(S,\varphi^{-1}( \mathsf V,\mathsf E,\zeta,\eta_S) \big) \,. \label{eq-core-argument-lem-4.7}
    \end{align}
    For convenience, we will write $\varphi^{-1}(\eta_K)$ instead of $\varphi^{-1}(\mathsf V,\mathsf E,\zeta,\eta_K)$ below. Let $\| \eta_K \|_0$ be the number of $e \in E^{\mathsf a}(K) \cap E^{\mathsf a}(S)$ such that $\eta_K(e) \neq \eta_S(e)$. Note that 
    \begin{align*}
        & |\mathsf{dif}(\varphi^{-1}(\eta_K)|-|\mathsf{dif}(\varphi^{-1}(\eta_S))| \leq 2\| \eta_K \|_0 \,; \\
        & |E_\ell(\varphi^{-1}(\eta_K)|-|E_\ell((\varphi^{-1}(\eta_S))| \leq \| \eta_K \|_0  \,.
    \end{align*}
    By the definition of $\varphi$, we note that $E_0(\varphi^{-1}(\eta_K)) = E_0(\varphi^{-1}(\eta_S))$. By~\eqref{eq-assum-upper-bound}, we have $\epsilon_\ell^2\lambda_\ell\le 1$ for all $1\le \ell\le L$. 
    Thus, we have
    \begin{align}
        \frac{ \Xi(\varphi^{-1}(\eta_K)) }{ \Xi(\varphi^{-1}(\eta_S))) } &\overset{\eqref{eq-def-Xi(H)}}{=} \rho^{ 2(|\mathsf{dif}(\varphi^{-1}(\eta_K)|-|\mathsf{dif}(\varphi^{-1}(\eta_S))|) } \prod_{1 \leq \ell \leq L} (\epsilon^2_\ell \lambda_\ell)^{ \frac{1}{2}(|E_\ell(\varphi^{-1}(\eta_K)|-|E_\ell((\varphi^{-1}(\eta_S))|) } \nonumber \\
        &\leq \rho^{ -4\| \eta_K \|_0 } \prod_{1 \leq \ell \leq L} (\epsilon^2_\ell \lambda_\ell)^{ -\frac{1}{2}\| \eta_K \|_0 } \,. \label{eq-bound-ratio-Xi}
    \end{align}
    Similarly, using
    \begin{align*}
        & \Big( \big| E_\ell(S) \triangle E_\ell(\varphi^{-1}(\eta_K)) \big|- \big| E_\ell(S) \triangle E_\ell((\varphi^{-1}(\eta_S)) \big| \Big) \\
        \leq\ & \big| E_\ell(\varphi^{-1}(\eta_K)) \triangle E_\ell((\varphi^{-1}(\eta_S)) \big| \leq 2\| \eta_K \|_0 \,,
    \end{align*}
    we can show that
    \begin{align}
        &\frac{ \mathtt P(S,\varphi^{-1}(\eta_K)) }{ \mathtt P(S,\varphi^{-1}(\eta_S))) } \nonumber \\
        =\ & \rho^{ 2(|\mathsf{dif}(\varphi^{-1}(\eta_K) \setminus V(S)|-|\mathsf{dif}(\varphi^{-1}(\eta_S)) \setminus V(S)|) } \prod_{1 \leq \ell \leq L} (\epsilon^2_\ell \lambda_\ell)^{ \frac{1}{2}(|E_\ell(S) \triangle E_\ell(\varphi^{-1}(\eta_K)|-|E_\ell(S) \triangle E_\ell((\varphi^{-1}(\eta_S))|) } \nonumber \\
        \leq\ & \rho^{ -4\| \eta_K \|_0 } \prod_{1 \leq \ell \leq L} (\epsilon^2_\ell \lambda_\ell)^{ -\| \eta_K \|_0 } \,, \label{eq-bound-ratio-mathtt-P}
    \end{align}
    where the equality follows from~\eqref{eq-def-mathtt-P} and $E_0(\varphi^{-1}(\eta_K)) = E_0(\varphi^{-1}(\eta_S))$; the inequality follows from $\rho\le 1$ and $\epsilon_\ell^2\lambda_\ell\le 1$ for all $1\le \ell\le L$.
    Finally, note that 
    \begin{align}\label{eq-bound-ratio-C-Omega}
        \frac{ C^{|\mathsf L(S \Cap\varphi^{-1}(\eta_K))|} }{ C^{|\mathsf L(S \Cap\varphi^{-1}(\eta_S))|} } \leq C^{2\| \eta_K \|_0}, \quad  \frac{ \Omega(S \Cap \varphi^{-1}(\eta_S)) }{ \Omega( S \Cap \varphi^{-1}(\eta_K) ) } \overset{\eqref{eq-def-Omega(S)}}{=} n^{\|\eta_K\|_0} \,.
    \end{align}
    By~\eqref{eq-def-of-mathcalF-SK},~\eqref{eq-bound-ratio-Xi},~\eqref{eq-bound-ratio-mathtt-P}, and~\eqref{eq-bound-ratio-C-Omega}, we obtain \begin{align*}
        \frac{\mathcal F(S,\varphi^{-1}(\eta_K))}{\mathcal F(S,\varphi^{-1}(\eta_S))}\le C^{2\| \eta_K \|_0} n^{-\| \eta_K \|_0} \rho^{-8\| \eta_K \|_0}\prod_{1\le \ell \le L}(\epsilon^2_\ell\lambda_\ell)^{-\tfrac{3}{2}\| \eta_K \|_0}.
    \end{align*}
    We thus have
    \begin{align*}
        & \sum_{ \eta_K } \mathcal F\big( S,\varphi^{-1}(\eta_K) \big) = \mathcal F\big(S,\varphi^{-1}(\eta_S) \big) + \sum_{m>0} \sum_{ \| \eta_K\|_0=m } \mathcal F\big(S,\varphi^{-1}(\eta_K) \big) \\
        \leq\ & \mathcal F\big(S,\varphi^{-1}(\eta_S) \big) \Big( 1+ \sum_{m>0} C^{2m}n^{-m} \rho^{ -8m } \prod_{1 \leq \ell \leq L} (\epsilon^2_\ell \lambda_\ell)^{ -3m/2 } \cdot \#\{ \| \eta_K \|_0=m \} \Big) \\
        \leq\ & \mathcal F\big(S,\varphi^{-1}(\eta_S) \big) \Big( 1+ \sum_{m>0}C^{2m} n^{-m} \rho^{ -8m } \prod_{1 \leq \ell \leq L} (\epsilon^2_\ell \lambda_\ell)^{ -3m/2} \cdot \aleph^m \Big) =(1+o(1)) \mathcal F\big(S,\varphi^{-1}(\eta_S) \big) \,,
    \end{align*}
    leading to \eqref{eq-core-argument-lem-4.7}. Note that we have $E(K) \cap E(S)=E(K \Cap S)$ if and only if $\eta_K(e)=\eta_S(e)$ for all $e \in E^{\mathsf a}(K) \cap E^{\mathsf a}(S)$. Thus, we have
    \begin{align*}
        \eqref{eq-var-Phi-i,j-relax-2} &= n^2 \sum_{ \substack{ [S],[K] \in \mathcal J \\ \mathsf L(S)=\mathsf L(K)=\{ u,v \} } } \mathcal F(S,K) \leq [1+o(1)] n^2 \sum_{ \substack{ [S],[K] \in \mathcal J \\ \mathsf L(S)=\mathsf L(K)=\{ u,v \} \\ E(S) \cap E(K)=E(S \Cap K) } } \mathcal F(S,K) \\
        &= [1+o(1)]\sum_{ \substack{ [S],[K] \in \mathcal J \\ \mathsf L(S)=\mathsf L(K)=\{ u,v \} \\ E(S) \cap E(K)=E(S \Cap K) } } \frac{ n^2 C^{|\mathsf{L}(S \Cap K)|+|\mathsf{L}(S \cap K)|+1} \Xi(S) \Xi(K) \mathtt P(S,K) \Omega(S \Cap K) }{ \Omega(S) \Omega(K) \beta_{\mathcal J}^2 } \\
        &\leq[1+o(1)] \sum_{ \substack{ [S],[K] \in \mathcal J \\ \mathsf L(S)=\mathsf L(K)=\{ u,v \} \\ E(S) \cap E(K)=E(S \Cap K) } } \frac{ n^2 C^{2|\mathsf{L}(S \cap K)|+1} \Xi(S) \Xi(K) \mathtt P(S,K) \Omega(S \Cap K) }{ \Omega(S) \Omega(K) \beta_{\mathcal J}^2 } \,,
    \end{align*}
    where the last inequality follows from $|\mathsf{L}(S \Cap K)| \leq |\mathsf L(S \cap K)|$ if $E(S \cap K)=E(S \Cap K)$ (since in this case $S \cap K$ is $S \Cap K$ adding some isolated vertices).
\end{proof}

Based on Lemma~\ref{lem-reduce-to-good-S,K}, it suffices to bound the right-hand side of \eqref{eq-var-Phi-i,j-relax-3}. We first introduce some notations. Suppose that $S \cap K$ has $\mathtt T+1$ connected components $S_1,\ldots,S_{\mathtt T+1}$ such that $a_u \in V(S_1),a_v \in V(S_{\mathtt T+1})$. In particular, if $S=K$ we just let $\mathtt T=0$ and $S_1=S=K$. Since $E(S) \cap E(K)=E(S \Cap K)$, we have that the decorated graph obtained by restricting $K$ in $V(S_i)$ equals $S_i$. Note that $S$ is a decorated path, thus $S_{\mathtt t}$ is a path for all $1 \leq \mathtt t \leq \mathtt T+1$ (we allow a path to be an isolated vertex). 
Thus, we also have $E(S) \setminus E(S_1 \cup S_2 \ldots \cup S_{\mathtt T})$ can be decomposed into paths $\widetilde{S}_1,\ldots,\widetilde{S}_{\mathtt T}$ such that $\widetilde{S}_{\mathtt t}$ lies between $S_{\mathtt t}$ and $S_{\mathtt t+1}$. 
Since $a_u\in V(S_1)$ and $a_v\in V(S_{\mathtt T+1})$, the intersection $S\cap K$ must contain
one more connected component than the remaining connected components that are formed by $E(S) \setminus E(S_1 \cup S_2 \ldots \cup S_{\mathtt T})$.
Similarly, $E(K) \setminus E(S_1 \cup S_2 \ldots \cup S_{\mathtt T})$ can be decomposed into disjoint paths $\widetilde{K}_1,\ldots,\widetilde{K}_{\mathtt T}$ such that $\widetilde{K}_{\mathtt t}$ lies between $K_{\mathtt t}$ and $K_{\mathtt t+1}$. Define
\begin{align}{\label{eq-def-S-cap-S-setminus-K-setminus}}
    S_{\cap}:=\big( S_1,\ldots,S_{\mathtt T+1} \big), \quad S_{\setminus}:=\big( \widetilde{S}_1,\ldots,\widetilde{S}_{\mathtt T} \big), \quad K_{\setminus}:=\big( \widetilde{K}_1,\ldots,\widetilde{K}_{\mathtt T} \big) \,.
\end{align}
See Figure~\ref{fig-example-S-cap} for an illustration.
In addition, define $\mathcal A_{\mathtt T}$ to be the collection of $S_{\cap},S_{\setminus},K_{\setminus}$ such that there exists $[S],[K] \in \mathcal J,E(S) \cap E(K)=E(S \Cap K)$ and $\mathsf L(S)=\mathsf L(K)=\{ a_u,a_v \}$ and $S_{\cap},S_{\setminus},K_{\setminus}$ is the desired decomposition of $S \cap K, S \setminus K, K \setminus S$ respectively. It is also clear that all $S,K$ such that $E(S) \cap E(K)=E(S \Cap K)$ is uniquely determined by $S_{\cap},S_{\setminus},K_{\setminus}$. The next lemma shows how to bound $\Xi(S)$, $\Xi(K)$ and $\mathtt P(S,K)$ using the quantity above. 

\begin{figure}[htbp]
\centering
\includegraphics[width=0.9\linewidth]{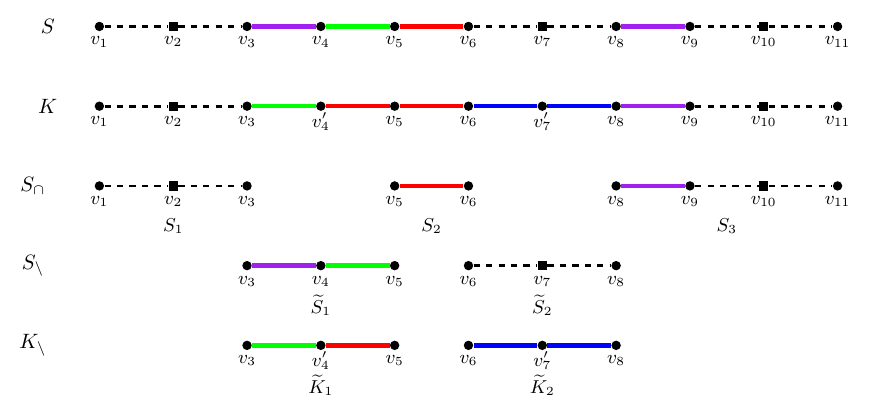}
\caption{An example of decorated paths $S$ and $K$ with $E(S)\cap E(K) = E(S\Cap K)$, together with $S_{\cap}$, $S_{\setminus}$, and $K_{\setminus}$, for $\mathtt T=2$.}
\label{fig-example-S-cap}
\end{figure}

\begin{lemma}{\label{lem-bound-Xi-mathtt-P}}
    For all $E(S) \cap E(K)=E(S \Cap K)$ and $S_{\cap},S_{\setminus},K_{\setminus} \in \mathcal A_{\mathtt T}$. We have 
    \begin{align}
        & \Xi(S) \leq \rho^{-10\mathtt T} \cdot \prod_{1 \leq \mathtt t \leq \mathtt T+1} \Xi(S_{\mathtt t}) \prod_{1 \leq \mathtt t \leq \mathtt T} \Xi(\widetilde S_{\mathtt t})  \,, \label{eq-bound-Xi(S)} \\
        & \Xi(K) \leq \rho^{-10\mathtt T} \cdot \prod_{1 \leq \mathtt t \leq \mathtt T+1} \Xi(K_{\mathtt t}) \prod_{1 \leq \mathtt t \leq \mathtt T} \Xi(\widetilde K_{\mathtt t}) \,,  \label{eq-bound-Xi(K)} \\
        & \mathtt P(S,K) \leq \rho^{-20\mathtt T} \prod_{1 \leq \mathtt t \leq \mathtt T} \Xi(\widetilde S_{\mathtt t}) \prod_{1 \leq \mathtt t \leq \mathtt T} \Xi(\widetilde K_{\mathtt t}) \,. \label{eq-bound-mathtt-P(S,K)}
    \end{align}
\end{lemma}

The proof of Lemma~\ref{lem-bound-Xi-mathtt-P} is postponed to Section~\ref{subsec:proof-lem-4.8} of the appendix. 
Note that $S_{\mathtt t} = K_{\mathtt t}$ for $1\le \mathtt t\le \mathtt T+1$.
Using Lemma~\ref{lem-bound-Xi-mathtt-P}, we have that
\begin{align*}
     \Xi(S)\Xi(K)\mathtt P(S,K) \leq\ & \rho^{-40\mathtt T} \prod_{1\leq \mathtt t \leq \mathtt T+1} \Xi(S_{\mathtt t})^2 \cdot \prod_{1\leq \mathtt t \leq \mathtt T} \Xi(\widetilde S_{\mathtt t})^2 \prod_{1\leq \mathtt t \leq \mathtt T} \Xi(\widetilde K_{\mathtt t})^2 \\
     =\ & \rho^{-40\mathtt T} \Xi(S_{\cap})^2 \Xi(S_{\setminus})^2 \Xi(K_{\setminus})^2 \,,
\end{align*}
where \begin{align*}
     \Xi(S_{\cap}) := \prod_{1\le \mathtt t\le \mathtt T+1} \Xi(S_{\mathtt t}),\quad \Xi(S_{\setminus}) := \prod_{1\le \mathtt t\le \mathtt T} \Xi(\widetilde{S}_{\mathtt t}),\quad \Xi(K_{\setminus}) := \prod_{1\le \mathtt t\le \mathtt T} \Xi(\widetilde{K}_{\mathtt t}).
\end{align*}
We similarly define $\Omega(S_{\cap}),\Omega(S_{\setminus})$, and $\Omega(K_{\setminus})$.
In addition, since $E(S)\cap E(K) = E(S\Cap K)$, it is clear that 
\begin{align*}
    |\mathsf{L}(S \cap K)|\leq 2(\mathtt T+1) \mbox{ and } \frac{ \Omega(S)\Omega(K) }{ \Omega(S \Cap K) } = \Omega(S_{\cap}) \Omega(S_{\setminus}) \Omega(K_{\setminus}) \,.
\end{align*}
Thus, we can write \eqref{eq-var-Phi-i,j-relax-3} as (recall the definition of $\mathcal{A}_\mathtt T$ below \eqref{eq-def-S-cap-S-setminus-K-setminus})
\begin{align}
    \eqref{eq-var-Phi-i,j-relax-3} &\leq \sum_{\mathtt T \geq 0} \sum_{ S_{\cap}, S_{\setminus}, K_{\setminus} \in \mathcal A_{\mathtt T} } \frac{ n^2 C^{4\mathtt T+5} \rho^{-40\mathtt T} \Xi(S_{\cap})^2 \Xi(S_{\setminus})^2 \Xi(K_{\setminus})^2 }{ \Omega(S_{\cap}) \Omega(S_{\setminus}) \Omega(K_{\setminus}) \beta_{\mathcal J}^2 }  \nonumber  \\ 
    &\leq \sum_{\mathtt T \geq 0} \frac{C^{4\mathtt T+5}\rho^{-40\mathtt T} n^{-3\mathtt T+1}}{\beta_{\mathcal J}^2} \sum_{ S_{\cap}, S_{\setminus}, K_{\setminus} \in \mathcal A_{\mathtt T} } \prod_{1 \leq \mathtt t \leq \mathtt T+1} \frac{ n\Xi(S_{\mathtt t})^2 }{ \Omega(S_{\mathtt t}) } \prod_{1 \leq \mathtt t \leq \mathtt T} \frac{ n\Xi(\widetilde S_{\mathtt t})^2 }{ \Omega(\widetilde S_{\mathtt t}) }  \prod_{1 \leq \mathtt t \leq \mathtt T} \frac{ n\Xi(K_{\mathtt t})^2 }{ \Omega(\widetilde K_{\mathtt t}) }  \,.  \label{eq-var-Phi-i,j-relax-4}
\end{align}
We will bound \eqref{eq-var-Phi-i,j-relax-4} via the following lemma, thus completing the proof of Lemma~\ref{lem-conditional-var-Phi-i,j}. The proof of the following lemma is postponed to Section~\ref{subsec:proof-lem-4.9} of the appendix.
\begin{lemma}{\label{lem-recovery-most-technical}}
    Suppose that \eqref{eq-assum-upper-bound} holds and we choose $\aleph$ according to \eqref{eq-condition-weak-recovery}. Then $\eqref{eq-var-Phi-i,j-relax-4} \leq O(1)$. 
\end{lemma}

\subsection{Proof of Theorem~\ref{MAIN-THM-recovery}}{\label{subsec:proof-thm-3.9}}

With Proposition~\ref{main-prop-recovery}, we are ready to prove Theorem~\ref{MAIN-THM-recovery}.
\begin{proof}[Proof of Theorem~\ref{MAIN-THM-recovery}]
    Note that Proposition~\ref{main-prop-recovery} implies that there exists a sufficiently large constant $C\geq 100\rho^{-2}$ where
    \begin{align*}
        \mathbb E_{\Pb}\Big[ \langle \Phi^{\mathcal J}, \bm x \bm x^{\top} \rangle \Big] \geq \rho^2 n^2 \mbox{ and } \mathbb E_{\Pb}\Big[ \| \Phi^{\mathcal J} \|_{\Fop}^2 \Big] \leq C n^2 \,.
    \end{align*}
    Thus, using Chebyshev's inequality, we have $\Pb(\| \Phi^{\mathcal J} \|_{\Fop} \geq C^2 n)\leq C^{-3}$, and thus
    \begin{align*}
        \mathbb E\Big[ \langle \Phi^{\mathcal J}, \bm x \bm x^{\top} \rangle \cdot \mathbf 1_{ \{ \| \Phi^{\mathcal J} \|_{\Fop} \geq C^2 n \} } \Big] &\overset{\mathrm{(a)}}{\le} n \mathbb E\Big[ \| \Phi^{\mathcal J} \|_{\Fop} \cdot \mathbf 1_{ \{ \| \Phi^{\mathcal J} \|_{\Fop} \geq C^2 n \} } \Big] \\
        &\overset{\mathrm{(b)}}{\le} n \Pb\Big( \| \Phi^{\mathcal J} \|_{\Fop} \geq C^2 n \Big)^{\frac{1}{2}} \mathbb E\Big[ \| \Phi^{\mathcal J} \|^2_{\Fop}  \Big]^{\frac{1}{2}} \leq C^{-1} n^2 \,,
    \end{align*}
    where $\mathrm{(a)}$ and $\mathrm{(b)}$ follow from the Cauchy-Schwarz inequality and $\|\bm x\|_{\Fop}^2 = n^2$.
    Thus, we have 
    \begin{align*}
        \mathbb E_{\Pb}\Big[ \langle \Phi^{\mathcal J}, \bm x \bm x^{\top} \rangle \cdot \mathbf 1_{ \{ \| \Phi^{\mathcal J} \|_{\Fop} \leq C^2n \} } \Big] \geq (\rho^2-C^{-1})n^2 \,.
    \end{align*}
    By Markov's inequality, we obtain
    \begin{align*}
        \Pb\left( \langle \Phi^{\mathcal J}, \bm x \bm x^{\top} \rangle \geq \frac{\rho^2 n^2}{2} \right) \geq C^{-2} \,.
    \end{align*}
    Combined with the fact that $\Pb( \| \Phi^{\mathcal J} \|_{\Fop} \geq C^2 n )\leq C^{-3}$, we get that
    \begin{align*}
        \Pb\left( \langle \Phi^{\mathcal J}, \bm x \bm x^{\top} \rangle \geq \frac{ \rho^2 }{ 2C^4 } \| \Phi^{\mathcal J} \|_{\Fop}^2 \right) \geq \Omega(1) \,.
    \end{align*}
    Similarly as in \cite[Section~2.3]{HS17}, we compute $\widehat{\Phi}$ with minimum Frobenius norm that satisfies the following constraints (below we use $\mathsf{ddiag}(M)$ to denote the diagonal part of a square matrix $M$):
    \begin{align*}
        \mathsf{ddiag}(\widehat{\Phi})=\mathbb I_n \,; \quad \frac{ \langle \widehat{\Phi},\Phi^{\mathcal J} \rangle }{ n \| \Phi^{\mathcal J} \|_{\Fop} } \geq \frac{\rho^2}{4C^4} \,; \quad \widehat{\Phi} \succeq 0 \,.
    \end{align*}
    We see that
    \begin{align*}
        \Pb\left( \langle \widehat\Phi, \bm x \bm x^{\top} \rangle \geq \frac{ \rho^2 }{ 4C^4 } \| \widehat\Phi \|_{\Fop} \| \bm x \bm x^{\top} \|_{\Fop} \right) \geq \Omega(1) \,.
    \end{align*}
    Thus, we have that (note that $\langle \widehat\Phi,\bm x \bm x^{\top} \rangle\geq 0$ since $\widehat\Phi\succeq0$) 
    \begin{align*}
        \mathbb E_{\Pb}\Bigg[ \frac{ \langle \widehat\Phi, \bm x \bm x^{\top} \rangle }{ \| \widehat\Phi \|_{\Fop} \| \bm x \bm x^{\top} \|_{\Fop} } \Bigg] \geq \frac{\rho^2}{4C^4} \cdot \Pb\Bigg( \frac{ \langle \widehat\Phi, \bm x \bm x^{\top} \rangle }{ \| \widehat\Phi \|_{\Fop} \| \bm x \bm x^{\top} \|_{\Fop} } \geq \frac{\rho^2}{4C^4} \Bigg) = \Omega(1) \,,
    \end{align*}
    i.e., $\widehat\Phi$ achieves weak recovery defined in Definition~\ref{def-weak-recovery}.
\end{proof}

\section{Numerical experiments}{\label{sec:numer-results}}

In this section, we present numerical results on synthetic data to verify our theoretical results. We first consider the detection problem by running Algorithm~\ref{alg:detection-meta} on a contextual inhomogeneous two-layer model, formally defined in Definition~\ref{def-contxtual-multilayer-SBM}. To this end, we generate $100$ i.i.d.\ tuples $(\bm Y,\bm G_1,\bm G_2)$ from $\Qb_{n,p,\lambda_1,\lambda_2}$, and another $100$ i.i.d.\ tuples from the contextual two-layer SBM $\Pb_{n,p;\mu,\rho,\lambda_1,\lambda_2,\epsilon_1,\epsilon_2}$.
%we independently generate $100$ tuples of $(\bm Y,\bm G_1,\bm G_2)$ that are independently sampled from $\Qb_{ n,p,\lambda_1,\lambda_2 }$, and another $100$ tuples $(\bm Y,\bm G_1,\bm G_2)$ from the of contextual $2$-layer SBMs $\Pb_{ n,p;\mu,\rho,\lambda_1,\lambda_2,\epsilon_1,\epsilon_2 }$.

\begin{figure}[htbp] 
\centering
\subfloat{%
  \includegraphics[width=0.35\linewidth]{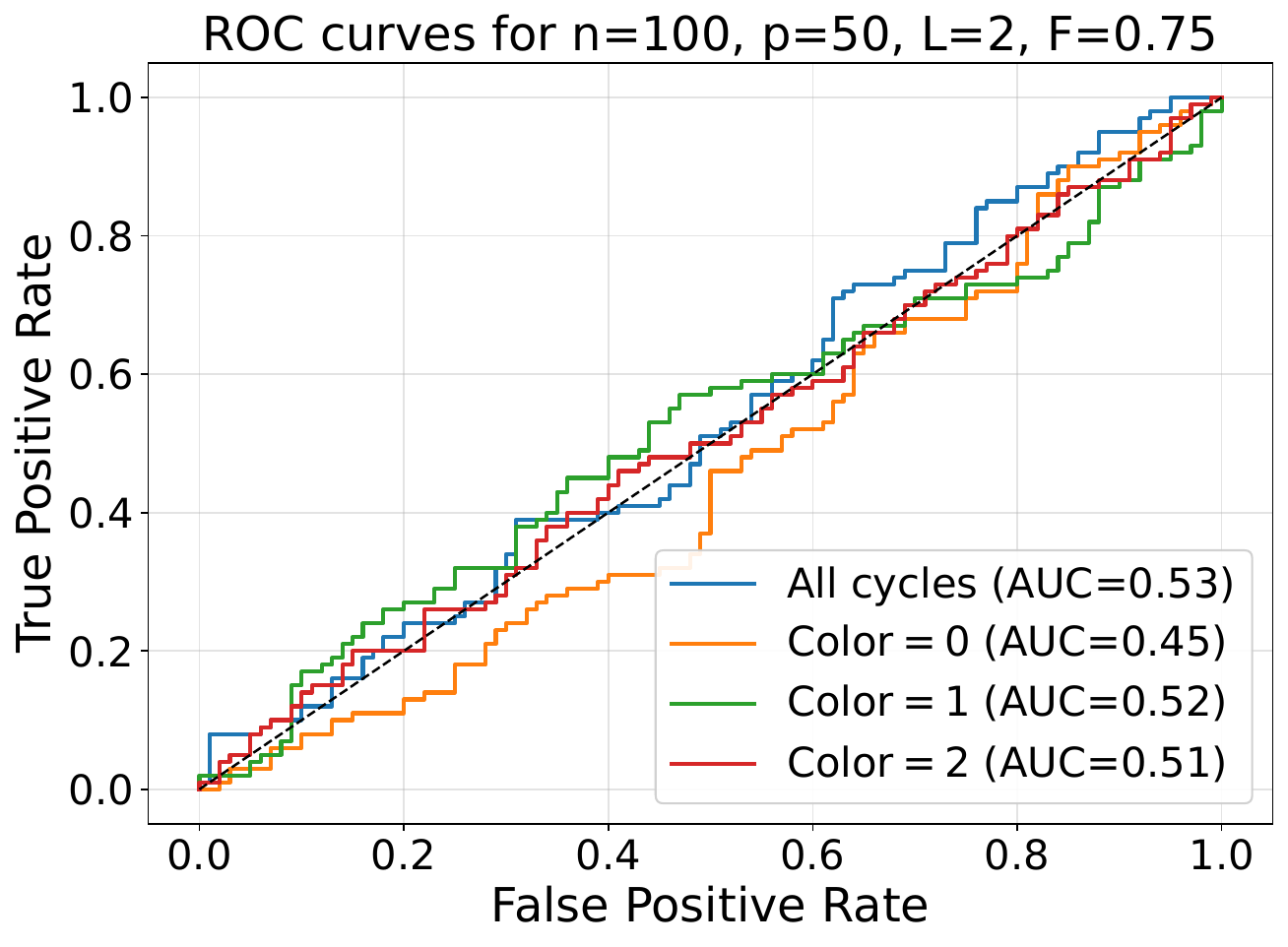}%
}\hspace{4mm}
\subfloat{%
  \includegraphics[width=0.35\linewidth]{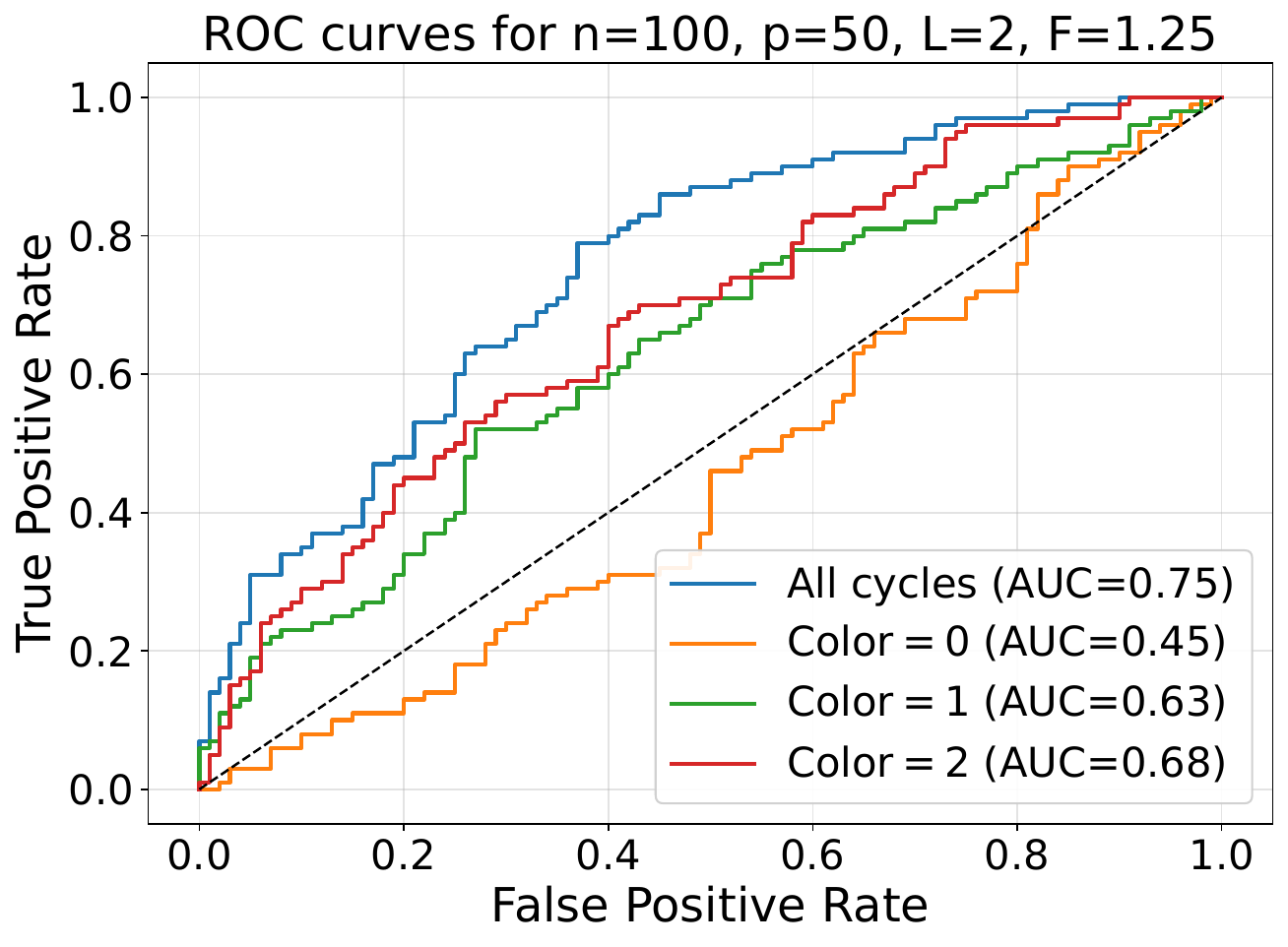}%
}

\subfloat{%
  \includegraphics[width=0.35\linewidth]{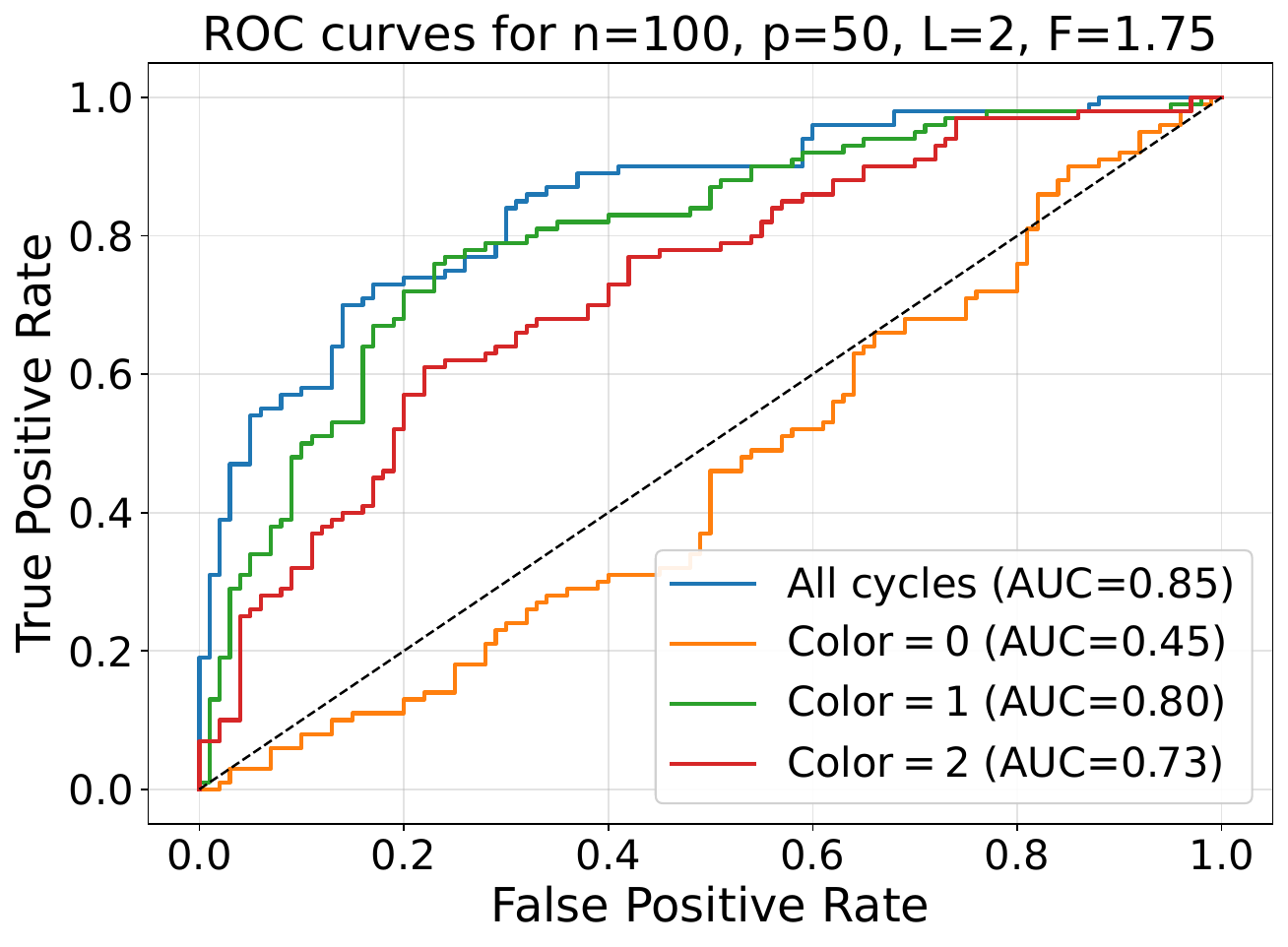}%
}
\hspace{4mm}
\subfloat{%
  \includegraphics[width=0.35\linewidth]{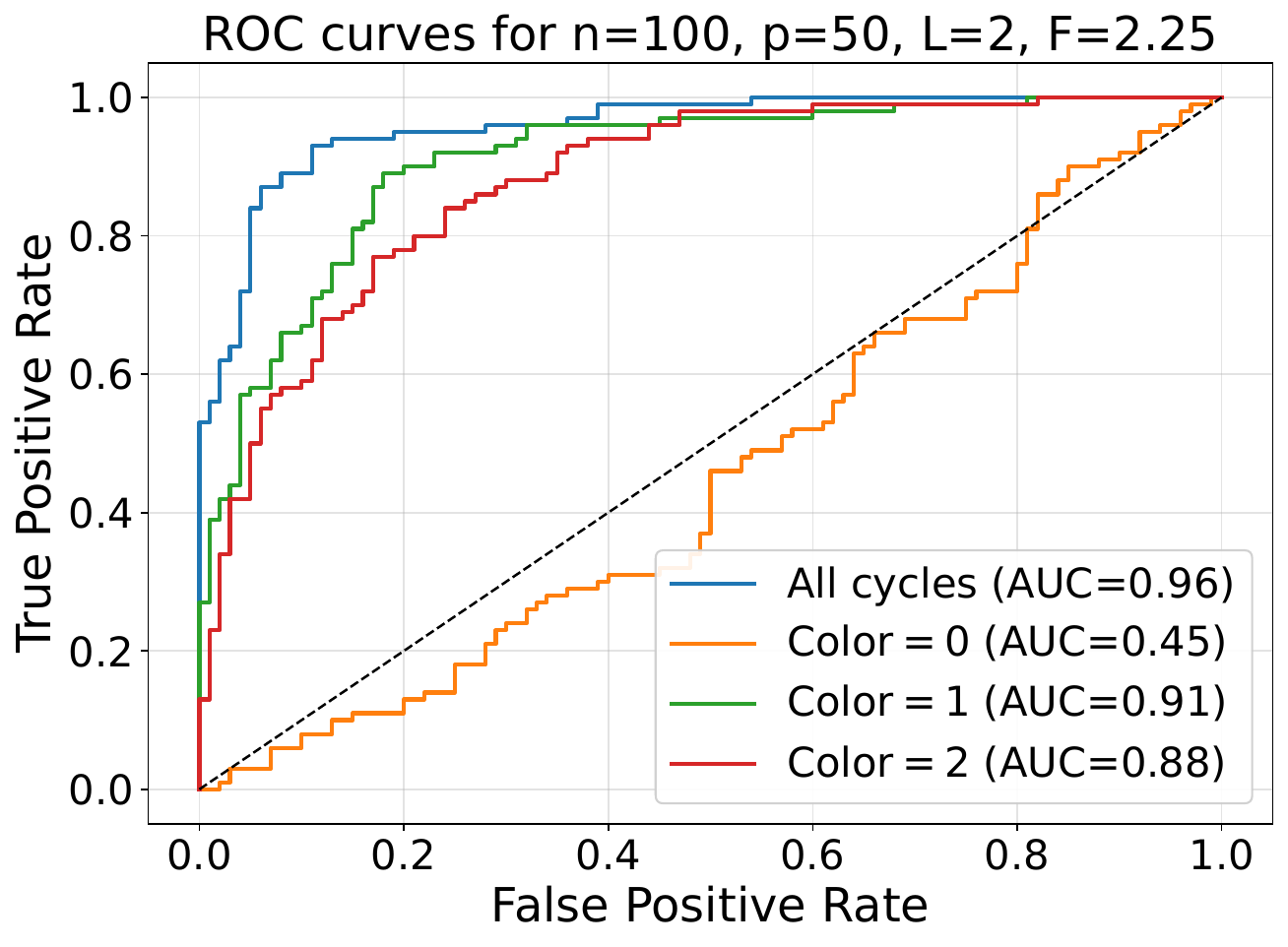}%
}
\caption{ROC curves for the contextual $2$-layer SBMs with $n=100$, $p=50$, $\gamma=2$ and $F(\mu,p,\gamma;\{\lambda_\ell\},\{\epsilon_\ell\})\in\{0.75,1.25,1.75, 2.25\}$.}
\label{fig:ROC-detection-simu-L2Cycle4}
\end{figure}

To compare the performance of our test statistic with single-channel statistics under different settings, we plot the Receiver Operating Characteristic (ROC) curves by varying detection threshold and plotting the true positive rate (one minus Type-II error) against the false positive rate (Type-I error). For comparison, we also plot the ROC curve for the single-channel method and random classifier, which is simply the diagonal. Finally, we compute the area under the curve (AUC), which can be interpreted as the probability that the test statistic for a graph pair drawn from the correlated model exceeds that for a graph pair drawn from the independent model.

In Figure~\ref{fig:ROC-detection-simu-L2Cycle4}, for each plot we fix $n=100,p=50,\gamma=2,L=2,\rho = 0.6,\mu=0.5, \epsilon_1 = \epsilon_2 = 0.5$, and vary $\lambda_1 = \lambda_2\in \{3,5,7,9\}$. The resulting $F(\mu,\rho,\gamma;\{ \lambda_1,\lambda_2 \},\{ \epsilon_1,\epsilon_2 \}) $ values are $\{ 0.75,1.25,1.75,2.25 \}$. As it was shown in \cite{MNS15} that optimal tests in single SBM are given by counting cycles, we focus on the counts of all decorated cycles (as proposed in our work) and the counts of single-colored cycles (corresponding to the single-channel method) with length $\aleph=4$. We observe that as $F$ increases, the ROC curve is moving toward the upper left corner and the AUC increases, demonstrating that our test statistic has improved performance. Moreover, we observe that the performance of our test statistics differs significantly between $F<1$ and $F>1$. This is consistent with the theoretical results, showing that our test statistic works as long as $F>1$.

We next study the recovery problem by running Algorithm~\ref{alg:recovery-meta} on the 2-layer contextual SBM. Specifically, we use the decorated path family $\mathcal J$ consisting of all decorated paths of length $\aleph=4$. To compare the performance of our recovery estimator $\hat{\Phi}$ with single-channel baselines, we report the Frobenius cosine similarity $\tfrac{|\langle \hat{\Phi}, \bm x\bm x^\top\rangle|}{\|\hat{\Phi}\|_{\Fop}\,\|\bm x\bm x^\top\|_{\Fop}}$ as a function of the aggregate signal strength $F(\mu,p,\gamma;\{\lambda_\ell\},\{\epsilon_\ell\})$. We also plot the same cosine similarity for the raw statistic $\Phi^{\mathcal J}$ constructed in Algorithm~\ref{alg:recovery-meta}.

In Figure~\ref{fig:ROC-recovery-simu-L2path5}, we fix $n=100$, $p=50$, $\gamma=2$, $\rho=0.5$, $\mu=0.75$, and set $\epsilon_\ell=0.5$ for all $\ell$. We further impose $\lambda_\ell=\lambda$ for all $\ell$, and vary $\lambda$ so that $F(\mu,p,\gamma;\{\lambda_\ell\},\{\epsilon_\ell\})\in\{0.3,0.5,0.7,0.9,1.1,1.3,1.5,1.7\}$. We consider both the statistic based on all decorated paths and the single-color variants with $\aleph=4$. As $F$ increases, the cosine similarity generally increases, indicating improved recovery performance; in contrast, the ``Color=$0$'' baseline remains  essentially flat since we vary only $\lambda$ (and thus $F$) while keeping $\mu$ fixed. Moreover, we observe a clear qualitative difference between the regimes $F<1$ and $F>1$, which is consistent with the phase transition as shown in Theorem~\ref{MAIN-THM-informal}.

\begin{figure}[htbp] 
\centering
\subfloat{%
\includegraphics[width=0.4\linewidth]{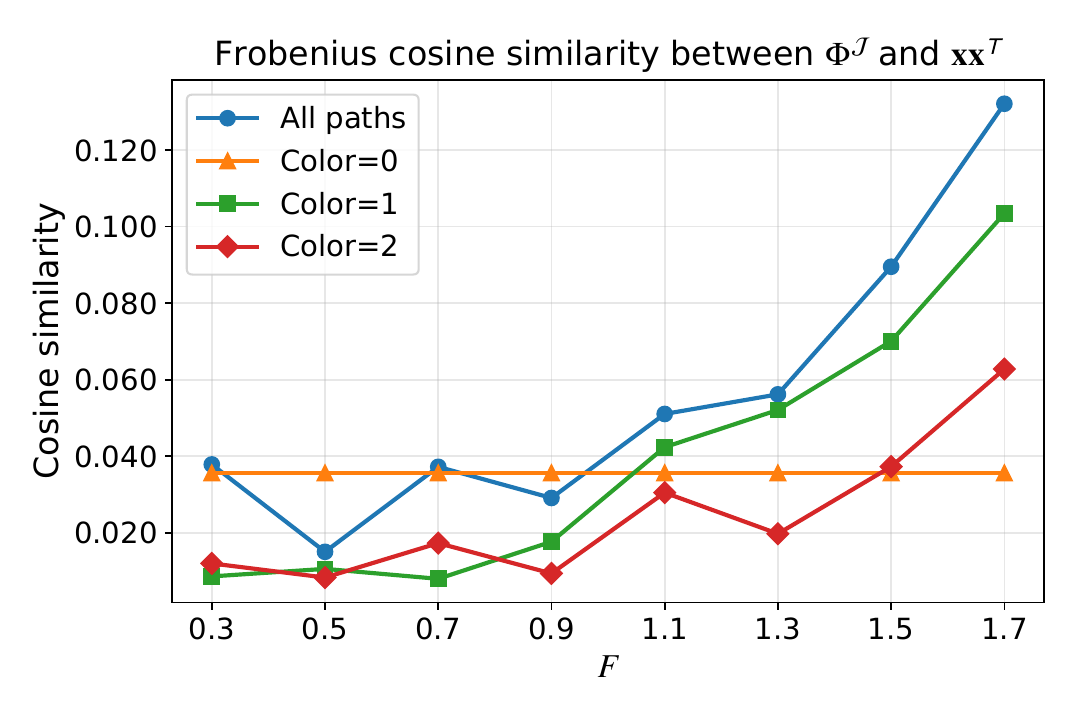}%
}\hspace{4mm}
\subfloat{%
  \includegraphics[width=0.4\linewidth]{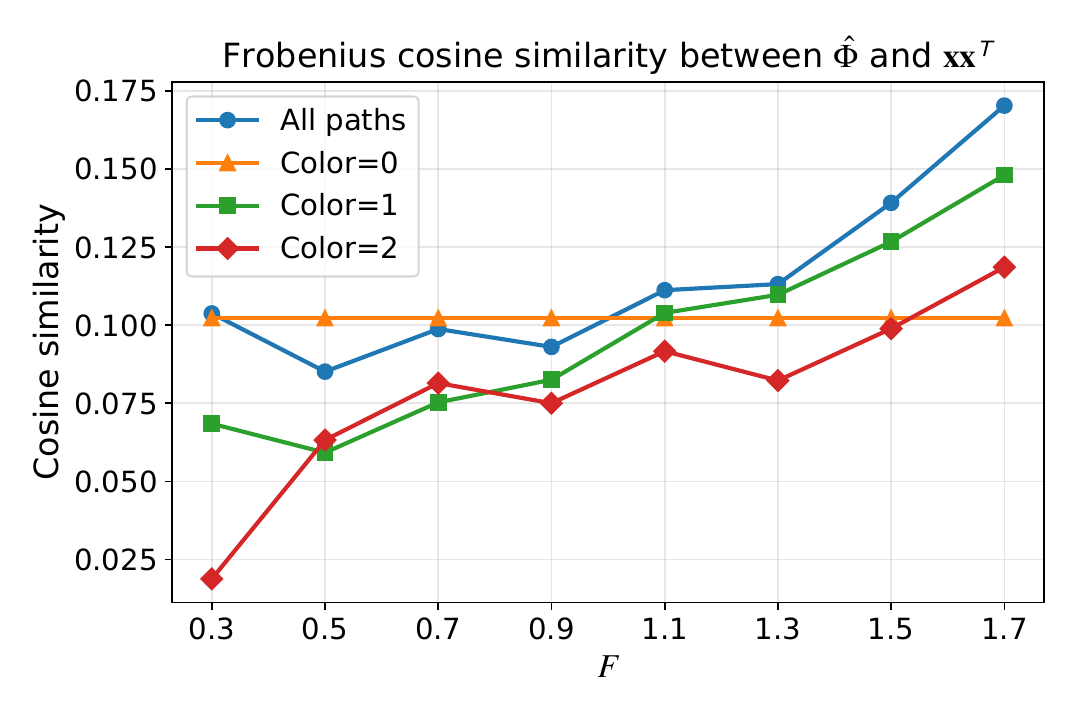}%
}
\caption{Cosine similarity $\frac{|\langle\hat{\Phi},\bm x\bm x^\top \rangle|}{\|\hat{\Phi} \|_{\Fop}\|\bm x\bm x^\top\|_{\Fop}}$ with $n=100$, $p=50$, $L=2$ and $F(\mu,p,\gamma;\{\lambda_\ell\},\{\epsilon_\ell\})\in\{0.3,0.5,0.7,0.9,1.1,1.3,1.5,1.7\}$.}
\label{fig:ROC-recovery-simu-L2path5}
\end{figure}

\section*{Acknowledgements}
S. Gong and Z. Li are partially supported by the National Key R\&D program of China (Project No. 2023YFA1010103) and the NSFC Key Program (Project No. 12231002).

\appendix

\section{Supplementary proofs in Section~\ref{sec:info-lower-bound}}{\label{sec:supp-proofs-sec-2}}

\subsection{Proof of Lemma~\ref{lem-chi-2-divergence-relax}}{\label{subsec:proof-lem-2.4}}

We prove Lemma~\ref{lem-chi-2-divergence-relax} by direct calculation. Denote $\Pb_{t,\bm x,\bm x_1,\ldots,\bm x_L}:=\Pb_t(\cdot\mid\bm x,\bm x_1,\ldots,\bm x_L)$. It is straightforward that
\begin{align*}
    \frac{ \mathrm{d}\Pb_t }{ \mathrm{d}\Qb }(\bm Y,\bm G_1,\ldots,\bm G_{L}) &= \mathbb E_{\bm x,\bm x_1,\ldots,\bm x_L}\left[ \frac{ \mathrm{d}\Pb_{t,\bm x,\bm x_1,\ldots,\bm x_{L}} }{ \mathrm{d}\Qb } (\bm Y,\bm G_1,\ldots,\bm G_{L}) \right] \\
    &= \mathbb E_{\bm x,\bm x_1,\ldots,\bm x_L}\left[ \frac{ \mathrm{d}\Pb_{t,\bm x} }{ \mathrm{d}\Qb } (\bm Y) \prod_{1 \leq \ell \leq L} \frac{ \mathrm{d}\Pb_{\bm x_{\ell}} }{ \mathrm{d}\Qb } (\bm G_{\ell}) \right] \,.
\end{align*}
Since $\bm Y$ and $\bm G_1,\ldots, \bm G_L$ are conditionally independent given $\bm x,\bm x_1,\ldots, \bm x_L$, applying the
replica's trick yields
\begin{align}
    \chi^2(\Pb_t\|\Qb) &= \mathbb E_{\Qb}\left\{ \mathbb E_{\bm x,\bm x_1,\ldots,\bm x_L}\left[ \frac{ \mathrm{d}\Pb_{t,\bm x} }{ \mathrm{d}\Qb } (\bm Y) \prod_{1 \leq \ell \leq L} \frac{ \mathrm{d}\Pb_{\bm x_{\ell}} }{ \mathrm{d}\Qb } (\bm G_{\ell}) \right]^2 \right\} \nonumber \\
    &= \mathbb E_{\Qb}\left\{ \mathbb E_{ \substack{ \bm x,\bm x_1,\ldots,\bm x_L \\ \bm x',\bm x_1',\ldots,\bm x_L' } }\left[ \frac{ \mathrm{d}\Pb_{t,\bm x} }{ \mathrm{d}\Qb } (\bm Y) \prod_{1 \leq \ell \leq L} \frac{ \mathrm{d}\Pb_{\bm x_{\ell}} }{ \mathrm{d}\Qb } (\bm G_{\ell}) \cdot \frac{ \mathrm{d}\Pb_{t,\bm x'} }{ \mathrm{d}\Qb } (\bm Y) \prod_{1 \leq \ell \leq L} \frac{ \mathrm{d}\Pb_{\bm x_{\ell}'} }{ \mathrm{d}\Qb } (\bm G_{\ell}) \right] \right\} \nonumber \\
    &= \mathbb E_{ \substack{ \bm x,\bm x_1,\ldots,\bm x_L \\ \bm x',\bm x_1',\ldots,\bm x_L' } } \left\{ \mathbb E_{\Qb}\left[ \frac{ \mathrm{d}\Pb_{t,\bm x} }{ \mathrm{d}\Qb } (\bm Y) \frac{ \mathrm{d}\Pb_{t,\bm x'} }{ \mathrm{d}\Qb } (\bm Y) \right] \prod_{1 \leq \ell \leq L} \mathbb E_{\Qb}\left[ \frac{ \mathrm{d}\Pb_{\bm x_{\ell}} }{ \mathrm{d}\Qb } (\bm G_{\ell}) \frac{ \mathrm{d}\Pb_{\bm x_{\ell}'} }{ \mathrm{d}\Qb } (\bm G_{\ell}) \right] \right\} \,. \label{eq-chi-2-div-reduce-1}
\end{align}
We now calculate the right-hand side of \eqref{eq-chi-2-div-reduce-1}. Note that under $\Pb_t$ we have $\bm u$ is uniformly distributed over the $p$-dimensional sphere $\sqrt{(1+t)p} \cdot \mathbb S^{p-1}$. Thus,
\begin{align*}
    \frac{ \mathrm{d}\Pb_{t,\bm x} }{ \mathrm{d}\Qb } (\bm Y) &= \mathbb E_{\bm u \sim \operatorname{Unif}(\sqrt{(1+t)p} \cdot \mathbb S^{p-1})}\left[ \frac{ \mathrm{d}\Pb_{t,\bm x,\bm u} }{ \mathrm{d}\Qb } (\bm Y) \right] \\
    &= \mathbb E_{\bm u \sim \operatorname{Unif}(\sqrt{(1+t)p} \cdot \mathbb S^{p-1})}\left[ \exp\left( \frac{ \sqrt{\mu} }{ \sqrt{n} } \langle \bm Y,\bm x \bm u^{\top} \rangle - \frac{\mu}{2n} \|\bm x\|^2 \|\bm u\|^2 \right) \right] \\
    &= \mathbb E_{\bm u \sim \operatorname{Unif}(\sqrt{(1+t)p} \cdot \mathbb S^{p-1})}\left[ \exp\left( \frac{ \sqrt{\mu} }{ \sqrt{n} } \langle \bm Y,\bm x \bm u^{\top} \rangle - \frac{(1+t)\mu p}{2} \right) \right] \,.
\end{align*}
Thus, we have
\begin{align}
    &\mathbb E_{\Qb}\left[ \frac{ \mathrm{d}\Pb_{t,\bm x} }{ \mathrm{d}\Qb } (\bm Y) \frac{ \mathrm{d}\Pb_{t,\bm x'} }{ \mathrm{d}\Qb } (\bm Y) \right] \nonumber \\ 
    =\ & \mathbb E_{\bm Y \sim \Qb}\mathbb E_{\bm u,\bm u' \sim \operatorname{Unif}(\sqrt{(1+t)p} \cdot \mathbb S^{p-1})}\left[ \exp\left( \frac{ \sqrt{\mu} }{ \sqrt{n} } \langle \bm Y,\bm x \bm u^{\top}+\bm x' (\bm u')^{\top} \rangle - (1+t)\mu p \right) \right] \nonumber \\
    \overset{\mathrm{(a)}}{=}\ & \mathbb E_{\bm u,\bm u' \sim \operatorname{Unif}(\sqrt{(1+t)p} \cdot \mathbb S^{p-1})}\left[ \exp\left( \frac{\mu}{2n} \| \bm x \bm u^{\top}+\bm x' (\bm u')^{\top} \|_{\Fop}^2 - (1+t)\mu p \right) \right] \nonumber \\
    \overset{\mathrm{(b)}}{=}\ & \mathbb E_{\bm u,\bm u' \sim \operatorname{Unif}(\sqrt{(1+t)p} \cdot \mathbb S^{p-1})}\left[ \exp\left( \frac{\mu}{n} \langle \bm x, \bm x' \rangle \langle \bm u,\bm u' \rangle \right) \right] \,,  \label{eq-chi-2-div-reduce-2-part-0}
\end{align}
where $\mathrm{(a)}$ follows from $\E [\exp(cZ)] = \exp(\frac{1}{2}c^2)$ for $Z\sim \mathcal N(0,1)$ and any $c\in \mathbb R$, and $\mathrm{(b)}$ follows from $\| \bm x\bm u^\top \|_F^2 = \|\bm x'(\bm u')^\top \|_F^2=n(1+t)p$. Let $\xi = \xi(\bm u,\bm u') = \frac{1}{(1+t)p} \langle \bm u,\bm u'\rangle $ for $\bm u,\bm u'\sim \operatorname{Unif}(\sqrt{(1+t)p} \cdot \mathbb S^{p-1})$. Then $\xi\in [-1,1]$ and has density 
\begin{align*}
    f_\xi(x) = \frac{\Gamma(p/2)}{\Gamma((p-1)/2)\Gamma(1/2)} (1-x^2)^{(p-3)/2}.
\end{align*}
For any integer $k\ge 1$, we have $\E [\xi^{2k-1}] = 0$ and \begin{align*}
    \E [\xi ^{2k}] = \frac{(2k-1)!!}{p(p+2)\cdots (p+2k-2)}\le \frac{(2k-1)!!}{p^{k}}.
\end{align*}
Consequently, for any $c\in \mathbb R$, \begin{align*}
    \E[\exp(c \xi)] = \sum_{k=0}^\infty \frac{c^{2k}}{(2k)!} \E[\xi^{2k}]\le \sum_{k=0}^\infty \frac{c^{2k}}{(2k)!}\cdot \frac{(2k-1)!!}{p^{k}} = \sum_{k=0}^{\infty } \frac{1}{k!}\Big(\frac{c^2}{2p}\Big)^k = \exp\Big(\frac{c^2}{2p}\Big).
\end{align*}
Combining this with~\eqref{eq-chi-2-div-reduce-2-part-0}, we conclude that (recall that $n=\gamma p$) 
\begin{align}
    \mathbb E_{\Qb}\left[ \frac{ \mathrm{d}\Pb_{t,\bm x} }{ \mathrm{d}\Qb } (\bm Y) \frac{ \mathrm{d}\Pb_{t,\bm x'} }{ \mathrm{d}\Qb } (\bm Y) \right]\leq\ & \exp\left( \frac{1}{2p}\left(\frac{\mu(1+t)p}{n}\langle \bm x, \bm x' \rangle\right)^2 \right) = \exp\left( \frac{(1+t)^2\mu^2}{2\gamma n} \langle \bm x, \bm x' \rangle^2 \right) \,.  \label{eq-chi-2-div-reduce-2-part-1}
\end{align}
In addition, for any $1\le \ell \le L$, from \cite[Section~5]{MNS15} we have
\begin{align}\label{eq-chi-2-div-reduce-2-part-2}
    \mathbb E_{\Qb}\left[ \frac{ \mathrm{d}\Pb_{\bm x_{\ell}} }{ \mathrm{d}\Qb } (\bm G_{\ell}) \frac{ \mathrm{d}\Pb_{\bm x_{\ell}'} }{ \mathrm{d}\Qb } (\bm G_{\ell}) \right] = [1+O(n^{-1})] \exp\left( -\frac{\epsilon^2_{\ell}\lambda_{\ell}}{2} - \frac{ (\epsilon^2_{\ell}\lambda_{\ell})^2 }{4} + \frac{ \epsilon_\ell^2 \lambda_\ell }{2} \left( \frac{\langle \bm x_\ell,\bm x_\ell' \rangle}{\sqrt{n}} \right)^2 \right) \,.
\end{align}
Plugging \eqref{eq-chi-2-div-reduce-2-part-1} and \eqref{eq-chi-2-div-reduce-2-part-2} into \eqref{eq-chi-2-div-reduce-1} yields Lemma~\ref{lem-chi-2-divergence-relax}.

\subsection{Proof of Lemma~\ref{lem-Bernoulli-Gaussian-moment-compare}}{\label{subsec:proof-lem-2.5}}

Define $\zeta,\eta_1,\ldots,\eta_L \in \mathbb R^n$ such that for $1 \leq i \leq n$, we have
\begin{align*}
    \big( \zeta(i),\eta_1(i),\ldots,\eta_L(i) \big) \overset{i.i.d.}{\sim} \big( U,V_1,\ldots,V_L \big) \,.
\end{align*}
Then it suffices to show that for any $\alpha,\alpha_1,\ldots,\alpha_L \in \mathbb N$ we have
\begin{align*}
    \mathbb E\left[ \big( \langle \bm x,\bm x' \rangle \big)^{2\alpha} \big( \langle \bm x_1,\bm x_1' \rangle \big)^{2\alpha_1} \ldots \big( \langle \bm x_L,\bm x_L' \rangle \big)^{2\alpha_L} \right] \leq \mathbb E\left[ \big( \langle \zeta,\mathbbm 1_n \rangle \big)^{2\alpha} \big( \langle \eta_1,\mathbbm 1_n \rangle \big)^{2\alpha_1} \ldots \big( \langle \eta_L,\mathbbm 1_n \rangle \big)^{2\alpha_L} \right] \,,
\end{align*}
where $\mathbbm 1_n \in \mathbb{R}^n$ denotes the all-ones vector (i.e., $\mathbbm 1_n (i)=1$ for all $1\le i\le n$). Recall that the tuples $\big(\bm x(i),\bm x'(i),\bm x_1(i),\bm x_1'(i),\ldots,\bm x_L(i),\bm x_L'(i)\big), 1\le i\le n$ are mutually independent (i.e., independent across different indices $i$). Hence,
\begin{align*}
    & \mathbb E\Big[ \big( \langle \bm x,\bm x' \rangle \big)^{2\alpha} \big( \langle \bm x_1,\bm x_1' \rangle \big)^{2\alpha_1} \ldots \big( \langle \bm x_L,\bm x_L' \rangle \big)^{2\alpha_L} \Big] \\
    =\ & \sum_{ \substack{ \beta,\beta_1,\ldots,\beta_L \in \mathbb N^n \\ |\beta|=2\alpha,|\beta_\ell|=2\alpha_\ell } } \prod_{1 \leq i \leq n} \mathbb E\Big[ \big( \bm x(i)\bm x'(i) \big)^{\beta(i)} \big( \bm x_1(i)\bm x_1'(i) \big)^{\beta_1(i)} \ldots \big( \bm x_L(i)\bm x_L'(i) \big)^{\beta_L(i)} \Big] \,.
\end{align*}
Similarly, we have 
\begin{align*}
    & \mathbb E\Big[ \big( \langle \zeta,\mathbbm 1 \rangle \big)^{2\alpha} \big( \langle \eta_1,\mathbbm 1 \rangle \big)^{2\alpha_1} \ldots \big( \langle \zeta_L,\mathbbm 1 \rangle \big)^{2\alpha_L} \Big] \\
    =\ & \sum_{ \substack{ \beta,\beta_1,\ldots,\beta_L \in \mathbb N^n \\ |\beta|=2\alpha,|\beta_\ell|=2\alpha_\ell } } \prod_{1 \leq i \leq n} \mathbb E\Big[ \big( \zeta(i) \big)^{\beta(i)} \big( \eta_1(i) \big)^{\beta_1(i)} \ldots \big( \eta_L(i) \big)^{\beta_L(i)} \Big] \,.
\end{align*}
Thus, it suffices to show that for all $1 \leq i \leq n$ and all $\beta,\beta_1,\ldots,\beta_L \in \mathbb N$, we have
\begin{align}
    0 \leq \mathbb E\Big[ \big( \bm x(i)\bm x'(i) \big)^{\beta} \big( \bm x_1(i)\bm x_1'(i) \big)^{\beta_1} \ldots \big( \bm x_L(i)\bm x_L'(i) \big)^{\beta_L} \Big] \leq \mathbb E\Big[ \big( \zeta(i) \big)^{\beta} \big( \eta_1(i) \big)^{\beta_1} \ldots \big( \eta_L(i) \big)^{\beta_L} \Big] \,.  \label{eq-compare-Berboulli-Gaussian-moment-reduced}
\end{align}
Recall from \eqref{eq-def-x-ell} that $\bm x_\ell(i) \bm x'_{\ell}(i) = \bm x(i) \bm x'(i) \bm z_\ell(i) \bm z_\ell'(i)$, where $\{ \bm x(i)\bm x'(i), \bm z_1(i) \bm z_1'(i), \ldots, \bm z_L(i) \bm z_L'(i) \}$ are independent variables in $\{ -1,+1 \}$ with (also recall \eqref{eq-def-z-ell})
\begin{align*}
    \Pb\big( \bm x(i)\bm x'(i)=1 \big) =\frac{1}{2}, \quad \Pb\big( \bm z_\ell(i) \bm z_\ell'(i)=1 \big)=\frac{ 1+\rho^2 }{ 2 } \,.
\end{align*}
Thus we have
\begin{align}
    &\nonumber \mathbb E\left[ \big( \bm x(i)\bm x'(i) \big)^{\beta} \big( \bm x_1(i)\bm x_1'(i) \big)^{\beta_1} \ldots \big( \bm x_L(i)\bm x_L'(i) \big)^{\beta_L} \right] \\ 
    \nonumber =\ & \mathbb E\Big[\big( \bm x(i)\bm x'(i) \big)^{\beta+\beta_1+\cdots+\beta_L} \Big] \cdot \prod_{k=1}^L \mathbb E \Big[\big( \bm z_k(i) \bm z_k'(i) \big)^{\beta_k} \Big] \\
    \nonumber =\ & \mathbf 1_{ \beta+\beta_1+\ldots+\beta_L \text{ is even} } \cdot \prod_{k=1}^L\big(\mathbf{1}_{\beta_k\text{ is even}}+ \rho^2 \mathbf{1}_{\beta_k\text{ is odd}} \big) \\ 
    =\ &\mathbf 1_{ \beta+\beta_1+\ldots+\beta_L \text{ is even} } \cdot \rho^{2N_{\operatorname{odd}}}\,,  \label{eq:Bernoulli-moment-reduced}
\end{align}
where $N_{\operatorname{odd}} = \#\{k:\beta_k\text{ is odd}\}$. Now we turn to the Gaussian counterpart. Since we have $\eta_\ell(i)=\rho^2 \zeta(i)+\sqrt{1-\rho^4} \zeta_\ell(i)$, where $\zeta(i),\zeta_1(i),\ldots,\zeta_\ell(i)$ are i.i.d.\ standard normal variables, it is straightforward to check that for all $\beta,\beta_1,\ldots,\beta_L \in \mathbb N$ we have
\begin{align}
    \label{eq:gaussian-moment-zetai}&\mathbb E\Big[ \zeta(i)^2 \cdot \big( \zeta(i) \big)^{\beta} \big( \zeta_1(i) \big)^{\beta_1} \ldots \big( \zeta_L(i) \big)^{\beta_L} \Big] \geq \mathbb E\Big[ \zeta(i)^2 \Big] \mathbb E\Big[ \big( \zeta(i) \big)^{\beta} \big( \zeta_1(i) \big)^{\beta_1} \ldots \big( \zeta_L(i) \big)^{\beta_L} \Big] \geq 0 \,; \\\label{eq:gaussian-moment-zetaelli}
    &\mathbb E\Big[ \zeta_{\ell}(i)^2 \cdot \big( \zeta(i) \big)^{\beta} \big( \zeta_1(i) \big)^{\beta_1} \ldots \big( \zeta_L(i) \big)^{\beta_L} \Big] \geq \mathbb E\Big[ \zeta(i)^2 \Big] \mathbb E\Big[ \big( \zeta(i) \big)^{\beta} \big( \zeta_1(i) \big)^{\beta_1} \ldots \big( \zeta_L(i) \big)^{\beta_L} \Big] \geq 0 \,.
\end{align}
Indeed, since $\zeta(i),\zeta_1(i),\ldots,\zeta_L(i)$ are independent, the expectations factorize; moreover, for $\zeta\sim\mathcal N(0,1)$ one has $\E[\zeta^{\beta}]=0$ for odd $\beta$ and $\E[\zeta^{\beta}]\ge 0$ for even $\beta$. In particular, $\E[\zeta^{\beta+2}]=(\beta+1)\E[\zeta^{\beta}]\ge \E[\zeta^{2}]\,\E[\zeta^{\beta}]$ as $\E[\zeta^{2}]=1$, which yields the displayed inequalities. Recall that $\eta_\ell(i)=\rho^2 \zeta(i)+\sqrt{1-\rho^4} \zeta_\ell(i)$ for any $1\le \ell\le L$ and $\zeta(i),\zeta_1(i),\ldots,\zeta_\ell(i)$ are i.i.d.\ standard normal variables. Expanding each $(\eta_\ell(i))^{\beta_\ell}$ via the binomial theorem expresses $\zeta(i)^\beta\prod_{\ell=1}^L (\eta_\ell(i))^{\beta_\ell}$ as a finite sum of monomials in the independent variables $\zeta(i),\zeta_1(i),\ldots,\zeta_L(i)$ with nonnegative coefficients. Applying the previously established inequalities~\eqref{eq:gaussian-moment-zetai} and~\eqref{eq:gaussian-moment-zetaelli} term-by-term and summing up yields the following inequalities for $\eta_\ell(i)$:
\begin{align*}
    &\mathbb E\Big[ \zeta(i)^2 \cdot \big( \zeta(i) \big)^{\beta} \big( \eta_1(i) \big)^{\beta_1} \ldots \big( \eta_L(i) \big)^{\beta_L} \Big] \geq \mathbb E\Big[ \zeta(i)^2 \Big] \mathbb E\Big[ \big( \zeta(i) \big)^{\beta} \big( \eta_1(i) \big)^{\beta_1} \ldots \big( \eta_L(i) \big)^{\beta_L} \Big] \geq 0 \,; \\
    &\mathbb E\Big[ \eta_{\ell}(i)^2 \cdot \big( \zeta(i) \big)^{\beta} \big( \eta_1(i) \big)^{\beta_1} \ldots \big( \eta_L(i) \big)^{\beta_L} \Big] \geq \mathbb E\Big[ \eta_\ell(i)^2 \Big] \mathbb E\Big[ \big( \zeta(i) \big)^{\beta} \big( \eta_1(i) \big)^{\beta_1} \ldots \big( \eta_L(i) \big)^{\beta_L} \Big] \geq 0 \,.
\end{align*}
Indeed, this also follows from Wick's formula (see, e.g., \cite[Theorem 1.28]{J97}): expanding the relevant Gaussian moments as sums over pairings, the pairing that matches the two extra copies of $\zeta(i)$ (resp.\ $\eta_\ell(i)$) contributes $\E[\zeta(i)^2]\E[\cdots]$ (resp.\ $\E[\eta_\ell(i)^2]\E[\cdots]$), while all remaining pairings are products of (nonnegative) covariances and hence are nonnegative, yielding the desired lower bounds. Let $\lfloor x\rfloor$ be the largest integer not exceeding $x$. Since $\E [\eta_1(i)^2]=1$, and hence we can reduce the exponent by $2$ repeatedly:
\begin{align*}
    \mathbb E\Big[ \big( \zeta(i) \big)^{\beta} \big( \eta_1(i) \big)^{\beta_1} \ldots \big( \eta_L(i) \big)^{\beta_L} \Big]\ge \mathbb E\Big[ \big( \zeta(i) \big)^{\beta} \big( \eta_1(i) \big)^{\beta_1-2\lfloor \beta_1/2\rfloor} \ldots \big( \eta_L(i) \big)^{\beta_L} \Big].
\end{align*}
Similarly, we have 
\begin{align*}
    &\mathbb E\Big[ \big( \zeta(i) \big)^{\beta} \big( \eta_1(i) \big)^{\beta_1} \ldots \big( \eta_L(i) \big)^{\beta_L} \Big] \\
    \geq\ & \mathbb E\Big[ \big( \zeta(i) \big)^{\beta} \big( \eta_1(i) \big)^{\beta_1-2\lfloor\beta_1/2\rfloor} \ldots \big( \eta_L(i) \big)^{\beta_L-2\lfloor\beta_L/2\rfloor} \Big] =  \mathbb E\Big[ \big( \zeta(i) \big)^{\beta} \prod_{ 1 \leq \ell \leq L: \beta_{\ell} \text{ is odd} } \eta_\ell(i) \Big] \\
    \overset{\mathrm{(a)}}{\geq}\ & \mathbb E\Big[ \big( \zeta(i) \big)^{\beta} \prod_{ 1 \leq \ell \leq L: \beta_{\ell} \text{ is odd} } \big( \rho^2 \zeta(i) \big) \Big] \overset{\mathrm{(b)}}{\geq} \mathbf 1_{ \beta+\beta_1+\ldots+\beta_L \text{ is even} } \cdot \rho^{2N_{\mathrm{odd}}} \,,
\end{align*}
where $\mathrm{(a)}$ follows from $\eta_\ell(i) = \rho^2 \zeta(i)+\sqrt{1-\rho^4}\zeta_\ell(i)$ for all $1\le \ell\le L$ and~\eqref{eq:gaussian-moment-zetai}; $\mathrm{(b)}$ follows from $\E \big[ (\zeta(i))^{2k}\big]\ge 1$ for $\zeta(i)\sim \mathcal N(0,1)$ and any integer $k$. Combining this with~\eqref{eq:Bernoulli-moment-reduced}, we complete the proof of~\eqref{eq-compare-Berboulli-Gaussian-moment-reduced}, which in turn concludes the proof of Lemma~\ref{lem-Bernoulli-Gaussian-moment-compare}.

\subsection{Proof of Lemma~\ref{lem-behavior-Pb-Qb}}{\label{subsec:proof-lem-2.7}}

We first prove \eqref{eq-behavior-Qb}. Recall \eqref{eq-behavior-Y-1,2-Qb}, under $\Qb_{\bullet}$ we have $\bm Y_2=\bm Z_2$ which is a standard Gaussian matrix independent with $(\bm Y_1,\bm G_1,\ldots,\bm G_L)$ (and thus also independent with $\mathcal X$). 
Let $\bm Z_2 := (\bm Z_{2,1},\ldots,\bm Z_{2,p})$, where $\{\bm Z_{2,i}\}_{i=1}^p$ follows independent multivariate Gaussian distribution $\mathcal N(0, \mathbb I_n)$.
Thus, conditioned on $\mathcal X$ we have
\begin{align*}
     \mathbb E_{\Qb_{\bullet}}\left[ \big\langle \mathcal X, \bm Y_2 \bm Y_2^{\top} - p \mathbb I_n \big\rangle^2 \mid \mathcal X \right] &= \mathbb E_{\bm Z_2}\Big[ \big\langle \mathcal X, \bm Z_2 \bm Z_2^{\top} - p \mathbb I_n \big\rangle^2 \Big] \\&= \E_{\bm Z_2} \Big[ \big\langle \mathcal X, \sum_{i=1}^p(\bm Z_{2,i} \bm Z_{2,i}^{\top} -  \mathbb I_n) \big\rangle^2 \Big] = p \mathbb E_{\bm Z_2}\Big[ \big\langle \mathcal X, \bm Z_{2,1} \bm Z_{2,1}^{\top} - \mathbb I_n \big\rangle^2 \Big]\,.
\end{align*}
We note that  \begin{align*}
    \big\langle \mathcal X, \bm Z_{2,1} \bm Z_{2,1}^{\top} - \mathbb I_n \big\rangle = \bm Z_{2,1}^{\top} \mathcal X \bm Z_{2,1} - \mathsf{tr}(\mathcal X) = \bm Z_{2,1}^{\top} \left( \frac{\mathcal X+\mathcal X^\top}{2} \right) \bm Z_{2,1} - \mathsf{tr}\left(\frac{\mathcal X+\mathcal X^\top}{2}\right)
\end{align*}
and $\E_{\bm Z_2} \Big[\big\langle \mathcal X, \bm Z_{2,1} \bm Z_{2,1}^{\top} - \mathbb I_n \big\rangle\Big] = 0$. Therefore, 
\begin{align}
    \nonumber\mathbb E_{\Qb_{\bullet}}\Big[ \big\langle \mathcal X, \bm Y_2 \bm Y_2^{\top} - p \mathbb I_n \big\rangle^2 \mid \mathcal X \Big] &= p \E_{\bm Z_2}\left[\operatorname{Var}\left(\bm Z_{2,1}^{\top} \left(\frac{\mathcal X+\mathcal X^\top}{2}\right) \bm Z_{2,1}\right)\right] \\
    &\overset{(a)}{=} 2p\left\lVert\frac{\mathcal X+\mathcal X^\top}{2} \right\rVert_{\Fop}^2\leq 2p \| \mathcal X \|_{\Fop}^2 \,,  \label{eq-simple-2nd-moment-calculus}
\end{align}
where $\mathrm{(a)}$ follows from the variance formula for Gaussian quadratic forms (see, e.g., \cite[Theorem 12.12]{MN19}).
Thus, from a simple Chebyshev's inequality we see that \eqref{eq-behavior-Qb} holds. 

We then prove \eqref{eq-behavior-Pb}. Recall \eqref{eq-behavior-Y-1,2-Pb}, we can decompose $\langle \mathcal X, \bm Y_2 \bm Y_2^{\top} - p \mathbb I_n \rangle$ into the following terms:
\begin{align}
    \big\langle \mathcal X, \bm Y_2 \bm Y_2^{\top} - p \mathbb I_n \big\rangle &= \frac{ \mu(1+\kappa^2) \| u \|^2  }{ n(1+\kappa^{-2}) } \cdot \big\langle \mathcal X, \bm x \bm x^{\top} \big\rangle \label{eq-behavior-Pb-part-1}  \\
    &+ 2 \frac{ \sqrt{ \mu(1+\kappa^2) } }{ \sqrt{n(1+\kappa^{-2})} } \cdot \big\langle \mathcal X, \bm x \bm u^{\top} \bm Z_2^{\top} \big\rangle \label{eq-behavior-Pb-part-2} \\
    &+ \big\langle \mathcal X, \bm Z_2 \bm Z_2^{\top} - p \mathbb I_n \big\rangle \,.  \label{eq-behavior-Pb-part-3}
\end{align}
Note that $\widetilde{\Pb}_{\bullet}=\widetilde{\Pb}(\cdot\mid\mathcal E_{\diamond})$, where $\mathcal E_{\diamond}$ is defined in \eqref{eq-def-E-diamond}. Thus, under $\widetilde{\Pb}_{\bullet}$ we always have $\| \bm u \|^2 \geq (1-\iota)p=(1-\iota)n/\gamma$. Combined with \eqref{eq-generate-estimator}, we see that  
\begin{align}
    & \widetilde{\Pb}_{\bullet}\left( \eqref{eq-behavior-Pb-part-1} \geq \frac{ c\mu(1-\iota)(1+\kappa^2)n }{ 2\gamma(1+\kappa^{-2}) } \cdot \| \mathcal X \|_{\Fop}  \right) \nonumber \\
    =\ &\widetilde{\Pb}_{\bullet}\left( \eqref{eq-behavior-Pb-part-1} \geq \frac{c \mu(1-\iota)(1+\kappa^2) }{2 \gamma(1+\kappa^{-2}) } \cdot \| \mathcal X \|_{\Fop} \| \bm x \bm x^{\top} \|_{\Fop} \right) \nonumber \\
    \ge \ & \widetilde{\Pb}_{\bullet}\left( \frac{ \mu(1+\kappa^2)  }{ n(1+\kappa^{-2}) }\cdot \frac{(1-\iota)n}{\gamma} \cdot \big\langle \mathcal X, \bm x \bm x^{\top} \big\rangle \geq \frac{c \mu(1-\iota)(1+\kappa^2) }{2 \gamma(1+\kappa^{-2}) } \cdot \| \mathcal X \|_{\Fop}\| \bm x \bm x^{\top} \|_{\Fop} \right) \nonumber \\
    \overset{\eqref{eq-generate-estimator}}{\ge}\ & \frac{c}{2} - o(1) \,.  \label{eq-behavior-Pb-part-1-bound}
\end{align}
In addition, using \eqref{eq-simple-2nd-moment-calculus} and a Chebyshev's inequality, we see that
\begin{align}
   \widetilde{\Pb}_{\bullet}\left( |\eqref{eq-behavior-Pb-part-3}| \geq \frac{c \mu(1-\iota)(1+\kappa^2)n }{ 8\gamma(1+\kappa^{-2}) } \cdot \| \mathcal X \|_{\Fop} \right) \leq n^{-\Omega(1)} \,.   \label{eq-behavior-Pb-part-3-bound}
\end{align}
Finally, note that conditioned on $\mathcal X$, we have $\langle \mathcal X, \bm x \bm u^{\top} \bm Z_2^{\top} \rangle \sim \mathcal N(0, \| \bm u \bm x^{\top} \mathcal X \|_{\Fop}^2 )$. Again, since under $\widetilde{\Pb}_{\bullet}$ we always have $\| \bm u \|^2 \leq (1+\iota)p=(1+\iota)n/\gamma$, we have
\begin{align*}
    \big\| \bm u \bm x^{\top} \mathcal X \big\|_{\Fop} \leq \|\bm u \| \| \bm x \| \| \mathcal X \|_{\Fop} \leq \sqrt{(1+\iota)/\gamma}\cdot n \| \mathcal X \|_{\Fop} \,.
\end{align*}
Thus, we have
\begin{align}
  \widetilde{\Pb}_{\bullet}\left( |\eqref{eq-behavior-Pb-part-2}| \geq \frac{c \mu(1-\iota)(1+\kappa^2)n }{ 8\gamma(1+\kappa^{-2}) } \cdot \| \mathcal X \|_{\Fop} \right) \leq n^{-\Omega(1)} \,.   \label{eq-behavior-Pb-part-2-bound}
\end{align}
Combining \eqref{eq-behavior-Pb-part-1-bound}, \eqref{eq-behavior-Pb-part-3-bound}, and \eqref{eq-behavior-Pb-part-2-bound}, we have that 
\begin{align*}
    &\widetilde{\Pb}_{\bullet}\left( \big\langle \mathcal X, \bm Y_2 \bm Y_2^{\top} - p \mathbb I_n \big\rangle \geq \frac{ c\mu(1-\iota)(1+\kappa^2)n }{ 4(1+\iota)(1+\kappa^{-2}) } \| \mathcal X \|_{\Fop} \right) \\
    \ge\ &  \widetilde{\Pb}_{\bullet}\left(\eqref{eq-behavior-Pb-part-1} \geq \frac{ c\mu(1-\iota)(1+\kappa^2)n }{ 2\gamma(1+\kappa^{-2}) } \cdot \| \mathcal X \|_{\Fop} \right) \\
    +\ & \widetilde{\Pb}_{\bullet}\left( |\eqref{eq-behavior-Pb-part-3}| < \frac{c \mu(1-\iota)(1+\kappa^2)n }{ 8\gamma(1+\kappa^{-2}) } \cdot \| \mathcal X \|_{\Fop} \right) \\ 
    +\ &\widetilde{\Pb}_{\bullet}\left( |\eqref{eq-behavior-Pb-part-2}| < \frac{c \mu(1-\iota)(1+\kappa^2)n }{ 8\gamma(1+\kappa^{-2}) } \cdot \| \mathcal X \|_{\Fop} \right)\ge \frac{c}{2}-o(1)\,. 
\end{align*}
We see that \eqref{eq-behavior-Pb} holds. This completes the proof of Lemma~\ref{lem-behavior-Pb-Qb}.

\section{Supplementary proofs in Section~\ref{sec:efficient-algs}}{\label{sub:supp-proofs-sec-3}}

\subsection{Proof of Lemma~\ref{lem-bound-beta-mathcal-H}}{\label{subsec:proof-lem-3.2}}

For all $[H] \in \mathcal H$, let $H \subset\mathsf K_{n,p}$ be a representative of $[H]$ and list $V^{\mathsf a}(H)$ in the counterclockwise order $V^{\mathsf a}(H)=\{ a_{v_1},\ldots,a_{v_{\aleph}} \}$. We construct a graph $\widetilde{H}$ as follows. Define $V(\widetilde{H})=V^{\mathsf a}(H)$. In addition, for each $1 \leq i \leq \aleph$ (below we use the convention that $a_{v_{\aleph+1}}=a_{v_1}$)
\begin{itemize}
    \item If $(a_{v_i},a_{v_{i+1}}) \in E(H)$, then $\chi_H((a_{v_i},a_{v_{i+1}}))\in \{ 1,\ldots,L \}$. We let $(a_{v_i},a_{v_{i+1}}) \in E(\widetilde H)$ and let $\gamma_{\widetilde H}((a_{v_i},a_{v_{i+1}}))=\chi_H((a_{v_i},a_{v_{i+1}}))$;
    \item If $(a_{v_i},a_{v_{i+1}}) \not\in E(H)$, then there exists $b_{u_i} \in V^{\mathsf b}$ such that $(a_{v_i},b_{u_i}),(a_{v_{i+1}},b_{u_i}) \in E_{0}(H)$. We let $(a_{v_i},a_{v_{i+1}}) \in E(\widetilde H)$ and let $\chi_{\widetilde H}((a_{v_i},a_{v_{i+1}}))=0$.
\end{itemize}
Let $\widetilde{\mathcal H}(\aleph)$ be the collection of unlabeled $[\widetilde H]=(V(\widetilde H),E(\widetilde H),\gamma_{\widetilde H})$ such that $(V(\widetilde H),E(\widetilde H))$ is an unlabeled cycle with length $\aleph$ and $\chi_{\widetilde H}$ maps $E(\widetilde H)$ to $\{ 0,1,\ldots,L \}$. It is clear that
\begin{equation*}
    \varphi: \mathcal H(\aleph) \to \widetilde{\mathcal H}(\aleph) \quad \varphi(H)=\widetilde{H}
\end{equation*}
is a bijection between $\mathcal H$ and $\widetilde{\mathcal H}$. In addition, it is clear that
\begin{align*}
    & |\mathsf{Aut}(H)|=|\mathsf{Aut}(\varphi(H))|, \quad |\mathsf{dif}_0(H)|=|\mathsf{dif}_0(\varphi(H))|, \quad |\mathsf{dif}(H)|=|\mathsf{dif}(\varphi(H))| \,, \\ 
    & |E_0(H)|=2|E_0(\varphi(H))|, \quad |E_{\ell}(H)|=|E_{\ell}(\varphi(H))| \mbox{ for } 1 \leq \ell \leq L \,.
\end{align*}
Thus, we have
\begin{align}
    \beta_{\mathcal H} &\overset{\eqref{eq-def-Xi(H)},\eqref{eq-def-beta-mathcal-H}}{=} \sum_{ [H] \in \mathcal H } \frac{ \rho^{2|\mathsf{dif}_0(H)|+4|\mathsf{dif}(H)|} (\mu^2/\gamma)^{\frac{1}{2}|E_0(H)|} (\epsilon^2_1\lambda_1)^{|E_1(H)|} \ldots (\epsilon^2_L \lambda_L)^{|E_L(H)|} }{ |\mathsf{Aut}(H)| } \nonumber \\
    &= \sum_{ [H] \in \widetilde{\mathcal H} } \frac{ \rho^{2|\mathsf{dif}_0(H)|+4|\mathsf{dif}(H)|} (\mu^2/\gamma)^{|E_0(H)|} (\epsilon^2_1\lambda_1)^{|E_1(H)|} \ldots (\epsilon^2_L \lambda_L)^{|E_L(H)|} }{ |\mathsf{Aut}(H)| } \,.  \label{eq-beta-mathcal-H-reduce-1}
\end{align}
We now bound the right-hand side of \eqref{eq-beta-mathcal-H-reduce-1}. Define $\mathcal W=\mathcal W(\aleph)$ to be the set of $\omega=(\omega_1,\ldots,\omega_{\aleph}) \in \{ 0,1,\ldots,L \}^{\aleph}$. In addition, for $\omega=(\omega_1,\ldots,\omega_{\aleph})\in\mathcal W$, we similarly write
\begin{align*}
    & E_{\ell}(\omega)=\big\{ 1\leq i\leq \aleph: \omega_i=\ell \big\} \mbox{ for } 1 \leq \ell \leq L \,; \\
    & \mathsf{dif}_0(\omega)=\big\{ 1 \leq i \leq \aleph-1: \{ \omega_i,\omega_{i+1} \}=\{ 0,\ell \} \mbox{ for some } \ell\neq 0 \big\} \,; \\
    & \mathsf{dif}(\omega)=\big\{ 1 \leq i \leq \aleph-1: \{ \omega_i,\omega_{i+1} \}=\{ \ell,\ell' \} \mbox{ for some } 1 \leq \ell<\ell' \leq L \big\} \,.
\end{align*}
Define 
\begin{equation}{\label{eq-def-widetilde-mathcal-H-bullet}}
    \widetilde{\mathcal H}_{\bullet}=\big\{ (v,[H]): [H] \in \mathcal H, v \in V(H) \big\} \,.
\end{equation}
Consider the mapping $\varphi:\widetilde{\mathcal H}_{\bullet} \to \mathcal W$ defined as follows: for all $(v,H) \in \widetilde{\mathcal H}_{\bullet}$, write $V(H)$ in the counterclockwise order $V(H)=\{ v_1,\ldots,v_{\aleph} \}$ with $v_1=v$ and let $\varphi(v,H)=\omega=(\omega_1,\ldots,\omega_\aleph)$ with $\omega_i=\chi_H((v_i,v_{i+1}))$ (we write $v_{\ell+1}=v_1$). It is clear that $\varphi$ is a bijection and
\begin{align*}
    &|E_{\ell}(H)|=|E_{\ell}(\varphi(v,H))| \mbox{ for } 0 \leq \ell \leq L \,, \\ 
    &|\mathsf{dif}(H)|-1 \leq |\mathsf{dif}(\varphi(v,H))| \leq |\mathsf{dif}(H)| \,, \\
    &|\mathsf{dif}_0(H)|-1 \leq |\mathsf{dif}_0(\varphi(v,H))| \leq |\mathsf{dif}_0(H)| \,.
\end{align*}
In addition, since $1 \leq |\mathsf{Aut}(H)| \leq \aleph$ for all $H \in \widetilde{\mathcal H}$, we have
\begin{align}
    \eqref{eq-beta-mathcal-H-reduce-1} &\leq \sum_{ [H] \in \widetilde{\mathcal H} } \rho^{2|\mathsf{dif}_0(H)|+4|\mathsf{dif}(H)|} (\mu^2/\gamma)^{|E_0(H)|} (\epsilon^2_1\lambda_1)^{|E_1(H)|} \ldots (\epsilon^2_L \lambda_L)^{|E_L(H)|}\nonumber \\
    &\leq \sum_{ (v,[H]) \in \widetilde{\mathcal H}_{\bullet} } \rho^{2|\mathsf{dif}_0(\varphi(v,H))|+4|\mathsf{dif}(\varphi(v,H))|} (\mu^2/\gamma)^{|E_0(\varphi(v,H))|} (\epsilon^2_1\lambda_1)^{|E_1(\varphi(v,H))|} \ldots (\epsilon^2_L \lambda_L)^{|E_L(\varphi(v,H))|}\nonumber \\\label{eq-beta-mathcalH-upbd}
    &= \sum_{ \omega\in\mathcal W } \rho^{2|\mathsf{dif}_0(\omega)|+4|\mathsf{dif}(\omega)|} (\mu^2/\gamma)^{|E_0(\omega)|} (\epsilon^2_1\lambda_1)^{|E_1(\omega)|} \ldots (\epsilon^2_L \lambda_L)^{|E_L(\omega)|} \,;
\end{align}
and also
\begin{align}
    \beta_{\mathcal H} &\geq \frac{1}{\aleph} \sum_{ [H] \in \widetilde{\mathcal H} } \rho^{2|\mathsf{dif}_0(H)|+4|\mathsf{dif}(H)|} (\mu^2/\gamma)^{|E_0(H)|} (\epsilon^2_1\lambda_1)^{|E_1(H)|} \ldots (\epsilon^2_L \lambda_L)^{|E_L(H)|} \nonumber \\
    &\overset{\mathrm{(a)}}{\ge} \frac{\rho^6}{\aleph^2} \sum_{ (v,[H]) \in \widetilde{\mathcal H}_{\bullet} } \rho^{2|\mathsf{dif}_0(\varphi(v,H))|+4|\mathsf{dif}(\varphi(v,H))|} (\mu^2/\gamma)^{|E_0(\varphi(v,H))|} (\epsilon^2_1\lambda_1)^{|E_1(\varphi(v,H))|} \ldots (\epsilon^2_L \lambda_L)^{|E_L(\varphi(v,H))|} \nonumber \\\label{eq-beta-mathcalH-lwbd}
    &\geq \frac{\rho^6}{\aleph^2} \sum_{ \omega\in\mathcal W } \rho^{2|\mathsf{dif}_0(\omega)|+4|\mathsf{dif}(\omega)|} (\mu^2/\gamma)^{|E_0(\omega)|} (\epsilon^2_1\lambda_1)^{|E_1(\omega)|} \ldots (\epsilon^2_L \lambda_L)^{|E_L(\omega)|} \,,
\end{align}
where $\mathrm{(a)}$ follows from $|\mathsf{dif}(H)|-1 \leq |\mathsf{dif}(\varphi(v,H))|$, $|\mathsf{dif}_0(H)|-1 \leq |\mathsf{dif}_0(\varphi(v,H))|$, and $\rho\le 1$. Thus, it suffices to show that there exists a constant $D=\Theta(1)$ such that
\begin{equation}{\label{eq-goal-lem-3.3}}
    D^{-1} \sigma_+(\mathbf P)^{\aleph} \leq \sum_{ \omega\in\mathcal W } \rho^{2|\mathsf{dif}_0(\omega)|+4|\mathsf{dif}(\omega)|} (\mu^2/\gamma)^{|E_0(\omega)|} (\epsilon^2_1\lambda_1)^{|E_1(\omega)|} \ldots (\epsilon^2_L \lambda_L)^{|E_L(\omega)|} \leq D \sigma_+(\mathbf P)^{\aleph} \,.
\end{equation}
We now give an inductive formula of 
\begin{align*}
    \sum_{ \omega\in\mathcal W(\aleph) } \rho^{2|\mathsf{dif}_0(\omega)|+4|\mathsf{dif}(\omega)|} (\mu^2/\gamma)^{|E_0(\omega)|} \prod_{1 \leq \ell \leq L} (\epsilon^2_\ell \lambda_\ell)^{|E_\ell(\omega)|} \,.
\end{align*}
Denote
\begin{equation}{\label{eq-def-mathsf-X-Y}}
\begin{aligned}
    &\mathsf{X}(\aleph) = \sum_{ \omega\in\mathcal W(\aleph), \omega_\aleph=0 } \rho^{2|\mathsf{dif}_0(\omega)|+4|\mathsf{dif}(\omega)|} (\mu^2/\gamma)^{|E_0(\omega)|} \prod_{1 \leq \ell \leq L} (\epsilon^2_\ell \lambda_\ell)^{|E_\ell(\omega)|}  \,; \\
    &\mathsf{Y}_\ell(\aleph) = \sum_{ \omega\in\mathcal W(\aleph), \omega_\aleph=\ell } \rho^{2|\mathsf{dif}_0(\omega)|+4|\mathsf{dif}(\omega)|} (\mu^2/\gamma)^{|E_0(\omega)|} \prod_{1 \leq \ell \leq L} (\epsilon^2_\ell \lambda_\ell)^{|E_\ell(\omega)|} \mbox{ for } 1 \leq \ell \leq L \,.
\end{aligned}
\end{equation}
For any $\omega\in\mathcal W(\aleph)$ with $\omega_{\aleph}=0$, $\omega = (\omega',0)$ with $\omega'\in \mathcal W(\aleph-1)$. There are two cases:
\begin{itemize}
    \item Case 1. $\omega_{\aleph-1} = 0$. We have $|\mathsf{dif}_0(\omega)| = |\mathsf{dif}_0(\omega')|$, $|\mathsf{dif}(\omega)| = |\mathsf{dif}(\omega')|$, $|E_0(\omega)| = |E_0(\omega')|+1$, and $|E_\ell(\omega)| = |E_\ell(\omega')|$ for all $1\le \ell \le L$. This contributes to a $\tfrac{\mu^2}{\gamma}$ factor.
    \item Case 2. $1\le \omega_{\aleph-1} = \ell\le L$. We have $|\mathsf{dif}_0(\omega)| = |\mathsf{dif}_0(\omega')|+1$, $|\mathsf{dif}(\omega)| = |\mathsf{dif}(\omega')|$, $|E_0(\omega)| = |E_0(\omega')|+1$, and $|E_\ell(\omega)| = |E_\ell(\omega')|$ for all $1\le \ell \le L$. This contributes to a $\tfrac{\rho^2\mu^2}{\gamma}$ factor.
\end{itemize}
Consequently, we have 
\begin{align*}
    \mathsf{X}(\aleph) &= \frac{\mu^2}{\gamma} \cdot \mathsf{X}(\aleph-1) + \sum_{1 \leq \ell \leq L} \frac{\rho^2 \mu^2}{\gamma} \cdot  \mathsf{Y}_{\ell}(\aleph-1) \,.
\end{align*}
Similarly, for any $1\le \ell \le L$,
\begin{align*}
    \mathsf{Y}_{\ell}(\aleph) = \epsilon^2_\ell \lambda_\ell \cdot \mathsf{Y}_{\ell}(\aleph-1) + \rho^2 \epsilon^2_\ell \lambda_\ell \cdot \mathsf{X}(\aleph-1) + \sum_{\ell'\neq\ell} \rho^4 \epsilon^2_\ell \lambda_\ell \cdot \mathsf{Y}_{\ell'}(\aleph-1) \,.
\end{align*}
Thus, we have
\begin{align*}
    \begin{pmatrix}
        \mathsf{X}(\aleph) \\ \mathsf Y_{1}(\aleph) \\ \mathsf Y_2(\aleph) \\ \vdots \\ \mathsf Y_{L}(\aleph)
    \end{pmatrix}
    = 
    \begin{pmatrix}
        \frac{\mu^2}{\gamma} & \frac{\rho^2\mu^2}{\gamma} & \frac{\rho^2\mu^2}{\gamma} & \ldots & \frac{\rho^2\mu^2}{\gamma} \\
        \rho^2 \epsilon^2_1 \lambda_1 & \epsilon^2_1 \lambda_1 & \rho^4 \epsilon^2_1 \lambda_1 & \ldots & \rho^4 \epsilon^2_1 \lambda_1 \\
        \rho^2 \epsilon^2_2 \lambda_2 & \rho^4 \epsilon^2_2 \lambda_2 & \epsilon^2_2 \lambda_2 & \ldots & \rho^4 \epsilon^2_2 \lambda_2 \\
        \vdots & \vdots & \vdots & \ddots & \vdots \\
        \rho^2 \epsilon^2_L \lambda_L & \rho^4 \epsilon^2_L \lambda_L & \rho^4 \epsilon^2_L \lambda_L & \ldots & \epsilon^2_L \lambda_L 
    \end{pmatrix}
    \begin{pmatrix}
        \mathsf{X}(\aleph-1) \\ \mathsf Y_{1}(\aleph-1) \\ \mathsf Y_2(\aleph-1) \\ \vdots \\ \mathsf Y_{L}(\aleph-1)
    \end{pmatrix} 
    \overset{\eqref{eq-def-mathbf-P}}{=} \mathbf P 
    \begin{pmatrix}
        \mathsf{X}(\aleph-1) \\ \mathsf Y_{1}(\aleph-1) \\ \mathsf Y_2(\aleph-1) \\ \vdots \\ \mathsf Y_{L}(\aleph-1)
    \end{pmatrix} \,.
\end{align*}
Note that $\mathbf P$ is similar to a positive definite matrix $\widetilde{\mathbf P}=\mathbf U^{\top} \mathbf V \mathbf U$, where
\begin{equation}{\label{eq-def-widetilde-mathbf-P}}
    \mathbf U=
    \begin{pmatrix}
        \sqrt{\frac{\mu^2}{\gamma}} & 0 & 0 & \ldots & 0 \\
        0 & \sqrt{\epsilon^2_1 \lambda_1} & 0 & \ldots & 0 \\
        0 & 0 & \sqrt{\epsilon^2_2 \lambda_2} & \ldots & 0 \\
        \vdots & \vdots & \vdots & \ddots & \vdots  \\
        0 & 0 & 0 & \ldots & \sqrt{\epsilon^2_L \lambda_L}
    \end{pmatrix} ,
    \ \mathbf V=
    \begin{pmatrix}
        1 & \rho^2 & \rho^2 & \ldots & \rho^2 \\
        \rho^2 & 1 & \rho^4 & \ldots & \rho^4 \\
        \rho^2 & \rho^4 & 1 & \ldots & \rho^4 \\
        \vdots & \vdots & \vdots & \ddots & \vdots \\
        \rho^2 & \rho^4 & \rho^4 & \ldots & 1 
    \end{pmatrix} \,.
\end{equation}
It is straightforward to check that $\widetilde{\mathbf P}$ is positive semidefinite. Thus, denote $\sigma_+(\mathbf P) =\sigma_1(\mathbf P)\geq \ldots \geq \sigma_{L+1}(\mathbf P)\geq 0$ be the eigenvalues of $\mathbf P$, standard results yield that there exists $\mathbf D \in \mathbb R^{(L+1)*(L+1)}$ with $\mathbf D_{l,1}>0$ such that
\begin{align}
    \begin{pmatrix}
        \mathsf{X}(\aleph) \\ \mathsf Y_{1}(\aleph) \\ \mathsf Y_2(\aleph) \\ \vdots \\ \mathsf Y_{L}(\aleph)
    \end{pmatrix}
    = \mathbf D 
    \begin{pmatrix}
        \sigma_1(\mathbf P)^{\aleph} \\ \sigma_2(\mathbf P)^{\aleph} \\ \sigma_3(\mathbf P)^{\aleph} \\ \vdots \\ \sigma_{L+1}(\mathbf P)^{\aleph}
    \end{pmatrix}
    \,. \label{eq-formula-X-ell-Y-ell}
\end{align}
Thus, it suffices to show that $\sigma_+(\mathbf P)>1$ if \eqref{eq-assum-upper-bound} holds. To this end, note that $\widetilde{\mathbf P}$ is precisely the covariance matrix of 
\begin{align*}
    \left( \sqrt{\tfrac{\mu^2}{\gamma}} U, \sqrt{\epsilon_1^2\lambda_1} V_1, \ldots, \sqrt{\epsilon^2_L \lambda_L} V_L \right), \mbox{ where } (U,V_1,\ldots,V_L) \mbox{ is defined in } \eqref{eq-Gaussian-covariance-structure} \,.
\end{align*}
For $X\sim \mathcal N(0,\Sigma_X)$ where $\Sigma_X=\widetilde{\mathbf P}$, it is well known that $\E[\exp(X^\top X/2)]<\infty$ if and only if $\lambda_{\max}(\Sigma_X)<1$, and the expectation diverges when $\lambda_{\max}(\Sigma_X)\ge 1$ (see, e.g., \cite[Section~1.4]{M09}).
Consequently,
\begin{align*}
    \sigma_+(\mathbf P)\geq 1 \Longleftrightarrow \sigma_+(\widetilde{\mathbf P})\geq 1 &\Longleftrightarrow \mathbb E\left[ \exp\left( \frac{1}{2}\left( \frac{\mu^2}{\gamma} U^2 + \epsilon_1^2 \lambda_1 V_1^2 + \ldots + \epsilon_L^2 \lambda_L V_L^2 \right) \right) \right] = \infty \\
    &\Longleftrightarrow F( \mu,\rho,\gamma,\{ \lambda_{\ell} \},\{ \epsilon_{\ell} \} )\geq 1 \,,
\end{align*}
where the last transition follows from the proof of Proposition~\ref{prop-bound-chi-2-divergence}. Finally, since both $\sigma_+(\mathbf P)$ and $F( \mu,\rho,\gamma,\{ \lambda_{\ell} \},\{ \epsilon_{\ell} \} )$ are continues and strictly monotone with respect to $( \mu,\rho,\gamma,\{ \lambda_{\ell} \},\{ \epsilon_{\ell} \} )$, we see that $\sigma_+(\mathbf P)>1$ whenever \eqref{eq-assum-upper-bound} holds.

\subsection{Proof of Lemma~\ref{lem-bound-beta-mathcal-J}}{\label{subsec:proof-lem-3.6}}

The proof of Lemma~\ref{lem-bound-beta-mathcal-J} is highly similar to the proof of Lemma~\ref{lem-bound-beta-mathcal-H}, so we will only provide an outline with the main differences while adapting arguments from Lemma~\ref{lem-bound-beta-mathcal-H} without presenting full details. We will only prove the bound regarding $\beta_{\mathcal J}$ and the bound regarding $\beta_{\mathcal J_{\star}}$ can be proved in the same manner. Denote $\widetilde{\mathcal J}$ as the collection of unlabeled $[\widetilde H]=(V(\widetilde H),E(\widetilde H),\gamma_{\widetilde H})$ such that $(V(\widetilde H),E(\widetilde H))$ is an unlabeled path with length $\aleph$ and $\gamma_{\widetilde H}$ maps $E(\widetilde H)$ to $\{ 0,1,\ldots,L \}$. Similarly as in Section~\ref{subsec:proof-lem-3.2}, it suffices to show that
\begin{equation}{\label{eq-lem-3.7-reduce-1}}
    D^{-1} \sigma_+(\mathbf P)^{\aleph} \leq \sum_{ [H] \in \widetilde{\mathcal J} } \frac{ \Xi(H)^2 }{ |\mathsf{Aut}(H)| } \leq D \sigma_+(\mathbf P)^{\aleph} \mbox{ for some } D=\Theta(1) \,.
\end{equation}
Recall that we defined $\mathcal W=\mathcal W(\aleph)$ to be the set of $\omega=(\omega_1,\ldots,\omega_{\aleph}) \in \{ 0,1,\ldots,L \}^{\aleph}$. Also denote 
\begin{equation*}
    \widetilde{\mathcal J}_{\bullet} = \big\{ (v,[H]): [H] \in \widetilde{\mathcal J}, v \in \mathsf L(H) \big\} \,.
\end{equation*}
It is clear that there exists a bijection $\psi:\widetilde{\mathcal J}_{\bullet} \to \mathcal W$ that preserves the quantity $|E_{\ell}(\cdot)|$, $|\mathsf{dif}_0(\cdot)|$ and $|\mathsf{dif}(\cdot)|$. 
Recall~\eqref{eq-beta-mathcalH-upbd} and~\eqref{eq-beta-mathcalH-lwbd}. Since $1\le |\mathsf{Aut}(H)|\le 2$ and $|\mathsf L(H)|\le 2$, we have \begin{align*}
    \frac{\rho^6}{4}\sum_{ \omega\in\mathcal W } \Xi(\omega)^2\le \beta_{\mathcal J}\le \sum_{ \omega\in\mathcal W } \Xi(\omega)^2.
\end{align*}
Recall that $\rho>0$ is a constant. Thus, similarly as \eqref{eq-goal-lem-3.3}, it suffices to show that there exists a constant $D=\Theta(1)$ such that (recall \eqref{eq-def-Xi(H)})
\begin{equation}{\label{eq-goal-lem-3.7}}
    D^{-1} \sigma_+(\mathbf P)^{\aleph} \leq \sum_{ \omega\in\mathcal W } \Xi(\omega)^2 = \mathsf X(\aleph) + \sum_{1 \leq \ell \leq L} \mathsf Y_{\ell}(\aleph) \leq D \sigma_+(\mathbf P)^{\aleph}\,.
\end{equation}
\eqref{eq-goal-lem-3.7} then directly follows from \eqref{eq-formula-X-ell-Y-ell}.

\section{Supplementary proofs in Section~\ref{sec:stat-analysis}}{\label{sec:supp-proofs-sec-4}}

\subsection{Proof of Lemma~\ref{lem-est-cov-f-S-f-K}}{\label{subsec:proof-lem-4.3}}

Recall \eqref{eq-def-overline-G-ell-i,j}. Conditioned on $\bm x,\bm u$ and $\bm x_1,\ldots,\bm x_L$, we have that
\begin{equation*}
    \big\{ \bm Y(i,k):i \in [n],k \in [p] \big\} \bigcup \Big( \cup_{1 \leq \ell \leq L} \big\{ \overline{\bm G}_{\ell}(i,j): (i,j) \in \operatorname{U}_n \big\} \Big)
\end{equation*}
are a collection of conditionally independent random variables, with
\begin{align*}
    &\mathbb E_{\Pb}\big[ \bm Y(i,k) \mid \bm x,\bm u,\bm x_{1},\ldots,\bm x_L \big]= \tfrac{ \sqrt{\mu} }{ \sqrt{n} } \bm x(i) \bm u(k) \,; \\
    &\mathbb E_{\Pb}\big[ \bm Y(i,k)^2 \mid \bm x,\bm u,\bm x_{1},\ldots,\bm x_L \big]= 1+ \tfrac{\mu}{n} \bm x(i)^2 \bm u(k)^2 \,; \\
    &\mathbb E_{\Pb}\big[ \overline{\bm G}_{\ell}(i,j) \mid \bm x,\bm u,\bm x_{1},\ldots,\bm x_L \big]= \tfrac{ \sqrt{\epsilon_{\ell}^2\lambda_{\ell}} }{ \sqrt{n} } \bm x_{\ell}(i) \bm x_{\ell}(j) \,; \\
    &\mathbb E_{\Pb}\big[ \overline{\bm G}_{\ell}(i,j)^2 \mid \bm x,\bm u,\bm x_{1},\ldots,\bm x_L \big]= (1+O(n^{-1}))(1+ \epsilon_{\ell} \bm x_{\ell}(i) \bm x_{\ell}(j)) \,.
\end{align*}
Thus, we have (recall \eqref{eq-def-f-H,A})
\begin{align}
    \mathbb E_{\Pb}\big[ f_S f_K \big] &= \mathbb E_{\bm x,\bm u,\bm x_{1},\ldots,\bm x_L}\mathbb E_{\Pb}\big[ f_S f_K \mid \bm x,\bm u,\bm x_1,\ldots,\bm x_L \big] \nonumber \\
    &= \mathbb E_{\bm x,\bm u,\bm x_{1},\ldots,\bm x_L}\Bigg\{ \prod_{(a_i,b_k)\in E_0(S) \triangle E_0(K)} \tfrac{ \sqrt{\mu} }{ \sqrt{n} } \bm x(i) \bm u(k) \prod_{(a_i,b_k)\in E_0(S) \cap E_0(K)} \big( 1+\tfrac{\mu}{n} \bm x(i)^2 \bm u(k)^2 \big) \nonumber \\
    &\quad \prod_{1 \leq \ell \leq L} \prod_{(a_i,a_j)\in E_{\ell}(S) \triangle E_{\ell}(K) } \tfrac{ \sqrt{\epsilon_{\ell}^2\lambda_{\ell}} }{ \sqrt{n} } \bm x_{\ell}(i) \bm x_{\ell}(j)\nonumber \\
    &\quad \prod_{1 \leq \ell \leq L} \prod_{(a_i,a_j)\in E_{\ell}(S) \cap E_{\ell}(K) } (1+O(n^{-1})\big( 1+ \epsilon_{\ell} \bm x_{\ell}(i) \bm x_{\ell}(j) \big) \bigg\} \nonumber  \,.
\end{align}
We first show \eqref{eq-cov-equal-0-non-intersect}. Note that if $V(S) \cap V(K)=\emptyset$ then $E_{\ell}(S) \cap E_{\ell}(K)=\emptyset$ for all $0 \leq \ell \leq L$. Thus, in this case we have
\begin{align*}
    \mathbb E_{\Pb}\big[ f_S f_K \big] = \mathbb E\Bigg\{ \prod_{(a_i,b_k)\in E_0(S) \cup E_0(K)} \tfrac{ \sqrt{\mu} }{ \sqrt{n} } \bm x(i) \bm u(k) \prod_{1 \leq \ell \leq L} \prod_{(a_i,a_j)\in E_{\ell}(S) \cup E_{\ell}(K) } \tfrac{ \sqrt{\epsilon_{\ell}^2\lambda_{\ell}} }{ \sqrt{n} } \bm x_{\ell}(i) \bm x_{\ell}(j) \bigg\} \,.
\end{align*}
Similar to \eqref{eq-cal-expectation}, we see that
\begin{align*}
    &\mathbb E\Bigg\{ \prod_{(a_i,b_k)\in E_0(S) \cup E_0(K)} \bm x(i) \bm u(k) \prod_{1 \leq \ell \leq L} \prod_{(a_i,a_j)\in E_{\ell}(S) \cup E_{\ell}(K) } \bm x_{\ell}(i) \bm x_{\ell}(j) \bigg\} \\
    =\ & \rho^{|\mathsf{dif}_0(S)|+ |\mathsf{dif}_0(K)| +2|\mathsf{dif}(S) |+ 2|\mathsf{dif}(K)| } \,.
\end{align*}
Thus, when $V(S) \cap V(K)=\emptyset$ we have
\begin{align*}
    \mathbb E_{\Pb}\big[ f_S f_K \big] &= \frac{ \gamma^{ \frac{1}{4}(|E_0(S)|+|E_0(K)|) } }{ n^{ \frac{1}{2}(|E(S)|+|E(K)|) } } \big( \tfrac{\mu^2}{\gamma} \big)^{ \frac{1}{4}(|E_0(S)|+|E_0(K)|) }  \prod_{1 \leq \ell \leq L} \big( \epsilon_{\ell}^2\lambda_{\ell} \big)^{ \frac{1}{2}(|E_\ell(S)|+|E_\ell(K)|) }  \\
    &\quad \cdot \rho^{|\mathsf{dif}_0(S)|+ |\mathsf{dif}_0(K)| +2|\mathsf{dif}(S) |+ 2|\mathsf{dif}(K)| } \\
    &\overset{\eqref{eq-def-Xi(H)}}{=} \frac{ \Xi(S)\Xi(K) }{ n^{ \frac{1}{2}(|E(S)|-\frac{1}{2}|E_0(S)|) } p^{ \frac{1}{4}|E_0(S)| } n^{ \frac{1}{2}(|E(K)|-\frac{1}{2}|E_0(K)|) } p^{ \frac{1}{4}|E_0(K)| } } \\
    &\overset{\eqref{eq-basic-property-mathcal-H}}{=} \frac{ \Xi(S)\Xi(K) }{ n^{ \frac{1}{2}\aleph } p^{ \frac{1}{4}|E_0(S)| } n^{ \frac{1}{2}\aleph } p^{ \frac{1}{4}|E_0(K)| } } \overset{\text{Lemma~}\ref{lem-mean-var-f-H-part-1}}{=} \mathbb E_{\Pb}[f_S] \mathbb E_{\Pb}[f_K] \,.
\end{align*}
Now we turn to \eqref{eq-est-general-cov-f-S-f-K}. 
Recall that $S\Cap K$ denotes the graph induced by the edge set $\cup_{0\le \ell\le L}(E_\ell(S)\cap E_\ell(K))$.
Note that
\begin{align*}
    & \left( \tfrac{\sqrt{\mu}}{\sqrt{n}} \right)^{|E_0(S) \triangle E_0(K)|} \prod_{1 \leq \ell \leq L} \left( \tfrac{ \sqrt{\epsilon_{\ell}^2\lambda_{\ell}} }{\sqrt{n}} \right)^{|E_{\ell}(S)\triangle E_\ell(K)|} \\
    =\ & \frac{ \gamma^{\frac{1}{4}|E_0(S) \triangle E_0(K)|} }{ n^{ \frac{1}{2}(|E(S)|+|E(K)|)-|E(S \Cap K)| } } \left( \frac{\mu^2}{\gamma} \right)^{ \frac{1}{4}|E_0(S) \triangle E_0(K)| } \prod_{1 \leq \ell \leq L} \big( \epsilon_{\ell}^2\lambda_{\ell} \big)^{\frac{1}{2}|E_{\ell}(S)\triangle E_\ell(K)|} \\
    =\ & \frac{ \Omega(S \Cap K) }{ \Omega(S)^{\frac{1}{2}} \Omega(K)^{\frac{1}{2}} } \left( \frac{\mu^2}{\gamma} \right)^{ \frac{1}{4}|E_0(S) \triangle E_0(K)| } \prod_{1 \leq \ell \leq L} \big( \epsilon_{\ell}^2\lambda_{\ell} \big)^{\frac{1}{2}|E_{\ell}(S)\triangle E_\ell(K)|} \,,
\end{align*}
where the second equality follows from $n=\gamma p$ and \eqref{eq-def-Omega(S)}:\begin{align*}
    \frac{ \Omega(S \Cap K) }{ \Omega(S)^{\frac{1}{2}} \Omega(K)^{\frac{1}{2}} }  &= \frac{n^{|E(S\Cap K)|-\tfrac{1}{2}|E_0(S\Cap K)|} (n/\gamma)^{\tfrac{1}{2}|E_0(S\Cap K)|}}{\big(n^{|E(S)|-\tfrac{1}{2}|E_0(S)|} (n/\gamma)^{\tfrac{1}{2}|E_0(S)|}\big)^{1/2}\big(n^{|E(K)|-\tfrac{1}{2}|E_0(K)|} (n/\gamma)^{\tfrac{1}{2}|E_0(K)|}\big)^{1/2}}\\&=\frac{\gamma^{\tfrac{1}{4}(|E_0(S)|+|E_0(K)|-2|E_0(S\Cap K)|)}}{n^{\tfrac{1}{2}(|E(S)|+|E(K)|-2|E(S\Cap K)|)}} =  \frac{ \gamma^{\frac{1}{4}|E_0(S) \triangle E_0(K)|} }{ n^{ \frac{1}{2}(|E(S)|+|E(K)|)-|E(S \Cap K)| } } \,.
\end{align*} 
Since $\aleph = o(\frac{\log n}{\log \log n})$, we have 
\begin{align*}
    \prod_{1 \leq \ell \leq L} \prod_{(a_i,a_j)\in E_{\ell}(S) \cap E_{\ell}(K) } (1+O(n^{-1}))\le (1+O(n^{-1}))^{L\aleph} = 1+o(1) \,.
\end{align*}
Thus, to prove Lemma~\ref{lem-est-cov-f-S-f-K}, it suffices to show that
\begin{align}
    &\mathbb E\Bigg\{ \prod_{(a_i,b_k)\in E_0(S) \triangle E_0(K)} \bm x(i) \bm u(k) \prod_{1 \leq \ell \leq L} \prod_{(a_i,a_j)\in E_{\ell}(S) \triangle E_{\ell}(K) } \bm x_{\ell}(i) \bm x_{\ell}(j) \nonumber \\
    &\quad\quad \prod_{(a_i,b_k)\in E_0(S) \cap E_0(K)} \big( 1+\tfrac{\mu}{n} \bm x(i)^2 \bm u(k)^2 \big) \prod_{(a_i,a_j)\in E_{\ell}(S) \cap E_{\ell}(K) } \big( 1+ \epsilon_{\ell} \bm x_{\ell}(i) \bm x_{\ell}(j) \big) \Bigg\} \label{eq-lem-4.3-goal-LHS} \\
    \leq\ & \frac{1}{2}C^{|\mathsf{L}(S \Cap K)|+|\mathsf{L}(S \cap K)|+1} \rho^{|\mathsf{dif}_0(S) \setminus V(K)|+ |\mathsf{dif}_0(K) \setminus V(S)| +2|\mathsf{dif}(S) \setminus V(K)|+ 2|\mathsf{dif}(K) \setminus V(S)| } \,. \label{eq-lem-4.3-goal-RHS}
\end{align}
Similarly as in the proof of Lemma~\ref{lem-mean-var-f-H-part-1}, we will count the contribution of each vertex but the proof here will be substantially more difficult. Expanding the terms $\big( 1+\tfrac{\mu}{n} \bm x(i)^2 \bm u(k)^2 \big) $ for all $(a_i,b_k)\in E_0(S) \cap E_0(K)$ and $\big( 1+ \epsilon_{\ell} \bm x_{\ell}(i) \bm x_{\ell}(j) \big) $ for all $(a_i,a_j)\in E_{\ell}(S) \cap E_{\ell}(K)$, we can write \eqref{eq-lem-4.3-goal-LHS} as (note that $\epsilon_\ell\in(0,1)$)
\begin{align}
    &\sum_{ \substack{ \mathtt E \subset E_0(S) \cap E_0(K) \\ \mathtt F_{\ell} \subset E_{\ell}(S \Cap K) } } \big(\tfrac{\mu}{n}\big)^{|\mathtt E|} \prod_{1 \leq \ell \leq L} \epsilon_\ell^{|\mathtt F_\ell|} \mathbb E\Bigg\{ \prod_{ (a_i,b_k) \in E_0(S) \triangle E_0(K) } \bm x(i) \bm u(k) \prod_{ (a_i,b_k) \in \mathtt E } \bm x(i)^2 \bm u(k)^2 \nonumber \\
    &\quad\quad\quad\quad\quad\quad\quad \quad\quad\quad\quad\quad\quad\quad \quad\quad\quad\ \prod_{1 \leq \ell \leq L} \prod_{ (a_i,a_j) \in \mathtt F_{\ell} \cup (E_{\ell}(S) \triangle E_{\ell}(K)) } \bm x_{\ell}(i) \bm x_{\ell}(j) \Bigg\} \nonumber \\
    \leq\ & \sum_{ \substack{ \mathtt E \subset E_0(S) \cap E_0(K) \\ \mathtt F_{\ell} \subset E_{\ell}(S \Cap K) } } \big(\tfrac{\mu}{n}\big)^{|\mathtt E|} \mathbb E\Bigg\{ \prod_{b_k \in V^{\mathsf b}(S) \cup V^{\mathsf b}(K) } \bm u(k)^{ \Delta(k) } \prod_{ a_i \in V^{\mathsf a}(S) \cup V^{\mathsf a}(K) } \bm x(i)^{ \Lambda_0(i) } \bm x_1(i)^{\Lambda_1(i)} \ldots \bm x_L(i)^{\Lambda_L(i)} \Bigg\} \,.  \label{eq-cov-f-S-f-K-relax-1}
\end{align}
Here (below we write $v \sim e$ if a vertex $v$ is incident to an edge $e$)
\begin{align}
    &\Delta(k)= \#\big\{ e \in E_0(S) \triangle E_0(K): b_k \sim e \big\} + 2 \#\big\{ e \in \mathtt E: b_k \sim e \big\} \,; \label{eq-def-Delta(k)} \\
    &\Lambda_0(i)= \#\big\{ e \in E_0(S) \triangle E_0(K): a_i \sim e \big\} + 2 \#\big\{ e \in \mathtt E: a_i \sim e \big\} \,; \label{eq-def-Lambda(i)} \\
    &\Lambda_\ell(i)= \#\big\{ e \in E_{\ell}(S) \triangle E_{\ell}(K): a_i \sim e \big\} + \#\big\{ e \in \mathtt F_{\ell}: a_i \sim e \big\} \,. \label{eq-def-Delta-ell(i)}
\end{align}
We now analyze \eqref{eq-cov-f-S-f-K-relax-1} as follows: it is clear that
\begin{itemize}
    \item If $a_i \in \mathsf{dif}_0(S) \setminus V(K)$, (since $S$ is a decorated cycle or decorated path) there exists $a_j \in \{ a_1,\ldots,a_n \}$ and $b_k \in \{ b_1,\ldots,b_p \}$ such that $(a_i,a_j) \in E_{\ell}(S) \setminus E(K)$ and $(a_i,b_j) \in E_0(S) \setminus E(K)$. In addition, it is clear that for all $e \in E(S \Cap K)$ we have $a_i \not\sim e$. Thus, we have 
    \begin{align*}
        \Lambda_0(i)=\Lambda_\ell(i)=1, \Lambda_{\ell'}(i)=0 \mbox{ for all } \ell' \neq \ell \,.
    \end{align*}
    This yields that (recall \eqref{eq-def-x-ell})
    \begin{align*}
        \mathbb E\Big[ \bm x(i)^{ \Lambda_0(i) } \bm x_1^{\Lambda_1(i)} \ldots \bm x_L(i)^{\Lambda_L(i)} \Big]= \rho \,.
    \end{align*}
    Similarly, if $a_i \in \mathsf{dif}_0(K) \setminus V(S)$ we have
    \begin{align*}
        \mathbb E\Big[ \bm x(i)^{ \Lambda_0(i) } \bm x_1^{\Lambda_1(i)} \ldots \bm x_L(i)^{\Lambda_L(i)} \Big]= \rho \,.
    \end{align*}
    \item If $a_i \in \mathsf{dif}(S) \setminus V(K)$, (since $S$ is a decorated cycle or decorated path) there exists $a_j, a_k \in \{ a_1,\ldots,a_n \}$ such that $(a_i,a_j) \in E_{\ell}(S) \setminus E(K)$ and $(a_i,a_k) \in E_{\ell'}(S) \setminus E(K)$. In addition, it is clear that for all $e \in E(S \Cap K)$ we have $a_i \not\sim e$. Thus, we have
    \begin{align*}
        \Lambda_{\ell}(i)=\Lambda_{\ell'}(i)=1, \Lambda_0(i)=\Lambda_{\ell''}(i)=0 \mbox{ for all } \ell'' \neq \ell,\ell' \,.
    \end{align*}
    This yields that (recall \eqref{eq-def-x-ell})
    \begin{align*}
        \mathbb E\Big[ \bm x(i)^{ \Lambda_0(i) } \bm x_1^{\Lambda_1(i)} \ldots \bm x_L(i)^{\Lambda_L(i)} \Big]= \rho^2 \,.
    \end{align*}
    Similarly, if $a_i \in \mathsf{dif}(K) \setminus V(S)$ we have
    \begin{align*}
        \mathbb E\Big[ \bm x(i)^{ \Lambda_0(i) } \bm x_1(i)^{\Lambda_1(i)} \ldots \bm x_L(i)^{\Lambda_L(i)} \Big]= \rho^2 \,.
    \end{align*}
    \item For all other $a_i$, since $\bm x(i),\bm x_\ell(i) \in \{-1,+1 \}$, we have (recall \eqref{eq-def-x-ell})
    \begin{align*}
        0 \leq \mathbb E\Big[ \bm x(i)^{ \Lambda_0(i) } \bm x_1^{\Lambda_1(i)} \ldots \bm x_L(i)^{\Lambda_L(i)} \Big] \leq 1 \,.
    \end{align*}
    \item If $b_k \in V(S) \setminus V(K)$, then (since $S$ is a decorated cycle or decorated path) there exists $a_j, a_k \in \{ a_1,\ldots,a_n \}$ such that $(a_i,b_k),(a_j,b_k) \in E_{0}(S) \setminus E(K)$. Thus, we have $\Delta(k)=2$. This yields that 
    \begin{align*}
        \mathbb E\Big[ \bm u(k)^{ \Delta(k) } \Big] = 1 \,.
    \end{align*}
    Similarly, if $b_k \in V(K) \setminus V(S)$ we have
    \begin{align*}
        \mathbb E\Big[ \bm u(k)^{ \Delta(k) } \Big] = 1 \,.
    \end{align*}
    \item If $b_k \in V(S) \cap V(K)$ and $b_k \not\in \mathsf{L}(S \cap K) \cup V(\mathtt E)$, then there exists $a_j, a_k \in \{ a_1,\ldots,a_n \}$ such that $(a_i,b_k),(a_j,b_k) \in E_{0}(S) \cap E_0(K)$. Then since $(a_i,b_k),(a_j,b_k) \not\in \mathtt E$ it is clear that $\Delta(k)=0$. This yields that
    \begin{align*}
        \mathbb E\Big[ \bm u(k)^{ \Delta(k) } \Big] = 1 \,.
    \end{align*}
    \item If $b_k \in V(S) \cap V(K)$ and $b_k \in \mathsf{L}(S \cap K) \cup V(\mathtt E)$, since the degree of $b_k$ in $S$ and $K$ is bounded by $2$, we have $\Delta(k) \leq 8$. 
    This yields that
    \begin{align*}
        0 \leq \mathbb E\Big[ \bm u(k)^{ \Delta(k) } \Big] \leq C' \,.
    \end{align*}
\end{itemize}
Combining the above arguments, we see that for all $\mathtt E \subset E_0(S) \cap E_0(K)$ and $\mathtt F_{\ell} \subset E_{\ell}(S \Cap K)$, we have (note that $|V(\mathtt E)| \leq 2|\mathtt E|$)
\begin{align}
    &\mathbb E\Bigg\{ \prod_{b_k \in V^{\mathsf b}(S) \cup V^{\mathsf b}(K) } \bm u(k)^{ \Delta(k) } \prod_{ a_i \in V^{\mathsf a}(S) \cup V^{\mathsf a}(K) } \bm x(i)^{ \Lambda_0(i) } \bm x_1^{\Lambda_1(i)} \ldots \bm x_L(i)^{\Lambda_L(i)} \Bigg\} \nonumber  \\
    \leq\ & \rho^{ |\mathsf{dif}_0(S) \setminus V(K)|+|\mathsf{dif}_0(K) \setminus V(S)|+2|\mathsf{dif}(S) \setminus V(K)|+2 |\mathsf{dif}(K) \setminus V(S)| } \cdot (C')^{ |\mathsf{L}(S \cap K)|+2|\mathtt E| } \,.  \label{eq-cov-f-S-f-K-relax-2-part-1}
\end{align}
In addition, denote $\mathtt F=\mathtt F_1 \cup \ldots \cup \mathtt F_{\ell}$. Note that $\mathtt F \subset E(S \Cap K)$ and $(\mathtt F_1,\ldots,\mathtt F_{\ell})$ is fixed when we fix $\mathtt F$. We now argue that if $\mathsf L(\mathtt F) \not\subset \mathsf L(S \Cap K) \cup V(\mathtt E)$, then
\begin{align}
    \mathbb E\Bigg\{ \prod_{b_k \in V^{\mathsf b}(S) \cup V^{\mathsf b}(K) } \bm u(k)^{ \Delta(k) } \prod_{ a_i \in V^{\mathsf a}(S) \cup V^{\mathsf a}(K) } \bm x(i)^{ \Lambda_0(i) } \bm x_1^{\Lambda_1(i)} \ldots \bm x_L(i)^{\Lambda_L(i)} \Bigg\} =0 \,. \label{eq-cov-f-S-f-K-relax-2-part-2}
\end{align}
To this end, note that if $\mathsf L(\mathtt F) \not\subset \mathsf L(S \Cap K) \cup V(\mathtt E)$, then there exists a vertex $a_i \in \mathsf L(\mathtt F) \setminus (\mathsf L(S \Cap K)\cup V(\mathtt E))$. Thus, one of the two cases must occur:

\noindent{\bf Case 1:} There exists $a_j,a_k \in \{ a_1,\ldots,a_n \}$ such that $(a_i,a_j) \in \mathtt F \subset E_{\ell}(S \Cap K)$ and $(a_i,a_k) \in E_{\ell'}(S \Cap K) \setminus \mathtt F$. In this case, we have 
\begin{align*}
    \Lambda_{\ell}(i)=1, \ \Lambda_0(i)=\Lambda_{\ell'}(i)=0 \mbox{ for all } \ell' \neq \ell \,.
\end{align*}
This yields that 
\begin{align*}
    \mathbb E\Big[ \bm x(i)^{ \Lambda_0(i) } \bm x_1(i)^{\Lambda_1(i)} \ldots \bm x_L(i)^{\Lambda_L(i)} \Big]= 0 \,.
\end{align*}

\noindent{\bf Case 2:} There exists $a_j \in \{ a_1,\ldots,a_n \}$ and $b_k\in \{ b_1,\ldots,b_p \}$ such that $(a_i,a_j) \in \mathtt F \subset E_{\ell}(S \Cap K)$ and $(a_i,b_k)\in E_0(S\Cap K)\setminus \mathtt F$. In this case, we have
\begin{align*}
    \Lambda_0(i) \in \{ 0,2 \}, \ \Lambda_\ell(i)=1, \Lambda_{\ell'}(i)=0 \mbox{ for all } \ell' \neq \ell \,.
\end{align*}
This yields that 
\begin{align*}
    \mathbb E\Big[ \bm x(i)^{ \Lambda_0(i) } \bm x_1(i)^{\Lambda_1(i)} \ldots \bm x_L(i)^{\Lambda_L(i)} \Big]= 0 \,.
\end{align*}
Combining the above two cases yields \eqref{eq-cov-f-S-f-K-relax-2-part-2}. Plugging \eqref{eq-cov-f-S-f-K-relax-2-part-1} and \eqref{eq-cov-f-S-f-K-relax-2-part-2} into \eqref{eq-cov-f-S-f-K-relax-1}, we get that
\begin{align}
    \eqref{eq-cov-f-S-f-K-relax-1} \leq\ & \sum_{ \mathtt E \subset E_0(S) \cap E_0(K) }  \#\big\{ \mathtt F \subset E^{\mathtt a}(S \Cap K): \mathsf L(\mathtt F)\subset \mathsf{L}(S \Cap K) \cup V(\mathtt E) \big\} \big(\tfrac{\mu}{n}\big)^{|\mathtt E|} \cdot (C')^{ |\mathsf{L}(S \cap K)|+2|\mathtt E| } \nonumber \\
    & \cdot \rho^{ |\mathsf{dif}_0(S) \setminus V(K)|+|\mathsf{dif}_0(K) \setminus V(S)|+2|\mathsf{dif}(S) \setminus V(K)|+2 |\mathsf{dif}(K) \setminus V(S)| } \,. \label{eq-cov-f-S-f-K-relax-2}   
\end{align}
Note that since $S,K$ are paths or cycles, we have $S \Cap K$ is either a cycle or a collection of paths (we allow a path be a single vertex). Thus, $\mathsf L(S \Cap K)$ and $V(\mathtt E)$ will cut $S \Cap K$ into at most $|\mathsf{L}(S \Cap K)|+|V(\mathtt E)|+1$ non-intersecting paths, and thus we have
\begin{align*}
    \#\big\{ \mathtt F \subset E^{\mathtt a}(S \Cap K): \mathsf L(\mathtt F)\subset \mathsf{L}(S \Cap K) \cup V(\mathtt E) \big\} \leq 2^{|\mathsf{L}(S \Cap K)|+|V(\mathtt E)|+1} \leq 2^{|\mathsf{L}(S \Cap K)|+2|\mathtt E|+1} \,.
\end{align*}
Thus, we have
\begin{align}
    \eqref{eq-cov-f-S-f-K-relax-2} \leq\ & \sum_{ \mathtt E \subset E_0(S) \cap E_0(K) } \big(\tfrac{\mu}{n}\big)^{|\mathtt E|} \cdot (C')^{ |\mathsf{L}(S \cap K)|+2|\mathtt E| } \cdot 2^{|\mathsf{L}(S \Cap K)|+2|\mathtt E|+1} \nonumber \\
    & \cdot \rho^{ |\mathsf{dif}_0(S) \setminus V(K)|+|\mathsf{dif}_0(K) \setminus V(S)|+2|\mathsf{dif}(S) \setminus V(K)|+2 |\mathsf{dif}(K) \setminus V(S)| } \nonumber \\
    \leq\ & (C')^{ |\mathsf{L}(S \cap K)| } 2^{ |\mathsf{L}(S \Cap K)|+1 } \rho^{ |\mathsf{dif}_0(S) \setminus V(K)|+|\mathsf{dif}_0(K) \setminus V(S)|+2|\mathsf{dif}(S) \setminus V(K)|+2 |\mathsf{dif}(K) \setminus V(S)| } \,.   \label{eq-cov-f-S-f-K-relax-3}
\end{align}
It is clear that \eqref{eq-cov-f-S-f-K-relax-3} immediately implies the bound in \eqref{eq-lem-4.3-goal-RHS} by picking $C = 2C'$.

\subsection{Proof of Lemma~\ref{lem-detection-most-technical}}{\label{subsec:proof-lem-4.4}}

We first prove \eqref{eq-bound-var-Pb-f-H-relax-3-Part-1}. Note by \eqref{eq-def-mathtt-P} that $\mathtt P(S,S)=1$. Thus, 
\begin{align*}
    \eqref{eq-var-Pb-f-H-relax-3-Part-1} &\overset{\eqref{eq-def-Omega(S)}}{=} \sum_{ S \in \mathsf K_{n,p}: [S] \in \mathcal H } \frac{C \Xi(S)^2 }{ n^{\aleph} p^{\frac{1}{2}|E_0(S)|} \beta_{\mathcal H}^2 } = \sum_{ [H] \in \mathcal H } \frac{C \Xi(H)^2 }{ n^{\aleph} p^{\frac{1}{2}|E_0(H)|} \beta_{\mathcal H}^2 } \cdot \#\{ S \in \mathsf{K}_{n,p}: S \cong H \} \\
    &= [1+o(1)] \cdot \sum_{ [H] \in \mathcal H } \frac{C \Xi(H)^2 }{ n^{\aleph} p^{\frac{1}{2}|E_0(H)|} \beta_{\mathcal H}^2 } \cdot \frac{ n^{\aleph} p^{\frac{1}{2}|E_0(H)|} }{ |\mathsf{Aut}(H)| } \overset{\eqref{eq-def-beta-mathcal-H}}{=} \frac{C(1+o(1))}{\beta_{\mathcal H}} \overset{\text{Lemma~\ref{lem-bound-beta-mathcal-H}}}{=} o(1) \,,
\end{align*}
leading to \eqref{eq-bound-var-Pb-f-H-relax-3-Part-1}. Now we prove \eqref{eq-bound-var-Pb-f-H-relax-3-Part-2}. Recall $\Xi(S)$, $\Omega(K)$ and $\mathtt P(S,K)$ defined in~\eqref{eq-def-Xi(H)}, ~\eqref{eq-def-Omega(S)} and ~\eqref{eq-def-mathtt-P}, respectively. Using \eqref{eq-condition-strong-detection} , we see that 
\begin{align*}
    C^{|\mathsf L(S \cap K)|+|\mathsf L(S \Cap K)|}=n^{o(1)}, \quad \mathtt P(S,K)=n^{o(1)} \cdot\Xi(S)\Xi(K), \quad \Omega(K) \overset{\eqref{eq-basic-property-mathcal-H}}{=} n^{|E(K)|+o(1)} \,.
\end{align*}
Thus, we have (note that $|E(S)|=|V(S)|$ and $|E(K)|=|V(K)|$)
\begin{align}
    \eqref{eq-var-Pb-f-H-relax-3-Part-2} = \sum_{ \substack{ S,K \subset \mathsf K_{n,p}: [S],[K] \in \mathcal H \\ V(S) \cap V(K) \neq \emptyset, S \neq K } } \frac{ \Xi(S)^2 \Xi(K)^2 }{ n^{|V(K)|-|E(S \Cap K)|-o(1)} \Omega(S) \beta_{\mathcal H}^2 } \,.  \label{eq-var-Pb-f-H-relax-3-Part-2-simplify-1}
\end{align}
Note that for $S \neq K$ and $V(S) \cap V(K) \neq \emptyset$, we must have
\begin{align*}
    |E(S \Cap K)| \leq |E(S) \cap E(K)| \leq |V(S) \cap V(K)|-1 \,.
\end{align*}
Thus, we have
\begin{align}
    \eqref{eq-var-Pb-f-H-relax-3-Part-2-simplify-1} = \sum_{ \substack{ S \in \mathsf{K}_{n,p}: [S] \in \mathcal H \\ 0 \leq r \leq \aleph \\ r+1 \leq t \leq \aleph } } \sum_{ [H] \in \mathcal H } \frac{ \Xi(S)^2 \Xi(H)^2 }{ \Omega(S) n^{|V(H)|-r-o(1)} \beta_{\mathcal H}^2 } \cdot \mathsf{ENUM}(S,[H];t) \,,  \label{eq-var-Pb-f-H-relax-3-Part-2-simplify-2}
\end{align}
where 
\begin{equation}{\label{eq-def-ENUM-detection}}
    \mathsf{ENUM}(S,[H];t) = \#\big\{ K \in \mathsf{K}_n: K \cong H, |V(K) \cap V(S)|=t \big\} \,.
\end{equation}
We now bound $\mathsf{ENUM}(S,[H];t)$ as follows. Note that we have at most $\binom{|V(S)|}{t}$ ways to choose $V(K) \cap V(S)$ and at most $\binom{n+p-|V(S)|}{|V(H)|-t}$ ways to choose $V(K) \setminus V(S)$. In addition, given $V(K)$ we have at most $\frac{|V(H)|!}{|\mathsf{Aut}(H)|}$ ways to choose $K$. Thus, (note that $|V(H)| \leq 2\aleph$)
\begin{align*}
    \mathsf{ENUM}(S,[H];t) \leq \binom{|V(S)|}{t} \binom{n+p-|V(S)|}{|V(H)|-t} \cdot \frac{|V(H)|!}{|\mathsf{Aut}(H)|} \leq \frac{ n^{|V(H)|-t+o(1)} }{ |\mathsf{Aut}(H)| } \,.
\end{align*}
Plugging this bound into \eqref{eq-var-Pb-f-H-relax-3-Part-2-simplify-2}, we get that
\begin{align}
    \eqref{eq-var-Pb-f-H-relax-3-Part-2-simplify-2} &\leq \sum_{ \substack{ S \in \mathsf{K}_{n,p}: [S] \in \mathcal H \\ 0 \leq r \leq \aleph \\ r+1 \leq t \leq \aleph } } \sum_{ [H] \in \mathcal H } \frac{ \Xi(S)^2 \Xi(H)^2 }{ \Omega(S) n^{|V(H)|-r-o(1)} \beta_{\mathcal H}^2 } \cdot \frac{ n^{|V(H)|-t+o(1)} }{ |\mathsf{Aut}(H)| } \nonumber \\
    &\overset{\eqref{eq-def-beta-mathcal-H}}{=} \sum_{ \substack{ S \in \mathsf{K}_{n,p}: [S] \in \mathcal H \\ 0 \leq r \leq \aleph \\ r+1 \leq t \leq \aleph } } \frac{ \Xi(S)^2 }{ \Omega(S) n^{t-r-o(1)} \beta_{\mathcal H} } = \sum_{ S \in \mathsf{K}_{n,p}: [S] \in \mathcal H } \frac{ \Xi(S)^2 }{ \Omega(S) n^{1-o(1)} \beta_{\mathcal H} } \nonumber \\
    &= \frac{ 1 }{ n^{1-o(1)} \beta_{\mathcal H} } \sum_{ [H] \in \mathcal H } \frac{ \Xi(H)^2 }{ \Omega(H) } \cdot \#\{ S \subset \mathsf K_{n,p}: S \cong H \} \nonumber \\
    &\overset{\eqref{eq-def-Omega(S)}}{=} \frac{ 1 }{ n^{1-o(1)} \beta_{\mathcal H} } \sum_{ [H] \in \mathcal H } \frac{ \Xi(H)^2 }{ \Omega(H) } \cdot \frac{ \Omega(H) }{ |\mathsf{Aut}(H)| } \overset{\eqref{eq-def-beta-mathcal-H}}{=} n^{-1+o(1)} \,,  \nonumber
\end{align}
leading to \eqref{eq-bound-var-Pb-f-H-relax-3-Part-2}.

\subsection{Proof of Lemma~\ref{lem-bound-Xi-mathtt-P}}{\label{subsec:proof-lem-4.8}}

Recall \eqref{eq-def-Xi(H)} that $$\Xi(H)= \rho^{ |\mathsf{dif}_0(H)|+2|\mathsf{dif}(H)| } \big( \tfrac{\mu^2}{\gamma} \big)^{ \frac{1}{4}|E_0(H)| } \prod_{1 \leq \ell \leq L} \big( \epsilon^2_{\ell} \lambda_{\ell} \big)^{ \frac{1}{2}|E_{\ell}(H)| }. $$ For all $0 \leq \ell \leq L$, it is straightforward to check that 
\begin{align}
    |E_{\ell}(S)| = \sum_{ 1 \leq \mathtt t \leq \mathtt T+1 } |E_{\ell}(S_{\mathtt t})| + \sum_{ 1 \leq \mathtt t \leq \mathtt T } |E_{\ell}(\widetilde{S}_{\mathtt t})| \,.  \label{eq-equality-E_ell(S)} 
\end{align}
In addition, for all $v \in V(S)$ we have
\begin{align*}
    \sum_{ 1 \leq \mathtt t \leq \mathtt T+1 } \mathbf 1(v \in \mathsf{dif}_0(S_{\mathtt t})) + \sum_{ 1 \leq \mathtt t \leq \mathtt T } \mathbf 1(v \in \mathsf{dif}_0(\widetilde S_{\mathtt t})) \leq \mathbf 1(v \in \mathsf{dif}_0(S)) + 1 \,.
\end{align*}
Thus, we have
\begin{align}
    |\mathsf{dif}_0(S)| &= \sum_{ v \in V(S) } \mathbf 1(v \in \mathsf{dif}_0(S)) \nonumber \\
    &\geq \sum_{ v \in V(S) } \Big( \sum_{ 1 \leq \mathtt t \leq \mathtt T+1 } \mathbf 1(v \in \mathsf{dif}_0(S_{\mathtt t})) + \sum_{ 1 \leq \mathtt t \leq \mathtt T } \mathbf 1(v \in \mathsf{dif}_0(\widetilde S_{\mathtt t}) ) \Big) - 2 \mathtt T \nonumber \\
    &= \sum_{ 1 \leq \mathtt t \leq \mathtt T+1 } |\mathsf{dif}_0(S_{\mathtt t}))| + \sum_{ 1 \leq \mathtt t \leq \mathtt T } |\mathsf{dif}_0(S_{\mathtt t})| - 2 \mathtt T \,. \label{eq-inequality-dif_0(S)}
\end{align}
Similarly we have
\begin{align}
    |\mathsf{dif}(S)| \geq \sum_{ 1 \leq \mathtt t \leq \mathtt T+1 } |\mathsf{dif}(S_{\mathtt t}))| + \sum_{ 1 \leq \mathtt t \leq \mathtt T } |\mathsf{dif}(S_{\mathtt t})| - 2 \mathtt T \,. \label{eq-inequality-dif(S)}
\end{align}
Plugging \eqref{eq-equality-E_ell(S)}--\eqref{eq-inequality-dif(S)} into \eqref{eq-def-Xi(H)} yields \eqref{eq-bound-Xi(S)}. Similarly we can show \eqref{eq-bound-Xi(K)}. Now we turn to \eqref{eq-bound-mathtt-P(S,K)}. Note that as $E(S) \cap E(K)=E(S \Cap K)$, for all $0 \leq \ell \leq L$ we have
\begin{align}
    |E_{\ell}(S) \triangle E_{\ell}(K)| \geq \sum_{1 \leq \mathtt t \leq \mathtt T} |E_{\ell}(\widetilde S_{\mathtt t} )| + \sum_{1 \leq \mathtt t \leq \mathtt T} |E_{\ell}(\widetilde K_{\mathtt t} )| \,.  \label{eq-equality-E_ell(S)-triangle-E_ell(K)} 
\end{align}
In addition, note that $\mathsf{dif}_0(\widetilde S_{\mathtt t}) \setminus \mathsf L(\widetilde S_{\mathtt t}) \subset \mathsf{dif}_0(S) \setminus V(K)$ are disjoint (and similarly for $\mathsf{dif}_0(K) \setminus V(S)$), we have
\begin{align}
    & |\mathsf{dif}_0(S) \setminus V(K)| \geq \sum_{1 \leq \mathtt t \leq \mathtt T} |\mathsf{dif}_0(\widetilde S_{\mathtt t})|-2\mathtt T \,; \label{eq-inequality-dif_0(S)-setminus-V(K)} \\
    & |\mathsf{dif}_0(K) \setminus V(S)| \geq \sum_{1 \leq \mathtt t \leq \mathtt T} |\mathsf{dif}_0(\widetilde K_{\mathtt t})|-2\mathtt T \,. \label{eq-inequality-dif_0(K)-setminus-V(S)}
\end{align}
Similarly we have
\begin{align}
    & |\mathsf{dif}(S) \setminus V(K)| \geq \sum_{1 \leq \mathtt t \leq \mathtt T} |\mathsf{dif}(\widetilde S_{\mathtt t})|-2\mathtt T \,; \label{eq-inequality-dif(S)-setminus-V(K)} \\
    & |\mathsf{dif}(K) \setminus V(S)| \geq \sum_{1 \leq \mathtt t \leq \mathtt T} |\mathsf{dif}(\widetilde K_{\mathtt t})|-2\mathtt T \,. \label{eq-inequality-dif(K)-setminus-V(S)}
\end{align}
Plugging \eqref{eq-equality-E_ell(S)-triangle-E_ell(K)}--\eqref{eq-inequality-dif(K)-setminus-V(S)} into \eqref{eq-def-mathtt-P} yields \eqref{eq-bound-mathtt-P(S,K)}.

\subsection{Proof of Lemma~\ref{lem-recovery-most-technical}}{\label{subsec:proof-lem-4.9}}

Before proving Lemma~\ref{lem-recovery-most-technical}, we first show the following lemma that will be used repeatedly in the later proof. Recall that $\mathcal J_\star(\aleph)$ denotes the collection of unlabeled decorated paths $[H]$ such that $|V^{\mathsf a}(H)| = \aleph+1$.
\begin{lemma}{\label{lem-contribution-sub-chain}}
    For all $\mathtt a \geq 0$ and given $u,v$, we have 
    \begin{align}{\label{eq-contribution-sub-chain}}
        \sum_{ S:[S] \in \mathcal J_\star(\mathtt a), \mathsf{L}(S)=\{ u,v \} } \frac{ n\Xi(S)^2 }{ \Omega(S) } \leq O(1) \cdot \sigma_+(\mathbf P)^{\mathtt a} \cdot n^{\mathbf 1(u=v)} \,.
    \end{align}
\end{lemma}
\begin{proof}
    When $u=v$, it is clear that left-hand side only sums over $S=\{ u \}$ being a single vertex and thus equals $n$, so \eqref{eq-contribution-sub-chain} holds in this case ($\mathtt a = 0$ and $n^{\mathbf 1(u=v)} = n$). For the case $u \neq v$, it is clear that the left-hand side of \eqref{eq-contribution-sub-chain} equals 
    \begin{align}
        &\sum_{[H] \in \mathcal J_{\star}(\mathtt a)} \frac{ n \Xi(H)^2 }{ \Omega(H) } \cdot \#\big\{ S \subset \mathsf K_{n,p}: S \cong H, \mathsf L(S)=\{ u,v \} \big\} \nonumber \\
        =\ &[1+o(1)]\sum_{[H] \in \mathcal J_{\star}(\mathtt a)} \frac{ n \Xi(H)^2 }{ \Omega(H) } \cdot \frac{ n^{|V^{\mathsf a}(H) \setminus \mathsf L(H)|} p^{|V^{\mathsf b}(H) \setminus \mathsf L(H)|} }{ |\mathsf{Aut}(H)| }  \,.  \label{eq-contribution-sub-chain-relax-1}
    \end{align}
    Note that we have
    \begin{align*}
        &|V^{\mathsf a}(H)\setminus \mathsf L(H)|-|E(H)|+\tfrac{1}{2}|E_0(H)| =O(1) \,, \\
        &|V^{\mathsf b}(H)\setminus \mathsf L(H)|-\tfrac{1}{2}|E_0(H)|=O(1) \,, \\
        &|V^{\mathsf a}(H)\setminus \mathsf L(H)|+|V^{\mathsf b}(H)\setminus \mathsf L(H)|=|E(H)|-1 \,.
    \end{align*}
    Consequently,
    \begin{align*}
        n\cdot n^{|V^{\mathsf a}(H) \setminus \mathsf L(H)|} p^{|V^{\mathsf b}(H) \setminus \mathsf L(H)|}  = n\cdot n^{|V^{\mathsf a}(H) \setminus \mathsf L(H)|} (n/\gamma)^{|V^{\mathsf b}(H) \setminus \mathsf L(H)|} = O(1)\cdot n^{|E(H)|} \,. 
    \end{align*}
    By~\eqref{eq-def-Omega(S)}, we have 
    \begin{align*}
        \Omega(H) &= n^{|E(H)| - \tfrac{1}{2}|E_0(H)|}p^{\tfrac{1}{2}|E_0(H)|} = n^{|E(H)| - \tfrac{1}{2}|E_0(H)|}(n/\gamma)^{\tfrac{1}{2}|E_0(H)|}=O(1)\cdot n^{|E(H)|} \,.
    \end{align*}
    Thus, we have 
    \begin{align*}
        \frac{ n \cdot n^{|V^{\mathsf a}(H) \setminus \mathsf L(H)|} p^{|V^{\mathsf b}(H) \setminus \mathsf L(H)|} }{ \Omega(H) } = O(1) 
    \end{align*}
    and thus
    \begin{align*}
        \eqref{eq-contribution-sub-chain-relax-1} = O(1) \cdot \sum_{[H] \in \mathcal J_{\star}(\mathtt a)} \frac{ \Xi(H)^2 }{|\mathsf{Aut}(H)|  } \overset{\eqref{eq-def-beta-mathcal-J}}{=} O(1) \cdot \beta_{\mathcal J_\star(\mathtt a)} \overset{\text{Lemma~\ref{lem-bound-beta-mathcal-J}}}{=} O(1) \cdot \sigma_+(\mathbf P)^{\mathtt a} \,,
    \end{align*}
    so \eqref{eq-contribution-sub-chain} also holds in this case.
\end{proof}

Now we return to the proof of Lemma~\ref{lem-recovery-most-technical}. We will split \eqref{eq-var-Phi-i,j-relax-4} into several parts and bound each part separately. Note that $C $ and $\rho$ are constants, we write $C^{4\mathtt T+5} \rho^{-40 \mathtt T}$ as $O(1)^{\mathtt T}$ in the following. It suffices to show 
\begin{align*}
    \sum_{\mathtt T \geq 0} \frac{O(1)^{\mathtt T} n^{-3\mathtt T+1}}{\beta_{\mathcal J}^2} \sum_{ S_{\cap}, S_{\setminus}, K_{\setminus} \in \mathcal A_{\mathtt T} } \prod_{1 \leq \mathtt t \leq \mathtt T+1} \frac{ n\Xi(S_{\mathtt t})^2 }{ \Omega(S_{\mathtt t}) } \prod_{1 \leq \mathtt t \leq \mathtt T} \frac{ n\Xi(\widetilde S_{\mathtt t})^2 }{ \Omega(\widetilde S_{\mathtt t}) }  \prod_{1 \leq \mathtt t \leq \mathtt T} \frac{ n\Xi(K_{\mathtt t})^2 }{ \Omega(\widetilde K_{\mathtt t}) }\le O(1) \,.
\end{align*}
We will consider the following three cases: $\mathtt T = 0$, $\mathtt T = 1$, and $\mathtt T\ge 2$, respectively. Note that for the case $\mathtt T = 0$ and $\mathtt T=1$, we can simply write $O(1)^{\mathtt T}$ as $O(1)$.

\noindent {\bf Part 1: the case $\mathtt T=0$.} In this case, we have $V(S)=V(K)$ and thus
\begin{align*}
    S_{\setminus}=K_{\setminus}=\emptyset, \quad S_{\cap}=S_1=S \,.
\end{align*}
Thus, the contribution of this part in \eqref{eq-var-Phi-i,j-relax-4} is bounded by
\begin{align}
    & O(1) \cdot \frac{n}{\beta_{\mathcal J}^2} \sum_{ [S] \in \mathcal J_{\star}(\aleph): \mathsf{L}(S)=\{ a_u,a_v \} } \frac{ \Xi(S)^2 }{ \Omega(S) } \nonumber \\
    \leq\ & O(1) \cdot \frac{ n\sigma_+(\mathbf P)^\aleph }{ \beta_{\mathcal J}^2 } \overset{\text{Lemma~\ref{lem-bound-beta-mathcal-J}}}{\leq} O(1) \cdot n \sigma_+(\mathbf P)^{-\aleph} \overset{\eqref{eq-condition-weak-recovery}}{=} o(1) \,,  \label{eq-bound-part-1-contribution-var-Phi-i,j}
\end{align}
where the first inequality follows from taking $\mathtt a=\aleph$ in Lemma~\ref{lem-contribution-sub-chain}.

\noindent {\bf Part 2: the case $\mathtt T=1$.} In this case, we have
\begin{align*}
    S_{\cap} = S_{1} \cup S_{2}, \quad \widetilde S_{\setminus}=\widetilde{S}_{1}, \quad K_{\setminus}=\widetilde K_{1} \,. 
\end{align*}
Denote $\mathfrak V(S)=\mathbf 1_{ \{ S \text{ is single vertex} \} }$. We have at most $4$ choices of $\mathfrak V(S_1),\mathfrak V(S_2)$. In addition, given $\mathfrak V(S_1),\mathfrak V(S_2)$, we have 
\begin{align*}
    (n+p)^{2-\mathfrak V(S_1)-\mathfrak V(S_2)} = O(1) \cdot n^{ 2-\mathfrak V(S_1)-\mathfrak V(S_2) }
\end{align*}
choices of $\mathsf L(S_1),\mathsf L(S_2)$ (note that $a_u \in \mathsf L(S_1)$, $a_v \in \mathsf L(S_2)$, and there are $n+p$ choices for the other leaves). In addition, assuming that $|V^{\mathsf a}(S_1)|=\mathtt a$ and $|V^{\mathsf a}(S_2)|=\mathtt b$, we have
\begin{align*}
    |V^{\mathsf a}(\widetilde S_1)|,|V^{\mathsf a}(\widetilde K_1)| \in \{ \aleph-\mathtt a-\mathtt b, \aleph-\mathtt a-\mathtt b+1, \aleph-\mathtt a-\mathtt b+2 \} \,.
\end{align*}
Here, the exact value of $|V^{\mathsf a}(\widetilde S_1)|,|V^{\mathsf a}(\widetilde K_1)|$ depending on whether $V(S_1)\cap V(\widetilde{S}_1), V(\widetilde{S}_1)\cap V(S_2)\in V^{\mathsf a}(S)$ and $V(S_1)\cap V(\widetilde{K}_1),V(\widetilde{K}_1)\cap V(S_2)\in V^{\mathsf a}(K)$. Thus, using Lemma~\ref{lem-contribution-sub-chain}, given $\mathsf L(S_1),\mathsf L(S_2)$ (and thus $\mathsf L(\widetilde S_1),\mathsf L(\widetilde K_1)$ is also fixed) and $|V^{\mathsf a}(\widetilde S_1)|,|V^{\mathsf a}(\widetilde K_1)|\le \aleph-\mathtt a-\mathtt b+2$ we have
\begin{align*}
    &\sum_{ S_{\cap},S_{\setminus},K_{\setminus} } \frac{n\Xi(S_1)^2}{\Omega(S_1)} \cdot \frac{n\Xi(S_2)^2}{\Omega(S_2)} \cdot \frac{n\Xi(\widetilde S_1)^2}{\Omega(\widetilde S_1)} \cdot \frac{n\Xi(\widetilde K_1)^2}{\Omega(\widetilde K_1)} \\
    \le \ & O(1) \cdot n^{\mathfrak V(S_1)+\mathfrak V(S_2)} \cdot \sigma_+(\mathbf P)^{\mathtt a} \cdot \sigma_+(\mathbf P)^{\mathtt b} \cdot \sigma_+(\mathbf P)^{\aleph-\mathtt a-\mathtt b+2} \cdot \sigma_+(\mathbf P)^{\aleph-\mathtt a-\mathtt b+2} \\
    =\ & O(1) \cdot n^{\mathfrak V(S_1)+\mathfrak V(S_2)} \cdot \sigma_+(\mathbf P)^{2\aleph-\mathtt a-\mathtt b+4} \,.
\end{align*}
Thus, the contribution of this part in \eqref{eq-var-Phi-i,j-relax-4} is bounded by
\begin{align}
    & O(1) \cdot \frac{ n^{-2} }{ \beta_{\mathcal J}^2 } \sum_{ \mathtt a+\mathtt b\leq \aleph } n^{ 2-\mathfrak V(S_1)-\mathfrak V(S_2) } \cdot n^{\mathfrak V(S_1)+\mathfrak V(S_2)} \cdot \sigma_+(\mathbf P)^{2\aleph-\mathtt a-\mathtt b+4} \nonumber \\
    \overset{\text{Lemma~\ref{lem-bound-beta-mathcal-J}}}{\leq}\ & O(1) \cdot \sum_{ \mathtt a+\mathtt b\leq \aleph } \sigma_+(\mathbf P)^{-\mathtt a-\mathtt b+4} = O(1) \,, \label{eq-bound-part-2-contribution-var-Phi-i,j}
\end{align}
where the last inequality follows from $\sigma_+(\mathbf P)>1$ given by Lemma~\ref{lem-bound-beta-mathcal-J} with \eqref{eq-assum-upper-bound}.

\noindent{\bf Part 3: the case $\mathtt T \geq 2$.} In this case we have
\begin{align*}
    S_{\cap} = \cup_{1 \leq \mathtt t \leq \mathtt T+1} S_{\mathtt t}, \quad S_{\setminus} = \cup_{1 \leq \mathtt t \leq \mathtt T} \widetilde S_{\mathtt t}, \quad K_{\setminus} = \cup_{1 \leq \mathtt t \leq \mathtt T} \widetilde K_{\mathtt t} \,.
\end{align*}
Again, define $\mathfrak V(S)=\mathbf 1_{ \{ S \text{ is single vertex} \} }$, we have at most $2^{\mathtt T+1}$ choices for $\{ \mathfrak V(S_{\mathtt t}): 1 \leq \mathtt t \leq \mathtt T+1 \}$. In addition, given $\{ \mathfrak V(S_{\mathtt t}): 1 \leq \mathtt t \leq \mathtt T+1 \}$ we have
\begin{align*}
    (n+p)^{ 2\mathtt T-\mathfrak V(S_1)-\ldots-\mathfrak V(S_{\mathtt T+1}) } \leq (1+\gamma^{-1})^{2\mathtt T} n^{ 2\mathtt T-\mathfrak V(S_1)-\ldots-\mathfrak V(S_{\mathtt T+1}) }
\end{align*}
choices of $\mathsf L(S_1),\ldots,\mathsf L(S_{\mathtt T+1})$. Also, there are at most $\aleph^{3\mathtt T+1}$ choices of 
\begin{align*}
    \left( p_{\mathtt t}, \widetilde{p}_{\mathtt t},\widetilde{q}_{\mathtt t} \right) =\left( |V^{\mathsf a}(S_{\mathtt t})|, |V^{\mathsf a}(\widetilde S_{\mathtt t})|, |V^{\mathsf a}(\widetilde K_{\mathtt t})| \right) \,.
\end{align*}
And we must have 
\begin{equation}{\label{eq-require-p-t-q-t}}
    \begin{aligned}
        &\sum_{1 \leq \mathtt t \leq \mathtt T+1} p_{\mathtt t} + \sum_{1 \leq \mathtt t \leq \mathtt T} \widetilde{p}_{\mathtt t} \leq \aleph+2\mathtt T \,; \\
        &\sum_{1 \leq \mathtt t \leq \mathtt T+1} p_{\mathtt t} + \sum_{1 \leq \mathtt t \leq \mathtt T} \widetilde{q}_{\mathtt t} \leq \aleph+2\mathtt T \,. 
    \end{aligned}
\end{equation}
Finally, using Lemma~\ref{lem-contribution-sub-chain}, given $\mathsf L(S_1),\ldots,\mathsf L(S_{\mathtt T+1})$ and $( p_{\mathtt t}, \widetilde{p}_{\mathtt t},\widetilde{q}_{\mathtt t} )$ above we have
\begin{align*}
    &\sum_{\substack{S_{\cap},S_{\setminus},K_\setminus\in \mathcal A_{\mathtt T}:\\\text{ given }\mathsf L(S_1),\ldots,\mathsf L(S_{\mathtt T+1})\text{ and }( p_{\mathtt t}, \widetilde{p}_{\mathtt t},\widetilde{q}_{\mathtt t} )}} \prod_{1 \leq \mathtt t \leq \mathtt T+1} \frac{ n\Xi(S_{\mathtt t})^2 }{ \Omega(S_{\mathtt t}) } \prod_{1 \leq \mathtt t \leq \mathtt T} \frac{ n\Xi(\widetilde S_{\mathtt t})^2 }{ \Omega(\widetilde S_{\mathtt t}) }  \prod_{1 \leq \mathtt t \leq \mathtt T} \frac{ n\Xi(K_{\mathtt t})^2 }{ \Omega(\widetilde K_{\mathtt t}) } \\
    \leq\ & O(1)^{\mathtt T} n^{ \mathfrak V(S_1)+\ldots+\mathfrak V(S_{\mathtt T+1}) } \prod_{1 \leq \mathtt t \leq \mathtt T+1} \sigma_+(\mathbf P)^{p_{\mathtt t}} \prod_{1 \leq \mathtt t \leq \mathtt T} \sigma_+(\mathbf P)^{\widetilde p_{\mathtt t}} \prod_{1 \leq \mathtt t \leq \mathtt T} \sigma_+(\mathbf P)^{\widetilde q_{\mathtt t}} \\
    \overset{\eqref{eq-require-p-t-q-t}}{\leq}\ & O(1)^{\mathtt T} \cdot n^{ \mathfrak V(S_1)+\ldots+\mathfrak V(S_{\mathtt T+1}) } \sigma_+(\mathbf P)^{2\aleph} \,.
\end{align*}
It remains to sum over $\mathsf L(S_1),\ldots,\mathsf L(S_{\mathtt T+1})\text{ and }( p_{\mathtt t}, \widetilde{p}_{\mathtt t},\widetilde{q}_{\mathtt t} )$. Recall that there are at most $(1+\gamma^{-1})^{2\mathtt T} n^{ 2\mathtt T-\mathfrak V(S_1)-\ldots-\mathfrak V(S_{\mathtt T+1})}$ possible choices for $\mathsf L(S_1),\ldots,\mathsf L(S_{\mathtt T+1})$, and at most $\aleph^{3\mathtt T+1}$ choices for $( p_{\mathtt t}, \widetilde{p}_{\mathtt t},\widetilde{q}_{\mathtt t} )$.
Thus, the contribution of this part in \eqref{eq-var-Phi-i,j-relax-4} is bounded by 
\begin{align}
    &\sum_{\mathtt T \geq 2} \frac{ n^{-3\mathtt T+1} }{ \beta_{\mathcal J}^2 } O(1)^{\mathtt T+1} \cdot n^{ 2\mathtt T-\mathfrak V(S_1)-\ldots-\mathfrak V(S_{\mathtt T+1}) } \cdot \aleph^{3\mathtt T+1} \cdot n^{ \mathfrak V(S_1)+\ldots+\mathfrak V(S_{\mathtt T+1}) } \sigma_+(\mathbf P)^{2\aleph} \nonumber \\
    \leq\ & \sum_{\mathtt T \geq 2} \frac{ n^{-\mathtt T+1} }{ \beta_{\mathcal J}^2 } \cdot O(1)^{\mathtt T+1} \cdot \aleph^{3\mathtt T+1} \cdot \sigma_+(\mathbf P)^{2\aleph} \nonumber \\
    \overset{\text{Lemma~\ref{lem-bound-beta-mathcal-J}}}{\leq}\ & O(1) \cdot \sum_{\mathtt T \geq 2}  O(1)^{\mathtt T+1} \cdot \aleph^{3\mathtt T+1} \cdot n^{-\mathtt T+1} = n^{-1+o(1)} \,.  \label{eq-bound-part-3-contribution-var-Phi-i,j}
\end{align}
Combining \eqref{eq-bound-part-1-contribution-var-Phi-i,j}, \eqref{eq-bound-part-2-contribution-var-Phi-i,j} and \eqref{eq-bound-part-3-contribution-var-Phi-i,j} leads to Lemma~\ref{lem-recovery-most-technical}.

\bibliographystyle{alpha}

\end{document}